\documentclass[abstract=true]{scrartcl}
\usepackage{amsthm,amsmath,amssymb}
\usepackage{mathtools}
\usepackage{authblk}

\usepackage{hyperref}

\usepackage[utf8]{inputenc}
\usepackage{enumitem}
\usepackage{scalerel}

\usepackage{tikz}
\usetikzlibrary{arrows.meta,hobby,calc}

\usepackage{subcaption}
\usepackage{tikz}
\usepackage{pgfplots}
\pgfplotsset{compat=1.15}
\usepackage{pgfmath}
\usetikzlibrary{arrows.meta}
\usetikzlibrary{decorations.markings}
\usetikzlibrary{shapes.geometric}
\usetikzlibrary{calc}
\usetikzlibrary{fillbetween}

\definecolor{gen0}{rgb}{1,0,0}
\definecolor{gen1}{rgb}{0,0,1}
\definecolor{gen2}{rgb}{0,0.7,0}

\newtheorem{thm}{Theorem}[section]
\newtheorem{pro}[thm]{Proposition}
\newtheorem{lem}[thm]{Lemma}
\newtheorem{cor}[thm]{Corollary}

\theoremstyle{definition}
\newtheorem{cons}[thm]{Construction}
\newtheorem{rmk}[thm]{Remark}

\newcommand{\abs}[1]{\left\lvert#1\right\rvert}
\newcommand{\norm}[1]{\left\lVert#1\right\rVert}
\newcommand{\overbar}[1]{\mkern 2mu\overline{\mkern-2mu#1\mkern-2mu}\mkern 2mu}

\newcommand\Ccal{\mathcal{C}}
\newcommand\Ical{\mathcal{I}}
\newcommand\Ucal{\mathcal{U}}
\newcommand\Scal{\mathcal{S}}
\newcommand\Tcal{\mathcal{T}}
\newcommand\Gcal{\mathcal{G}}
\newcommand\Vcal{\mathcal{V}}
\newcommand\Ecal{\mathcal{E}}

\newcommand\Kcal{\mathcal{K}}
\newcommand\Acal{\mathcal{A}}
\newcommand\Wcal{\mathcal{W}}
\newcommand\Bcal{\mathcal{B}}
\newcommand\Fbf{\mathbf{F}}

\newcommand\ybf{\mathbf{y}}

\newcommand\N{\mathbb{N}}

\newcommand\AUT{\mathrm{AUT}}

\newcommand\SAW{\mathrm{SAW}}

\newcommand\SAR{\mathrm{SAR}}
\newcommand\Jac{\mathfrak{J}}
\newcommand\pair{\mathrm{pair}}
\newcommand\disj{\mathrm{disj}}
\newcommand\sta{\mathrm{star}}
\newcommand\nb{\mathrm{nb}}
\newcommand\refl{\mathrm{Refl}}
\newcommand\init{\mathrm{-}}
\newcommand\term{\mathrm{+}}
\newcommand\Eout{E}
\renewcommand\deg{\mathrm{deg}}
\newcommand{\lidi}{distance}

\title{Self-avoiding walk is ballistic on graphs with more than one end}

\author{Florian~Lehner\thanks{F.\ Lehner was supported by FWF (Austrian Science Fund) project P31889-N35.}}
\affil{
University of Auckland, New Zealand}

\author{Christian~Lindorfer\thanks{C.\ Lindorfer was partially supported by FWF (Austrian Science Fund) projects P31889-N35 and DK~W1230.}}
\affil{
Technische Universit\"at Graz,
Austria}

\author{Christoforos~Panagiotis\thanks{C.\ Panagiotis was supported by the Swiss National Science Foundation and the NCCR SwissMAP.}}
\affil{University of Bath, 
UK}

\date{\today} 

\begin{document}
\maketitle

\begin{abstract}
We prove that on any transitive graph $G$ with infinitely many ends, a self-avoiding walk of length $n$ is ballistic with extremely high probability, in the sense that there exist constants $c,t>0$ such that $\mathbb{P}_n(d_G(w_0,w_n)\geq cn)\geq 1-e^{-tn}$ for every $n\geq 1$. Furthermore, we show that the number of self-avoiding walks of length $n$ grows asymptotically like $\mu_w^n$, in the sense that there exists $C>0$ such that $\mu_w^n\leq c_n\leq C\mu_w^n$ for every $n\geq 1$. Our results extend more generally to quasi-transitive graphs with infinitely many ends, satisfying the additional technical property that there is a quasi-transitive group of automorphisms of $G$ which does not fix an end of $G$.
\end{abstract}

\section{Introduction}

A \textit{self-avoiding walk} on a graph $G$ is a walk that visits each vertex at most once. This notion was originally introduced in the work of the chemist Flory \cite{Flory53} to model long polymer chains, and it soon attracted the interest of the mathematical community. The primary focus has been on studying the asymptotic behaviour of a self-avoiding of a given length sampled uniformly at random, giving rise to questions that, while simple to pose, frequently prove challenging to resolve.

A significant amount of research on self-avoiding walks has concentrated on answering these questions in the case of lattices in $\mathbb{R}^d$, where the model is now well-understood in dimensions $d\geq 5$ by the seminal work of Hara and Slade \cite{HaSla1992,HaSlaCri}. The low-dimensional cases continue to present serious challenges. See \cite{HammersleyWelsh,KestenI,KestenII,DuHaSub,SAWEnds,PolyIns,SAWHex,KP23} and the references therein for some of the most important results. For a comprehensive introduction to the model in this context interested readers can refer to \cite{MR3025395,MR2986656}.

Over the years, the study of self-avoiding walk beyond lattices has received increasing attention. The systematic study of self-avoiding walk on general transitive graphs was initiated in a series of papers by Grimmett and Li \cite{GriLi2013,GriLi2014b,GriLi2014,GriLi2015,GriLi2016,GriLi2017,GriLi2017b,zbMATH07217533}, whose work is primarily concerned with properties of the connective constant.
Other works on self-avoiding walk in this context include \cite{MaWuSAW, NaPe,Extendable,MR3584819,Ben16,Li20,P19,BenPa,GeoWed}.
See \cite{zbMATH07217533} for a survey of these results.

In this paper, we consider the self-avoiding walk on quasi-transitive graphs with more than one end. Such graphs admit a group invariant \emph{tree decomposition} $\Tcal = (T,\Vcal)$; we refer the reader to Section~\ref{sec: tree dec} for the precise definition, but note that any such tree decomposition gives rise a quasi-transitive action of $\AUT(G)$ on $T$. The concept of tree decomposition was initially introduced by Halin in 1976 \cite{MR444522} and later rediscovered by Robertson and Seymour \cite{MR742386}, playing a pivotal role in proving the Graph Minor Theorem. Tree decompositions that are invariant under some group action were perhaps first studied by Dunwoody and Kr{\"o}n \cite{MR3385727}, drawing inspiration from a method involving edge cuts introduced by Dunwoody in \cite{MR671142}. In our context, it is crucial that these tree decompositions are not only invariant under a quasi-transitive group of automorphisms, but also satisfy some additional properties -- see Corollary~\ref{cor:reducedtd} for the precise statement.

One of the first instances in the study of self-avoiding walks on graphs with more than one end can be found in the work of Alm and Janson \cite{AlJa90}, where they specifically study the case of $2$-ended graphs. This particular case turns out to be more tractable due to the fact that the large-scale structure of the graph is similar to that of a line, that is, the tree $T$ of the tree decomposition is isomorphic to $\mathbb Z$. The remaining case of graphs with infinitely many ends proves to be more challenging, with the currently known results limited to either the properties of the SAW-generating functions away from its critical point \cite{saw-mcfl} or to graphs that satisfy additional geometric assumptions \cite{MR3584819,Li20}. Our main results answer two fundamental questions in the study of self-avoiding walks in the case of transitive graphs with infinitely many ends, and more generally, quasi-transitive graphs $G$ that satisfy the additional technical property that a quasi-transitive group of automorphisms $\Gamma$ of $G$ does not fix an end of the tree $T$; this is in particular true if $\AUT(G)$ does not fix an end of $G$.

Denote by $c_n$ the number of self-avoiding walks of length $n$ on $G$
started from some fixed vertex $o$. A fundamental quantity in the study of self-avoiding walk is the \textit{connective constant} $$\mu_w=\lim_{n\to \infty} c_n^{1/n},$$
where the fact that the limit exists and does not depend on the starting point follows from a standard subadditivity argument \cite{MR0091568}. The connective constant is not typically known or expected to take an interesting value, with a notable exception being the hexagonal lattice, where Duminil-Copin and Smirnov proved in a celebrated paper \cite{SAWHex} that $\mu_w=\sqrt{2+\sqrt{2}}$. For this reason, it is often more interesting to estimate the subexponential correction to $\mu_w^n$.

Our first result states that $c_n$ grows asymptotically like $\mu_w^n$. 
\begin{thm}\label{thm:c_n asym}
If a quasi transitive graph $G$ has infinitely many ends, and $\AUT(G)$ does not fix an end of $G$, then there exist $k\geq 1$, $a_1,a_2,\ldots,a_k> 0$ and $c>0$ such that for every $q\geq 1$ and $r=0,1,\ldots,k-1$ we have 
\begin{equation}\label{eq:c_n asym}
c_{qk+r}=a_r \mu_w^{qk+r} (1+O(e^{-c q})).    
\end{equation}
In particular, there exist $C_1,C_2>0$ such that 
\begin{equation}\label{eq: nu 1}
C_1\mu_w^n\leq c_n\leq C_2\mu_w^n
\end{equation}
for every $n\geq 1$, and we can choose $C_1=1$ if $G$ is a transitive graph.    
\end{thm}

Our next theorem concerns the displacement of a typical self-avoiding walk of length $n$. We define $\mathbb{P}_n$ to be
the uniform measure on self-avoiding walks of length $n$ in $G$ starting from $o$, and we write $w=(w_0,w_1,\ldots,w_n)$ for a random
self-avoiding walk sampled from $\mathbb{P}_n$. We write $d_G(\cdot,\cdot)$ for the graph distance in $G$.

\begin{thm}\label{thm:ballistic}
If $G$ has infinitely many ends, and $\AUT(G)$ does not fix an end of $G$, then there exist constants $c,t>0$ such that $\mathbb{P}_n(d_G(w_0,w_n)\geq cn)\geq 1-e^{-tn}$ for every $n\geq 1$. 
\end{thm}

It turns out that if $\AUT(G)$ fixes an end of $G$ then it is non-unimodular. To see this, let $x$ and $y$ be two vertices where $y$ lies in some small separator separating $x$ from the fixed end $\omega$, but far away from $x$. Then the size of the orbit of $y$ under $\AUT(G)_x$ is bounded by some absolute constant, but the size of the orbit of $x$ under $\AUT(G)_y$ grows as we increase the distance between $x$ and $y$. 
Transitive graphs whose automorphism group admits a non-unimodular transitive subgroup were treated by Hutchcroft \cite{Hutch2019}, who proved that \eqref{eq: nu 1} and the conclusion of Theorem~\ref{thm:ballistic} hold in this context. Thus we obtain the following result.

\begin{thm}\label{thm:transitive}
Let $G$ be a transitive graph with infinitely many ends. Then there exists constant $C>0$ such that  
\[
\mu_w^n \leq c_n \leq C\mu_w^n.
\] 
Moreover, there exist constants $c,t>0$ such that $\mathbb{P}_n(d_G(w_0,w_n)\geq cn)\geq 1-e^{-tn}$ for every $n\geq 1$.
\end{thm}
\subsection{Tools and proof methods}

The proofs of Theorems~\ref{thm:c_n asym} and \ref{thm:ballistic} follow an approach similar to that of Alm and Janson in \cite{AlJa90}. Leveraging the tree decomposition, we decompose the graph into finite or infinite
parts, and analyse the restrictions of self-avoiding walks to these parts. These restrictions give rise to the notions of \textit{configurations}, \textit{shapes}, and \textit{arrangements}. Similar notions were also introduced in \cite{saw-mcfl}, but as we are working in a greater generality (in particular, parts can be infinite in our setting), some modifications in the definitions are necessary. Intuitively, each shape describes the restriction of a self-avoiding walk to a part of the tree decomposition, and each configuration describes how a self-avoiding walk crosses between two neighbouring parts. Arrangements are collections of compatible shapes and configurations, that is shapes and configurations whose partial descriptions of a self-avoiding walk fit together; see Figure \ref{fig:arrangement-sketch} for an example. 

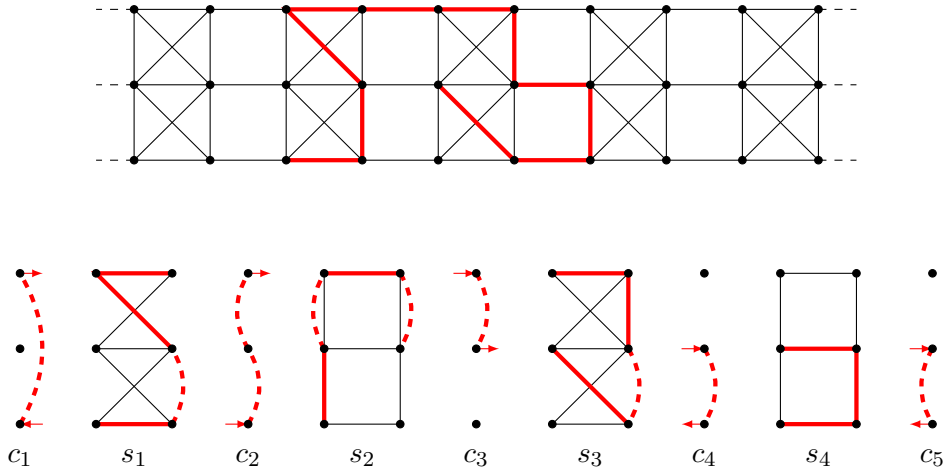
\begin{figure}
    \centering
    \colorlet{walkcol}{red} 
    \begin{tikzpicture}[vertex/.style={inner sep=1pt,circle,draw,fill}]
        
        \foreach \i in {0,...,9}
            \draw (\i,0) -- (\i,2);
        \foreach \i in {0,2,...,8}
        {
            \foreach \j in {0,1}
            {
                \draw (\i,\j)--(\i+1,\j+1);
                \draw (\i+1,\j)--(\i,\j+1);
            }
        }
        \foreach \j in {0,...,2}
        {
            \draw[dashed] (0,\j)--(-.5,\j);
            \draw[dashed] (9,\j)--(9.5,\j);
            \draw (0,\j)--(9,\j);
        }

        \draw[ultra thick,walkcol] (2,0)--(3,0)--(3,1)--(2,2)--(5,2)--(5,1)--(6,1)--(6,0)--(5,0)--(4,1);

        \foreach \j in {0,...,2}
            \foreach \i in {0,...,9}
                \node[vertex] at (\i,\j) {};

    \begin{scope}[yshift=-3.5cm,xshift=-1.5cm]
        \foreach \j in {0,...,2}
        {
            \foreach \i in {1,4,7,10}
                \draw (\i,\j)--(\i+1,\j);
        }
        \foreach \j in {0,1}
        {
            \foreach \i in {1,7}
            {
                \draw (\i,\j)--(\i+1,\j+1);
                \draw (\i+1,\j)--(\i,\j+1);
            }
            \foreach \i in {4,5,10,11}
                \draw (\i,\j)--(\i,\j+1);
        }

        \draw[-latex,walkcol] (0.3,0)--(0,0);
        \draw[-latex,walkcol] (0,2)--(0.3,2);
        \draw[ultra thick,walkcol,dashed] (0,0) to[bend right] (0,2);
        
        \draw[ultra thick,walkcol] (1,0)--(2,0);
        \draw[ultra thick,walkcol,dashed] (2,0) to[ bend right] (2,1);
        \draw[ultra thick,walkcol] (2,1)--(1,2)--(2,2);

        \draw[-latex,walkcol] (2.7,0)--(3,0);
        \draw[-latex,walkcol] (3,2)--(3.3,2);
        \draw[ultra thick,dashed,walkcol] (3,0) to[bend right] (3,1);
        \draw[ultra thick,dashed,walkcol] (3,1) to[bend left] (3,2);

        \draw[ultra thick,walkcol] (4,0)--(4,1);
        \draw[ultra thick,dashed,walkcol] (4,1) to[bend left] (4,2);
        \draw[ultra thick,walkcol] (4,2)--(5,2);
        \draw[ultra thick,dashed,walkcol] (5,2) to[bend left] (5,1);

        \draw[-latex,walkcol] (5.7,2)--(6,2);
        \draw[-latex,walkcol] (6,1)--(6.3,1);
        \draw[ultra thick,dashed,walkcol] (6,2) to[bend left] (6,1);

        \draw[ultra thick,walkcol] (7,2)--(8,2)--(8,1);
        \draw[ultra thick,dashed,walkcol] (8,1) to[bend left] (8,0);
        \draw[ultra thick,walkcol] (8,0)--(7,1);

        \draw[-latex,walkcol] (8.7,1)--(9,1);
        \draw[-latex,walkcol] (9,0)--(8.7,0);
        \draw[ultra thick,dashed,walkcol] (9,1) to[bend left] (9,0);

        \draw[ultra thick,walkcol] (10,1)--(11,1)--(11,0)--(10,0);

        \draw[-latex,walkcol] (11.7,1)--(12,1);
        \draw[-latex,walkcol] (12,0)--(11.7,0);
        \draw[ultra thick,walkcol,dashed] (12,1) to[bend right] (12,0);

        \node[anchor=north] at (0,-.2) {$c_1$};
        \node[anchor=north] at (3,-.2) {$c_2$};
        \node[anchor=north] at (6,-.2) {$c_3$};
        \node[anchor=north] at (9,-.2) {$c_4$};
        \node[anchor=north] at (12,-.2) {$c_5$};

        \node[anchor=north] at (1.5,-.2) {$s_1$};
        \node[anchor=north] at (4.5,-.2) {$s_2$};
        \node[anchor=north] at (7.5,-.2) {$s_3$};
        \node[anchor=north] at (10.5,-.2) {$s_4$};
         \foreach \j in {0,...,2}
            \foreach \i in {0,...,12}
                \node[vertex] at (\i,\j) {};
    \end{scope}
    \end{tikzpicture}
    \caption{A self avoiding walk and its decomposition into configurations and shapes. Note that every edge of the graph appears in exactly one part. Dashed edges correspond to detours whose edges do not lie in this part; configurations also keep track which side the detour lies on. The little arrows pointing in and out of every configuration indicate the side on which the first and last edge of the self-avoiding walk lie, respectively.}
    \label{fig:arrangement-sketch}
\end{figure}

Precise definitions are introduced in Section~\ref{sec: conf arrang}; for now we point out that there is a natural bijection between arrangements and self-avoiding walks and hence we may study arrangements in order to understand self-avoiding walks. We note that even though self-avoiding walks are highly non-Markovian, arrangements satisfy a spatial Markov property in the sense that compatibility of shapes and configurations is a local property; this means that in order to study arrangements, we do not have to consider their global structures, but may focus on the individual configurations and shapes. This observation is only useful for self-avoiding walks which intersect more than one part, as otherwise the self-avoiding walk is the same as its restriction to a part and we have not gained any insight. However, it turns out that the probability that a self-avoiding walk stays within one part is exponentially small -- see Lemma~\ref{lem:onepart}. 

The spatial Markov property of configurations gives rise to a natural recursive structure which enables us to obtain a system of equations linking the partition functions of arrangements `rooted at' given configurations.  This system can be encoded using a Jacobian matrix $\Jac(z)$, thereby reducing the problem of determining the asymptotic behaviour of self-avoiding walk to identifying which irreducible components of $\Jac(z)$ have spectral radius equal to $1$ at the critical point $z=1/\mu_w$. 

We will distinguish between irreducible components rooted at \emph{I-configurations} and \emph{U-configurations}. Precise definitions will be given later, but intuitively, an I-configuration means that the corresponding self-avoiding walk starts and ends on different sides, and a U-configuration between two parts means that the corresponding self-avoiding walk starts and ends on the same side. For instance, configurations $c_2$ and $c_3$ in the example in Figure \ref{fig:arrangement-sketch} are I-configurations, and configurations $c_1$, $c_4$, and $c_5$ are U-configurations.

As a key step in our approach, we establish that the spectral radius of each component of U-configurations is strictly smaller than $1$ at the critical point. After developing a geometric construction that allows us to transform walks, we employ an argument in the spirit of Kesten's pattern lemma \cite{KestenI} to show that self-avoiding walks consisting only of U-configurations are exponentially rare. In particular, this holds for \textit{self-avoiding polygons}, walks that start and end at the same vertex and are otherwise self-avoiding. The latter is known to imply that self-avoiding walk is ballistic \cite{P19}, and Theorem~\ref{thm:ballistic} follows.

To deduce Theorem~\ref{thm:c_n asym}, we establish that only one irreducible component possesses a spectral radius equal to $1$ at $z=1/\mu_w$, namely the \textit{persistent} I-configurations. Leveraging the Perron-Frobenius theorem, we conclude that the spectral radius is a simple eigenvalue, allowing us to deduce that the self-avoiding walk-generating function has only simple poles. This, in turn, implies that $c_n$ grows asymptotically like $\mu_w^n$.

\subsection{Paper organisation}

In Section \ref{sec: defs} we define the model and gather some relevant definitions. In Sections \ref{sec: tree dec} and \ref{sec: conf arrang} we introduce the notions of tree decompositions, configurations, and arrangements and prove some basic facts about them. In particular, we give a bijection between self-avoiding walks and certain configurations which will allow us to translate between the two objects. Next, in Section \ref{sec: system} we introduce the Jacobian matrix $\Jac(z)$ and we establish the connection with the spectral radius of its irreducible components. Subsequently, in Section~\ref{sec: analytic} we show the analyticity of $\Jac(z)$ and the partition functions of U-configurations at the critical point. In Section~\ref{sec: spectral} we prove that the spectral radius of several irreducible components is strictly smaller than $1$ at the critical point. Finally, in Section~\ref{sec: proofs} we finish the proof of Theorems~\ref{thm:c_n asym} and \ref{thm:ballistic}.

\section{Definitions and basic background}\label{sec: defs}

In this section, we gather some definitions that will be used throughout the paper. Most of our notation is standard, but there is some new notation as well; we hence encourage even readers familiar with graph theory to skim this section.

\subsection{Graph-theoretic definitions}
A \emph{digraph} $G$ consists of a set of \emph{vertices} $V(G)$ and a set of \emph{arcs} $E(G)$. Arcs are considered to be \emph{oriented}, so every arc $e$ is assigned an \emph{initial vertex} $e^- \in V(G)$ and a \emph{terminal vertex} $e^+ \in V(G)$, which are different vertices of $G$; note that we allow different edges to have the same initial and terminal vertex. A \emph{graph} is a digraph $G$ together with a bijection $\bar{} \; \colon E(G) \to E(G)$ such that $\bar{e}^+=e^-$, $\bar{e}^-=e^+$
and $\bar{\bar{e}}=e$. This means that arcs appear in pairs $e, \bar{e}$ having the same endpoints but different direction. We call such an unordered pair $\{e,\bar e\}$ an \emph{edge}. 
A digraph is called \emph{simple} if it contains no arcs which are different but have the same initial and terminal vertex. In this case we sometimes abuse notation and write $e=e^-e^+$. For a vertex $v$ of $G$, we denote by $\Eout(v)$ the set of all arcs with initial vertex $v$ and by $\deg(v)=\abs{\Eout(v)}$ the \emph{degree} of $v$, that is, $\deg(v)$ is the number of outgoing arcs of $v$. A digraph is called \emph{locally finite}, if all vertices have finite degree. 

\subsection{Definition of the model}

A \emph{walk} in a digraph is an alternating sequence $p=(v_0,e_1,v_1, \dots, e_n,v_n)$ of vertices $v_i \in V(G)$ and arcs $e_i \in E(G)$ such that $e_{i}^-=v_{i-1}$ and $e_i^+=v_i$ for every $i \in [n]$. We point out that we cannot define a walk purely as a sequence of vertices because there may be more than one edge connecting the same pair of vertices. The \emph{length} of the walk $p=(v_0,e_1,v_1, \dots, e_n,v_n)$, denoted by $\abs{p}$, is the number $n$ of arcs in $p$. Its \emph{initial vertex} is $p^-=v_0$ and its \emph{terminal vertex} is $p^+=v_n$. For convenience, we include in our definition the \emph{trivial walk} $(v)$ of length $0$, starting (and ending) at a vertex $v$ and also the \emph{empty walk} $\emptyset$ consisting of no vertices and no arcs.

A \emph{self avoiding walk} (or \emph{SAW}) is a walk $p$ whose vertices are pairwise different. We write $c_n(x)$ for the number of self-avoiding walks of length $n$ with initial vertex $x$, and $c_n$ for $c_n(o)$, where $o$ is some fixed root vertex in $G$. We define \[\mu_w(x)=\limsup_{n\to \infty}c_n(x)^{1/n}.\] We recall that on a quasi-transitive graph, the limit exists and does not depend on the base point \cite{MR0091568}. In this case, we will simply write $\mu_w$. 

A \emph{self-avoiding polygon} is a walk $p=(v_0,e_1,v_1, \dots, e_n,v_n)$ with $v_0=v_n$ and with $v_i\neq v_j$ for distinct pairs $i, j$ other than
the pair $0,n$. As in \cite{P19}, we identify two self-avoiding polygons which share the same set of edges. We write $p_n(x)$ for the number of self-avoiding polygons of length $n$ with initial vertex $x$. We also write $p_n$ for the number of self-avoiding polygons of length $n$ starting at the fixed root vertex $o$. We define
\[\mu_p(x)=\limsup_{n\to \infty}p_n(x)^{1/n},\]
and
\[\mu_p=\limsup_{n\to \infty}\left(\sup_{x\in V}p_n(x)\right)^{1/n}.\]

In certain cases, it is more convenient to work with self-avoiding walks that start at a vertex $x$ and end at a neighbour vertex of $x$. We call these walks \emph{self-avoiding returns}.

\subsection{Graph distance and connectedness}

Let $G$ be a digraph and let $u$ and $v$ be vertices of $G$. The distance $d_G(u,v)$ from $u$ to $v$ is the length of the shortest walk with initial vertex $u$ and terminal vertex $v$. If no such walk exists, $d_G(u,v)$ is infinite. We extend this notation to subsets of $V(G)$ in the obvious way: the distance of two sets is the minimal distance between elements of these sets. We also extend this definition to arcs of $G$. For technical reasons we prefer a pair of arcs to have distance 0 if and only if the two arcs have the same set of endpoints. Thus for two arcs $e,f \in E(G)$ we define 
\[
d_G(e,f) = 
\begin{cases}
0 \quad &\text{if } \{e^-,e^+\} = \{f^-, f^+\} \text{ (as sets),}\\
d_G(\{e^-,e^+\},\{f^-, f^+\})+1 \quad &\text{otherwise}.
\end{cases}
\]
Observe that if $G$ is a graph, the distance is symmetric on both $V(G)$ and $E(G)$, however it is only a metric on $V(G)$ because an arc and its inverse have distance $0$. 

A digraph $G$ is \emph{strongly connected} if for any two vertices $u,v$ there is a walk $p$ in $G$ starting at $u$ and ending at $v$. Note that this implies that any two vertices in the digraph are connected by walks in both directions.
A \emph{strong component} of $G$ is a maximal strongly connected subgraph. If $G$ is a graph, we usually omit the word ``strong'' in both notations. 

For $K \subseteq V(G)$ we denote by $G-K$ the subgraph obtained from $G$ by removing $K$ and all arcs incident to $K$. If removing $K$ disconnects $G$, then $K$ is called a \emph{separating set}. Furthermore, we denote by $G[K]$ the subgraph of $G$ \emph{induced} by $K$, that is the graph $G-(V(G)\setminus K)$.

A walk is \emph{closed} if its initial and terminal vertex coincide. A closed walk $p$ of length at least $3$ is called a \emph{cycle} if all vertices except the initial and terminal vertex are pairwise different; note that closed walks of length $2$ are not considered cycles in this paper even though they use different arcs for both directions. A \emph{tree} is a connected and cycle free graph which does not contain any cycle. A tree consisting only of vertices of degree at most 2 is called a \emph{path}. We point out that unlike walks, we consider paths to be graphs and therefore a path has no `direction'; a finite path can be seen as the support of a SAW. Given two disjoint subsets $A$ and $B$ of vertices of a graph $G$, an \emph{$A$--$B$-path} on $G$ is a finite path intersecting $A$ and $B$ only in its two endpoints.

\subsection{Surgery on walks and multi-walks}

Let $p$ be a walk. For two vertices $u$ and $v$ of $p$ we write $u p v$ for the maximal sub-walk of $p$ starting at $u$ and ending at $v$. If $u=v_0$ or $v=v_n$ we omit the corresponding vertex and denote the sub-walk by $p v$ or $u p$, respectively. We extend this notation even further. For walks $p_1, \dots, p_n$ and vertices $v_0, \dots, v_n$ in the respective walks, we denote the concatenation $(v_0 p_1 v_1)(v_1 p_2 v_2) \dots (v_{n-1} p_n v_n)$ of the sub-walks $v_{i-1} p_i v_i$ by $v_0 p_1 v_1 p_2 \dots p_n v_n$. If the terminal vertex $v$ of $p_1$ coincides with the initial vertex of $p_2$, we write $p_1p_2$ instead of $p_1vp_2$, and similarly for concatenations of multiple walks. If $e$ is an arc connecting the terminal vertex $v_1$ of $p_1$ to the initial vertex $v_2$ of $p_2$, then we write $p_1 e p_2$ instead of $p_1 v_1 (v_1,e,v_2) v_2 p_2$, and similarly for concatenations with more than two parts.

A \emph{multi-walk} $p$ is a sequence of vertices and arcs obtained by stringing together the sequences of vertices and arcs corresponding to walks $p_1, \dots, p_k$; the $p_i$ are called the \emph{walk components} of $p$. In other words, a multi-walk is a sequence of vertices and arcs, such that every arc in the sequence is preceded by its initial vertex and succeeded by its terminal vertex. Note that each of the walks $p_i$ is a sequence starting and ending with a vertex, so that the final vertex of $p_i$ and the initial vertex of $p_{i+1}$ will appear next to each other in the sequence $p$. In fact every appearance of two consecutive vertices in a multi-walk always marks the start of a new walk component.

\subsection{Rays and ends}
A \emph{ray} is a one-way infinite path and a \emph{double ray} is a two-way infinite path. 

Two rays in a graph $G$ are called \emph{equivalent}, if for every finite set $K \subseteq V(G)$ they end up in the same component of $G-K$, that is, all but finitely many of their vertices are contained in that component. An \emph{end} of $G$ is an equivalence class of rays with respect to this equivalence relation. Note that for every finite set $K \subseteq V(G)$ and every end $\omega$, there is a component $H$ of $G-K$ such that all but finitely many vertices of each ray in $\omega$ are contained in $H$; in such a case we say that $\omega$ \emph{lies in} $H$. Two ends $\omega_1$ and $\omega_2$ of a graph $G$ are \emph{separated} by $K$ if they lie in different components of $G-K$. Halin~\cite{MR190031} showed that an end containing arbitrarily many disjoint rays must contain an infinite family of disjoint rays, hence the maximum number of disjoint rays contained in an end $\omega$ is well defined and lies in $\N \cup \{\infty\}$. This number is called the \emph{size} of the end $\omega$. An end of finite size is called \emph{thin}, and an end of infinite size is called \emph{thick}.

\subsection{Graph automorphisms}

An \emph{automorphism} $\gamma$ of a graph $G$ is a permutation of $V(G)$ and $E(G)$ preserving the neighbourhood relation in $G$. More precisely, for every arc $e \in E(G)$ it satisfies $(\gamma e)^- = \gamma e^-$ and $\gamma \bar e = \overbar{\gamma e}$. The set of all automorphisms of $G$ forms a group which is called the \emph{automorphism group} of $G$ and denoted by $\AUT(G)$. For a subgroup $\Gamma \subseteq \AUT(G)$ we can define an equivalence relation on $V(G)$ by $u \sim v \iff \exists \;\gamma \in \Gamma\colon u = \gamma v$. The equivalence classes with respect to this relation are called \emph{orbits}, the orbit containing a vertex $v$ is denoted by $\Gamma v$. We say that $\Gamma$ acts \emph{(vertex transitively}, or simply \emph{transitively}, if there is exactly one orbit, and that it acts \emph{quasi-transitively}, if there are only finitely many orbits. In this case the graph $G$ is also called \emph{(quasi-)transitive}. Similarly, we say that $\Gamma$ acts \emph{arc transitively}, if the action of $\Gamma$ on $E(G)$ has a single orbit and \emph{edge transitively}, if it admits a single orbit on the set of edges. A subgraph $H$ of $G$ is called $\gamma$-invariant for $\gamma \in \AUT(G)$ if $\gamma(H)=H$.

It is well known, that if an infinite, locally finite, connected graph is quasi-transitive, then it has either one, two, or infinitely many ends. If it has one end, this end is thick. If it has two ends, both are thin and must have the same size. Finally, if it has infinitely many ends, then it must have thin ends. These and many more results were given by Halin in~\cite{MR335368}.

\subsection{Open subgraphs} 
It will be convenient for us to consider subgraphs which are allowed to contain edges of which they do not contain both endpoints.

An \emph{open subgraph} $H$ of $G$ consists of a vertex set $V(H) \subseteq V(G)$ and the set $E(H)=\{e \in E(G) \mid e^- \in V(H)\}$ of all outgoing arcs of $V(H)$ in $G$. Note that $H$ is uniquely determined by its vertex set $V(H)$; we say that $H$ is the open subgraph \emph{induced} by $V(H)$. We call $H$ finite, if $V(H)$ is finite. Arcs $e \in E(H)$ such that $e^+ \in V(G) \setminus V(H)$ are called \emph{boundary arcs} of $H$, and the set of all boundary arcs is denoted $\partial E(H)$. Observe that an open subgraph is a graph if and only if the set of boundary arcs is empty. The main idea behind the definition of open subgraphs is the following: for every partition of the vertex set $V(G)$, the graph $G$ is the disjoint union of the open subgraphs of $G$ induced by the sets in the partition. 

An open subgraph $H$ of $G$ is called \emph{open subtree} if $G[V(H)]$ is a tree and it is called \emph{open path} if $G[V(H)]$ is a path. The \emph{leaves} of an open subtree are the leaves of the underlying tree; in other words, when talking about the leaves of an open subtree, we ignore all boundary edges. For convenience, we let the length of an open path be the number of vertices it contains. Note that this is not the length of the underlying path---intuitively we may imagine that an open path contains an additional `half edge' in the for of an outgoing arc at each of its endpoints. An \emph{open star} $S$ in $G$ is an open path of length $1$; we point out that the underlying tree of an open star is not a star, but a single vertex. It is induced by a single vertex $s$ of $G$, called the \emph{center} of $S$ and we denote it by $\sta_G(s)$ or $\sta(s)$ if the graph is clear. Additionally, for an arc $e$ we write $\sta(e)$ for the open subtree with vertex set $\{e^-,e^+\}$. If $T$ is a tree, the \emph{(open) cone} $K_e$ in $T$ rooted at the arc $e$ consists of $e$ and the component of $T-\{e,\bar e\}$ containing the initial vertex $e^-$ of $e$.

Let $T$ be a tree and $e,f$ be two different arcs of $T$. We call $e$ and $f$ \emph{linkable} if they point away from each other, that is, they are not contained in the unique path in $T$ connecting $e^-$ and $f^-$. If $e$ and $f$ are linkable, the (unique) \emph{link} $W$ connecting $e$ and $f$ is the open path induced by this path: it connects $e^-$ and $f^-$, and $e$ and $f$ are boundary arcs at the two endpoints of $W$. Observe that whenever $e$ and $f$ are linkable, the length of their link $W$ is $d_T(e,f) = d_T(e^-,f^-)+1$.

\section{Tree decompositions}\label{sec: tree dec}

Recall that our general proof strategy is to decompose SAWs into configurations and shapes. In order to be able to do this, we first need to decompose the graph appropriately. To this end, the goal of the current section is to show that every quasi-transitive graph with more than one end admits a tree decomposition with certain `nice properties'. We start by giving some definitions. 

A \emph{tree decomposition} of a graph $G$ is a pair $\Tcal=(T,\Vcal)$, consisting of a tree $T$ and a function $\Vcal: V(T) \rightarrow 2^{V(G)}$ assigning a subset of $V(G)$ to every vertex of $T$, such that the following three conditions are satisfied:
\begin{enumerate}[label=(T\arabic*)]
\item \label{itm:td-coververtices}
$V(G)= \bigcup_{t \in V(T)} \Vcal(t)$.
\item \label{itm:td-coveredges}
For every arc $e \in E(G)$ there is a vertex $t \in V(T)$ such that $\Vcal(t)$ contains both endpoints of $e$.
\item \label{itm:td-nocrossedge}
$\Vcal(s) \cap \Vcal(t) \subseteq \Vcal(r)$ for every vertex $r$ on the unique $s$--$t$-path in $T$.
\end{enumerate}
For every $t \in V(t)$, the set $\Vcal(t)$ is called the \emph{part} of $t$. For an arc $e=st$ of $T$, the intersection $\Vcal(e)=\Vcal(s,t):= \Vcal(s) \cap \Vcal(t)$ is called the \emph{adhesion set} of $e$. Note that by definition $\Vcal(e) = \Vcal(\bar e)$.

A \emph{separation} of a graph $G$ is a pair $(A,B)$ of vertex sets such that $G[A] \cup G[B]=G$, in other words $A \cup B = V(G)$, and there are no edges between $A \setminus B$ and $B \setminus A$. The intersection $A \cap B$ is called the \emph{separator}, and $|A \cap B|$ is called the \emph{order} of $(A,B)$. A separation $(A,B)$ of finite order is said to \emph{separate} two ends $\omega_1$ and $\omega_2$ if one of them lies in a component of $B \setminus A$, and the other one lies in a component of $A \setminus B$. We say that $(A,B)$ \emph{minimally separates} $\omega_1$ and $\omega_2$ if $(A,B)$ has the minimal order among all separations separating $\omega_1$ and $\omega_2$. Note that every arc $a$ of the decomposition tree $T$ corresponds to a separation of $G$ with separator $\Vcal(a)$: we call 

\[\left(\bigcup_{t \in V(K_a)} \Vcal(t),\bigcup_{t \in V(K_{\bar{a}})} \Vcal(t)\right)\] 
the separation of $G$ \emph{induced by} $a$. Clearly, if the separation induced by $a$ is $(A,B)$, then the separation induced by $\bar a$ is $(B,A)$.

A tree decomposition $(T,\Vcal)$ of $G$ is called \emph{$\Gamma$-invariant} for a group $\Gamma \leq \AUT(G)$, if every $\gamma \in \Gamma$ maps parts onto parts and thereby induces an automorphism of $T$ (which we denote by $\gamma$ as well); this clearly induces an action of $\Gamma$ on $T$. 

The following result which guarantees the existence of a $\Gamma$-invariant tree decomposition is at the heart of our construction. It can be seen as a special case of \cite[Theorem 8.1]{MR3385727}; simply note that the separations $(A,B)$ with $|A\cap B| = k$ and infinite sides $A$ and $B$ correspond to a cut system in the sense of \cite{MR3385727}, and that the $G$-tree $T(\Ccal)$ can be interpreted as a tree decomposition whose induced separations are all contained in this cut system.

\begin{thm}
\label{thm:dunwoodykroen}
Let $\Gamma$ be a group acting quasi-transitively on a locally finite graph $G$. Assume that $G$ has more than one end, and let $k$ be the minimal integer such that there are two ends which can be separated by removing $k$ vertices. Then there is a $\Gamma$-invariant tree decomposition of $G$ in which every induced separation has order $k$ and separates two ends.
\end{thm}

Corollary \ref{cor:reducedtd} below describes the tree decompositions we will be working with for the rest of this paper. It is similar to \cite[Corollary 4.3]{hamann2018stallings} and \cite[Corollary 3.2]{saw-mcfl}. We still provide a proof for the convenience of the reader since these two results are slightly less general than ours. 

We call a tree decomposition \emph{reduced} if every induced separation minimally separates some pair of ends of $G$, and no two parts corresponding to adjacent vertices in $T$ coincide. We call it \emph{strongly reduced}, if in addition every induced separation has order $k$, where $k$ is the minimal integer such that there is a pair of ends which can be separated by removing $k$ vertices.

In order to extract a tree decomposition with the desired properties from the above theorem, we will need the following construction. Let $\Tcal = (T,\Vcal)$ be a tree decomposition of $G$, and let $F \subseteq E(T)$ be a subset of arcs such that $\bar e \in F$ whenever $e \in F$. We define the \emph{contraction} $\Tcal/F = (T/F,\Vcal/F)$ of the tree decomposition $\Tcal$ as follows.
The vertices of $T/F$ are components of $T-(E(T)\setminus F)$ with an arc between two of them if there is an arc in $T$ (which is necessarily an element of $E(T)\setminus F$) connecting the corresponding components. For a component $S\in V(T/F)$ we set $\Vcal(S) = \bigcup_{s \in V(S)} \Vcal(s)$. It is not hard to see that this defines a tree decomposition, and that the arcs of the decomposition tree are precisely the arcs in $E(T)\setminus F$. Moreover, the separations induced by an arc $e \in  E(T)\setminus F$ and its counterpart in $E(T/F)$ coincide.

\begin{cor}
\label{cor:reducedtd}
Let $\Gamma$ be a group acting quasi-transitively on a locally finite graph $G$. Then there is a strongly reduced, $\Gamma$-invariant tree decomposition $(T,\Vcal)$ such that the action of $\Gamma$ on $T$ is edge transitive, but not arc transitive.
\end{cor}
\begin{proof}
We start with the tree decomposition $\Tcal = (T, \Vcal)$ provided by Theorem \ref{thm:dunwoodykroen} and successively change it to satisfy the additional conditions.

First, let $e \in E(T)$ and let $F = E(T) \setminus (\Gamma e \cup \Gamma \bar e)$. Then $\Tcal/F$ is easily seen to be $\Gamma$-invariant, and the action of $\Gamma$ on $T/F$ is transitive on $E(T/F)$. 

We claim that $\Tcal/F$ is reduced. Assume that it is not. Then there are two parts corresponding to adjacent vertices of $T/F$ which coincide. Edge transitivity implies that all parts of $\Tcal/F$ coincide, and thus all adhesion sets coincide with the parts as well. This cannot be the case since $G$ is infinite, but the adhesion sets of $\Tcal/F$ are precisely the adhesion sets of $\Tcal$ corresponding to arcs in $\Gamma e$ and thus they are finite. Finally, $\Tcal/F$ is strongly reduced due to the size of the adhesion sets in the tree decomposition $\Tcal$ provided by Theorem \ref{thm:dunwoodykroen}. 

Note that if the action on $T/F$ in the above construction is not arc transitive, then we are done. Hence we may assume that we have a strongly reduced tree decomposition $\Tcal = (T, \Vcal)$ such that the action of $\Gamma$ on $T$ is arc transitive. Let $T'$ be the tree obtained by subdividing each edge in $T$ precisely once; more formally, $V(T') = V(T) \cup \{\{e,\bar e\} \mid e \in E(T)\}$, and $t \in V(T)$ and $\{e,\bar e\}$ are adjacent in $T'$ if and only if $t$ and $e$ are incident in $T$. Then it is straightforward to check that $\Tcal' = (T',\Vcal)$ defines a tree decomposition; recall that $\Vcal(st)$ is defined as $\Vcal(s) \cap \Vcal(t)$. This tree decomposition is still $\Gamma$-invariant, since the action on the tree $T$ induces an action on the tree $T'$ which takes parts to parts and adhesion sets to adhesion sets.

Next we show that $\Tcal'$ is strongly reduced. Let $st \in E(T)$. Then $|\Vcal(s)| = |\Vcal(t)|$ by arc transitivity, and consequently $|\Vcal(s)| > |\Vcal(st)|$ because otherwise the parts corresponding to $s$ and $t$ would coincide. Moreover, note that every adhesion set of $\Tcal'$ is also an adhesion set of $\Tcal$ since $\Vcal(s) \cap \Vcal(st) = \Vcal(s) \cap \Vcal(t)$. This shows that $\Tcal'$ is strongly reduced.

The action of $\Gamma$ is not transitive on the arcs of $T'$, since it is impossible to map the set $\Vcal(s)$ to the set $\Vcal(st)$ because the sets have different cardinalities. To see that the action is transitive on edges, note that by arc transitivity of the action on $T$, for any two arcs $e_1=s_1 t_1$ and $e_1=s_1 t_1$ there is $\gamma \in \Gamma$ such that $\gamma e_1 = e_2$ and thus in particular $\gamma s_1 = s_2$. We conclude that $\gamma$ maps the edge connecting $s_1$ and $\{e_1,\bar e_1\}$ to the edge connecting $s_2$ and $\{e_2,\bar e_2\}$ in $T'$.
\end{proof}

We now collect some useful properties of the tree decompositions provided by Corollary~\ref{cor:reducedtd}. For the remainder of this section let $\Gamma$ be a group acting quasi-transitively on a locally finite graph $G$ and  let $\Tcal = (T,\Vcal)$ be a tree decomposition as provided by Corollary \ref{cor:reducedtd}.

\begin{pro}
\label{pro:finitesubtree}
For every vertex $v \in V(G)$ there are only finitely many $s \in V(T)$ such that $v \in \Vcal(s)$.
\end{pro}

\begin{proof}
Assume for a contradiction that $v$ is contained in infinitely many parts. Let $U \subseteq V(G)$ be a maximal subset such that $v \in U$ and there are infinitely many parts containing all of $U$. Note that such a maximal subset exists because if a set of vertices is contained in two parts then by \ref{itm:td-nocrossedge} it is contained in two neighbouring parts. Therefore any set of vertices which is contained in infinitely many parts must be contained in an adhesion set, and thus there is an upper bound on the size of such a set.

By \ref{itm:td-nocrossedge}, the set $\{s \in V(T)\colon U \subseteq \Vcal(s)\}$ induces a subtree $T'$ of $T$. This subtree is infinite because there are infinitely many parts containing $U$. Define $F = E(T')$ and note that $U \subseteq \Vcal(e)$ for every $e \in F$. The separation $(A_e,B_e)$ induced by $e \in F$ separates two ends, hence $v$ must have neighbours $a_e \in A_e \setminus B_e$ and $ b_e \in B_e \setminus A_e$, otherwise it would be possible to separate the ends by fewer vertices.

By \ref{itm:td-coveredges}, both endpoints of the arc $v a_e \in E(G)$ must be contained in some part of the tree decomposition, and since $a_e \notin B_e$ this part is $\Vcal(s)$ for some $s \in K_e$. Similarly, both endpoints of the arc $vb_e \in E(G)$ are contained in a part $\Vcal(s)$ for some $s \in K_{\bar e}$.

Since $F$ contains infinitely many arcs, there is an infinite subset $S \subseteq V(T)$ such that $\Vcal(s)$ contains a neighbour of $v$ for every $s \in S$. But $v$ only has finitely many neighbours, so infinitely many parts contain the same neighbour $u$ of $v$.

Taking an appropriate infinite subset of $F$ (either the arcs of some infinite star, or arcs on a ray which are sufficiently far apart) we can make sure that for any two vertices in $S$, the path connecting them contains at least one vertex of $T'$. By \ref{itm:td-nocrossedge} this implies that $u$ is contained in $\Vcal(s)$ for infinitely many $s \in T'$; this contradicts the maximality of $U$.
\end{proof}

\begin{cor}\label{cor:finitesubtree}
    There is a constant $N$ such that $\Vcal(e) \cap \Vcal(f) = \emptyset$ whenever $e,f \in E(T)$ such that $d_T(e,f) \geq N$.
\end{cor}
\begin{proof}
    By Proposition~\ref{pro:finitesubtree} every vertex $v \in V(G)$ is contained in only finitely many parts of the tree decomposition. Moreover $G$ is quasi-transitive, so there is a constant $N$ such that every $v \in V(G)$ is contained in at most $N$ parts. Assume that $e$ and $f$ are arcs of $T$ such that $\Vcal(e) \cap \Vcal(f)$ contains a vertex $v$ and $d_T(e,f) \geq N$. Then the unique path in $T$ connecting $e^-$ and $f^-$ has at least $N+1$ vertices. It follows from \ref{itm:td-nocrossedge} that $v \in\Vcal(t)$ for each vertex $t$ of this path, contradicting the fact that every vertex is contained in at most $N$ parts.
\end{proof}

\begin{pro}
\label{pro:edgeplacement}
There is a map $\theta\colon E(G) \to V(T)$ such that for every arc $e \in E(G)$ 
\begin{enumerate}[label=(\roman*)]
    \item both endpoints of $e$ are contained in $\Vcal(\theta(e))$,
    \item $\theta(e) = \theta(\bar e)$, and
    \item $\gamma \theta(e) = \theta(\gamma e)$ for every $\gamma \in \Gamma$.
\end{enumerate}
\end{pro}

\begin{proof}
For every arc $e \in E(G)$, let $S(e) =\{s \in V(T)\colon e^-,e^+ \in \Vcal(s)\}$. The set $S(e)$ is non-empty by \ref{itm:td-coveredges}, the induced subgraph $T[S(e)]$ is a tree by \ref{itm:td-nocrossedge}, and this tree is finite by Proposition \ref{pro:finitesubtree}. 
Hence it either has a central vertex, or a central edge. In the former case, we let $\theta(e)$ be the central vertex of $T[S(e)]$. In the latter case, we note that, because $\Gamma$ acts edge transitively 
but not arc transitively, there is an orientation of the edges such that $\Gamma$ preserves the chosen orientations. Let $\theta(e)$ be the initial vertex of the central edge of $T[S(e)]$ with respect to this orientation.

Clearly, every $\gamma \in \Gamma$ which maps $e$ to $f$ also maps $S(e)$ to $S(f)$, and thus maps the central vertex or edge of $T[S(e)]$ to the central vertex or edge of $T[S(f)]$. Since $\gamma$ also preserves the orientation we picked above, we conclude that $\gamma \theta(e) = \theta(\gamma e)$ as desired.
\end{proof}

Now let $\Ecal(s)=\theta^{-1}(s) \subseteq E(G)$, where $\theta$ is the function given by Proposition \ref{pro:edgeplacement}. Then $\Ecal(s)$ is a subset of the edge set of $G[\Vcal(s)]$. 

Additionally we introduce for every arc $e=st$ of $T$ a new set of \emph{virtual arcs} $\Ecal(e)=\Ecal(st)$, such that every pair of vertices of $\Vcal(e)$ is connected by an arc in $\Ecal(e)$. By definition these arcs come in pairs connecting the same vertices but having different direction, so that $\Vcal(e)$ and $\Ecal(e)$ form a complete graph. Note that the sets $\Ecal(e)$ and $\Ecal(\bar{e})$ are disjoint by definition.

We define the \emph{adhesion graph} $\Gcal(e)=(\Vcal(e), \Ecal(e) \cup \Ecal(\bar{e})) (= \Gcal(\bar e)) $ and note that every pair of vertices is connected by precisely two arcs, and thus $\Gcal(e)$ is not simple unless it only consists of a single vertex.
In order to enhance readability, we usually write $\Ecal(s,t)$ instead of $\Ecal(st)$ and $\Gcal(s,t)$ instead of $\Gcal(st)$.

Finally, we assign to every vertex $s$ of $T$ the \emph{part graph} 
\[
\Gcal(s)=\left(\Vcal(s)\;,\; \Ecal(s) \uplus \biguplus_{e \colon e^-=s} \Ecal(e)\right).
\]
Again, $\Gcal(s)$ generally is not a simple graph since $\Ecal(s)$ and the various sets $\Ecal(s,t)$ potentially contain arcs with the same endpoints.

Recall that an open subtree $S$ of $T$ consists of a vertex set $V(S)$ and all arcs of $T$ starting at those vertices such that $T[V(S)]$ is a tree. The \emph{part graph induced by $S$} is 
\[
\Gcal(S)=\left(\bigcup_{t \in V(S)} \Vcal(t)\;, \; \bigcup_{t \in V(S)} \Ecal(t) \uplus \biguplus_{e \colon \partial E(S)} \Ecal({e})\right).
\]
This is a natural extension of part graphs; it is easy to see that $\Gcal(\sta(s))=\Gcal(s)$.

\begin{pro}
\label{pro:partsqtrans}
For every $s \in V(T)$, the setwise stabiliser $\Gamma_{\Vcal(s)}$ acts quasi-transitively on $\Gcal(s)$.
\end{pro}

\begin{proof}If $u \in \Vcal(t)$ does not lie in any adhesion set, then neither does any image of $u$ under a graph automorphism. In particular, any $\gamma \in \Gamma$ mapping $u$ to some vertex $v \in \Vcal(t)$ fixes $\Vcal(t)$ and thus, under the action of the stabiliser of $\Vcal(t)$ there are only finitely many orbits of vertices in $\Vcal(t)$ not contained in any adhesion set. 

Whenever $\gamma \in \Gamma$ fixes $s$ and maps a neighbour $t$ of $s$ onto some other neighbour $t'$, $\gamma$ lies in  $\Gamma_{\Vcal(t)}$ and maps the adhesion set $\Vcal(s,t)$ onto $\Vcal(s,t')$. By edge transitivity, there are at most two orbits of adhesion sets contained in $\Vcal(s)$ under the action of $\Gamma_{\Vcal(t)}$. As every adhesion set contains the same finite number of elements, $\Gamma_{\Vcal(t)}$ acts with finitely many orbits on the vertices which lie in adhesion sets of the tree decomposition.
\end{proof}

\begin{pro}\label{pro:distance-to-adhesion}
There is $M \in \mathbb N$ such that for every $s \in V(T)$ and every $v \in \Vcal(s)$ there is a neighbour $t$ of $s$ such that $d_G(v,\Vcal(s,t)) \leq M$. 
\end{pro}
\begin{proof}
If $u$ and $v$ are vertices whose distance to the nearest adhesion set differs, then they lie in different orbits with respect to the action of $\Gamma_{\Vcal(s)}$ because $\Gamma_{\Vcal(s)}$ maps adhesion sets to adhesion sets. Hence, if the distance from a vertex to the nearest adhesion set was unbounded, then the action of $\Gamma_{\Vcal(s)}$ would have infinitely many orbits on $\Vcal(s)$, thus contradicting Proposition \ref{pro:partsqtrans}.
\end{proof}

\begin{pro}
\label{pro:edges-dense}
Let $e$ and $f$ be arcs of $T$. If $\Gamma$ does not fix an end of $T$, then there is an automorphism $\gamma \in \Gamma$ such that $e$ and $\gamma(f)$ are linkable.

If $e$ and $f$ lie in the same orbit, then for every odd $k$ we can choose $\gamma$ such that $d_T(e, \gamma(f)) = k$.
\end{pro}
\begin{proof}
Since the action is transitive on edges but not on arcs, we can choose an orientation of the edges of $E$ which is preserved under the action of $\Gamma$. There are two possibilities: either every vertex has both incoming and outgoing arcs with respect to this orientation, or all arcs are oriented from one bipartite part to the other.

In the first case note that every vertex has at least two incoming and outgoing arcs (otherwise $\Gamma$ would fix an end), and hence there is $\gamma \in \Gamma$ such that $\gamma(f)^-=e^-$ and the two arcs are linked by the open path of length 1 containing only the vertex $e^-$. In the second case we can find an open path containing one or two vertices, depending on the orientations of the arcs $e$ and~$f$.

If $e$ and $f$ lie in the same orbit, then we can choose $\gamma$ such that $d_T(e, \gamma(f)) = 1$. Concatenation of these open paths shows that we can also choose $\gamma$ such that $d_T(e,\gamma(f))=k$ for any odd $k$.
\end{proof}

\begin{pro}\label{pro:adhesion-paths}
Let $e,f \in E(T)$. There is a family of $|\Vcal(e)|$ disjoint paths connecting $\Vcal(e)$ to $\Vcal(f)$ in $G$. Moreover, there is $N \in \mathbb N$ such that if $d_T(e,f) \geq N$, then for any pair of vertices $u \in \Vcal(e)$ and $v\in \Vcal(f)$ there is a $u$--$v$-path in $G$ which meets $\Vcal(e) \cup \Vcal(f)$ only in $u$ and $v$.
\end{pro}
\begin{proof}
Let $(A,B)$ and $(C,D)$ be the separations induced by $e$ and $f$, respectively. It is clear from the definition of induced separations that we may without loss of generality assume that $A \cap D \subseteq \Vcal(e) \cap \Vcal(f)$. There are ends $\omega_A$ and $\omega_D$ which lie in components of $A \setminus B$ and $D \setminus C$, respectively. If there were fewer than $|\Vcal(e)|$ disjoint paths connecting $\Vcal(e)$ to $\Vcal(f)$ in $G$, then by Menger's theorem it would be possible to separate these two ends by removing fewer than $|\Vcal(e)|$ vertices; this contradicts the fact that the tree decomposition is strongly reduced.

For the `moreover' part, let $e' \in E(T)$, and let $H$ be a finite connected subgraph of $G$ containing all vertices in $\Vcal(e')$. By Corollary~\ref{cor:finitesubtree} there is some $N' \in \mathbb N$ such that $\Vcal(f') \cap V(H) = \emptyset$ whenever the distance between $e'$ and $f'$ in $T$ is at least $N'$.

Now assume that the distance between $e$ and $f$ is at least $N=2N'$, and let $e'$ be an arc lying in the center of the path connecting $e$ and $f$ in $T$. The union of $H$ and the collection of disjoint paths from the first part is a connected subgraph of $G$ because each of the paths must pass through a vertex in $\Vcal(e')$. Vertices in $\Vcal(e) \cup \Vcal(f)$ have degree $1$ in this graph because of the choice of $N$. Hence they cannot appear as internal vertices of any path, and thus any path in this subgraph meets $\Vcal(e) \cup \Vcal(f)$ at most in its endpoints. Since the graph is connected, it contains the desired paths.
\end{proof}

\section{Configurations and arrangements}\label{sec: conf arrang}

The goal of this section is to utilize tree decompositions to build SAWs from so-called shapes living on the graphs $\Gcal (t)$ for $t \in V(T)$ and configurations living on  $\Gcal (e)$ for $e \in E(T)$ where $\Tcal=(T,\Vcal)$ is a tree decomposition of a quasi-transitive graph $G$; usually the one provided by Corollary \ref{cor:reducedtd}. In the case where $G$ is a one-dimensional lattice, this was already done by Alm and Janson \cite{AlJa90}. Here we use a more general approach similar to the one introduced in \cite{saw-mcfl} for tree decompositions whose parts are finite. As we also want to treat infinite parts, our definitions and notation are slightly different from the ones used there.

Let $G$ be connected, locally finite, simple and quasi-transitive, and let $\Tcal=(T,\Vcal)$ be a tree decomposition of $G$. Assume that we have a map $\theta$ mapping each edge to a part containing its endpoints as in Proposition \ref{pro:edgeplacement}. Define adhesion graphs and part graphs for this tree decomposition in the same way as we did after  Proposition \ref{pro:edgeplacement} for the tree decomposition provided by Corollary \ref{cor:reducedtd}.  A \emph{configuration} on an arc $e$ of $T$ with respect to $\Tcal$ is a triple $c=(q,x,y)$, where $x,y \in \{e,\bar{e}\}$ are not necessarily different orientations of $e$, and $q$ is either a SAW on the graph $\Gcal(e)$ or equal to the empty set $\emptyset$. In the latter case $c$ is called the \emph{empty configuration}.

We call $x$ the \emph{entry direction} and $y$ the \emph{exit direction} of $c$. The inverse of the configuration $c=(q,x,y)$ is the configuration $\bar c=(\bar q ,y,x)$ consisting of the reverse walk $\bar q$ of $q$, and which has entry direction $y$ and exit direction $x$.

A \emph{shape} on a vertex $s$ of $T$ with respect to $\Tcal$ is a SAW $p$ on the part graph $\Gcal(s)$. We say that a shape $p$ on $s$ and a configuration $(q,x,y)$ on an arc $e \in \Eout(s)$ are compatible if the following three conditions hold; the intersection of a walk $p$ with a subgraph is defined as the subsequence of $p$ consisting of the vertices and edges contained in that subgraph and therefore this intersection is in general not a walk but a multi-walk.
\begin{enumerate}[label=(C\arabic*)]
    \item \label{itm:compatible-intersection}
    $p \cap \Gcal(e) = q \cap \Gcal(s)$.
    \item \label{itm:compatible-in}
    If $x=e$, then $p$ starts in $\Vcal(e)$. 
    \item \label{itm:compatible-out}
    If $y=e$, then $p$ ends in $\Vcal(e)$. 
\end{enumerate}

An \emph{arrangement} on a finite open subtree $S$ of $T$ with respect to $\Tcal$ is a pair $A=(P,C)$ of maps assigning to every vertex $s \in V(S)$ a shape $P(s)$ and to every arc $e$ in $E(S)$ a configuration $C(e)=(Q(e),X(e),Y(e))=C(\bar e)$, where $X(e), Y(e)$ may be arcs in $E(S)$ or their inverses, such that for $s \in V(S)$ the following conditions hold:

\begin{enumerate}[label=(D\arabic*)]
    \item \label{itm:config-compatible}
    $P(s)$ and $C(e)$ are compatible for every $e \in \Eout(s)$.
    \item \label{itm:config-entry}
    There is at most one arc $e\in \Eout(s)$ such that $X(e) = e$. If there is no such arc, then $P(s)$ starts with a non-virtual arc.
    \item \label{itm:config-exit}
    There is at most one arc $e\in \Eout(s)$ such that $Y(e) = e$. If there is no such arc, then $P(s)$ ends with a non-virtual arc.
\end{enumerate}

The \emph{weight} $\norm{A}$ of the arrangement $A=(P,C)$ on the open subtree $S$ is the total number of non-virtual arcs contained in all the walks $P(s)$ for $s \in V(S)$, so $\norm{A}=\sum_{s \in V(S)} \norm{P(s)}$, where $\norm{P(s)}$ denotes the number of non-virtual arcs in $P(s)$.

The arrangement $A$ is called \emph{boring} on an arc $e \in E(S)$ if $X(e)=Y(e)$ and all arcs of $Q(e)$ are contained in $\Ecal(X(e))$; in this case we also say that the configuration $C(e)=(Q(e),X(e),Y(e))$ is \emph{boring}. Intuitively, if an arrangement is boring on an arc $e$, then this means that all non-trivial shapes lie on one side of this arc; in other words we won't lose any information by removing everything that lies on the other side of this arc from the open subtree. We can thus reduce an arrangement by iteratively pruning subtrees attached to edges on which the arrangement is boring until all such subtrees are trivial. We call an arrangement $A$ \emph{reduced} if it is non-boring on all arcs in $E(S)\setminus \partial E(S)$. Call a reduced arrangement \emph{complete} if the configuration $C(e)$ is boring and $X(e)= \bar e$ for every $e \in  \partial E(S)$.

Our goal is to establish a relation between (complete) arrangements on finite open subtrees $S$ of $T$ and self-avoiding walks of length at least 1 on $\Gcal(S)$. While our definition of arrangements differs slightly from the definitions of configurations in \cite{saw-mcfl}, we will follow the same strategy, thus our proofs are quite similar.

The main result of this section is Theorem \ref{thm:saw-to-arrangement} which intuitively states that every self avoiding walk has an arrangement associated to it and (subject to some technical conditions) this is a bijection. The technical details are quite messy, hence we start by providing a very rough sketch of the main ideas for the convenience of the reader. Tree decompositions can be made coarser by contracting edges, and finer by decontracting these edges again, and we can define corresponding operations on arrangements; see Figure \ref{fig:contraction} for a sketch and Constructions \ref{cons:contraction} and \ref{cons:projection} for details. Every self-avoiding walk $p$ consists of finitely many edges. Thus after some finite number of contractions in the tree decomposition, all edges of $p$ are contained in one part. Hence $p$ is a shape on this part, and this shape can be extended to a configuration on the open star. Decontracting this arrangement gives the arrangement on the original tree decomposition corresponding to $p$. We may also perform these steps in reverse order to translate an arrangement into a self-avoiding walk.

In the remainder of this section we provide the details to the above proof sketch. As a first step, we show that a self avoiding walk whose edges belong to a single part can indeed be transformed into an arrangement on the open star.

\begin{lem}\label{lem:arrangement-from-shape}
Let $S=\sta(s)$ be an open star in $T$ and $p$ be a shape on $s$. 
\begin{enumerate}[label=(\roman*)]
\item \label{itm:arrangement-existence}
There is an arrangement $A=(P,C)$ on $S$ such that $P(s)=p$.
\item \label{itm:config-from-shape}
For $e \in E(S)$ the walk $Q(e)$ of the configuration $C(e)$ is uniquely defined by $p$.
\item \label{itm:specific-entry-exit}
If $p$ starts in $\Vcal(f)$ for some $f \in E(S)$ and/or ends in $\Vcal(f')$ for some $f' \in E(S)$, we may choose $A$ such that $X(f)=f$ and/or $Y(f')=f'$ holds, respectively.
\end{enumerate}
\end{lem}
\begin{proof}
Let $e \in E(S)$. Our first goal is to construct a walk $q$ on $\Gcal(e)$ such that \ref{itm:compatible-intersection} is satisfied and show that this walk is unique. If $p \cap \Gcal(e) =\emptyset$, we set $q=\emptyset$ the empty walk. Otherwise  $p \cap \Gcal(e)$ is a multi-walk, let $p_1, \dots, p_n$ be its walk-components. Set $q=p_1 e_1 p_2 \dots e_{n-1} p_n$, where $e_i \in \Ecal(\bar e)$ connects $p_i^+$ and $p_{i+1}^-$. It is not hard to check that \ref{itm:compatible-intersection} is satisfied. Furthermore, by definition $\Vcal(e) \subseteq V(\Gcal(s))$ and $\Ecal(e) \subseteq E(\Gcal(s))$ holds, thus there is no other way to complete $q$.

Let now $P(s)=p$ and for $e \in E(S)$ write $Q(e)$ for the walk $q$ on $G(e)$ constructed above. If $P(s)$ starts in some adhesion set $\Vcal(f)$, we may choose $X(f)=f$ and $X(e)=\bar e$ for $e \in E(S)\setminus\{f\}$. Otherwise $P(s)$ starts with a non-virtual arc and we have to choose $X(e)=\bar e$ for every $e \in E(S)$. Similarly, if $P(s)$ ends in some adhesion set $\Vcal(f)$ we may choose $Y(f)=f$ and $X(e)=\bar e$ for $e \in E(S)\setminus\{f\}$, otherwise we have to choose $X(e)=\bar e$ for every $e \in E(S)$. In any case this construction satisfies \ref{itm:config-compatible} -- \ref{itm:config-exit}. 
\end{proof}

The following Lemma formalises the intuition behind the notion of boring configurations. In particular, it shows that it is usually sufficient to work with reduced arrangements, as any non-reduced arrangement can be restricted to a smaller open subtree without losing any information. 

\begin{lem} \label{lem:boringextension}
    Let $S=\sta(s)$ be an open star in $T$, let $e \in E(S)$ be one of its arcs and let $c=(q,x,y)$ be a non-empty boring configuration on $e$ such that $x=y=e$. Then there is a unique arrangement $A=(P,C)$ on $S$ such that $C(e)=c$. Additionally, $P(s)$ contains no non-virtual arcs and $C(f)$ is boring for every $f \in E(S)\setminus \{e\}$.
\end{lem}
\begin{proof}
By definition of boring configurations the walk $q$ contains only arcs in $\Ecal(e)$ and no arcs in $\Ecal(\bar e)$. Thus $q$ is a shape on $s$ and Lemma~\ref{lem:arrangement-from-shape} provides the existence of an arrangement $A=(P,C)$ on $S$ such that $P(s)=q$. Additionally $q$ starts and ends in $\Vcal(e)$, so we may choose $A$ such that $X(e)=Y(e)=e$. Note that this is the only way of choosing entry and exit directions which is consistent with the given boring configuration $c$, and it is easy to check that indeed $C(e)=c$.  

Let $f \neq e$ be an arc of $S$. Then properties \ref{itm:config-entry} and \ref{itm:config-exit} imply $X(f)=Y(f)=\bar f$ because $X(e)=Y(e) = e$. Also $P(s) \cap \Gcal(f)$ consists only of vertices, so \ref{itm:compatible-intersection} implies that all arcs of $Q(f)$ are contained in $\Ecal(\bar f)$, so that $C(f)$ is boring.
\end{proof}

\begin{rmk}\label{rmk:int-edges-non-empty}
Observe that for any arrangement $A=(P,C)$ on a finite open subtree $S$ of $T$ empty configurations can only occur on boundary arcs of $S$. This follows from the fact that if $X(e) = e$ for some interior arc $e$ of $S$, then the walk $P(e^-)$ has to start in $\Vcal(e)$ by \ref{itm:compatible-in} and thus $Q(e)$ is not empty by \ref{itm:compatible-intersection}.
\end{rmk}

In the next step we will construct coarser tree decompositions from a given tree decomposition $\Tcal=(T,\Vcal)$ of $G$ by contracting arcs of $T$ (similar to the construction used in the proof of Corollary \ref{cor:reducedtd}). 

We first recall the definition of contraction in the special case where only a single edge is contracted, and introduce some additional notation for this special case. Let $f$ be an arc of $T$. The tree decomposition $\Tcal/f=(T/f,\Vcal/f)$ of $G$ obtained from $\Tcal$ by contracting $f$ is given as follows. The vertex set $V(T/f)$ of the decomposition tree $T/f$ is obtained from $V(T)$ by replacing $f^-$ and $f^+$ by a single new vertex $s_f$ and its edge set $E(T/f)$ is obtained from $E(T)$ by deleting $f$ and $\bar{f}$ and for the remaining edges changing all endpoints in $\{f^-, f^+\}$ to the new vertex $s_f$. In other words, we contract $f$ and leave the names of vertices and edges unchanged wherever possible. Furthermore, the part corresponding to $s_f$ is $\Vcal/f(s_f)=\Vcal(f^-) \cup \Vcal(f^+)$, for all other vertices $t \in V(T/f)$ we define $\Vcal/f(t) = \Vcal(t)$. It is not hard to check that the result satisfies properties \ref{itm:td-coververtices} -- \ref{itm:td-nocrossedge}, thus being a tree decomposition. The part graph of the new vertex is $\Gcal(s_f)=\Gcal(\sta(f))$, all other part graphs and adhesion graphs are inherited from $\Tcal$ and stay the same.

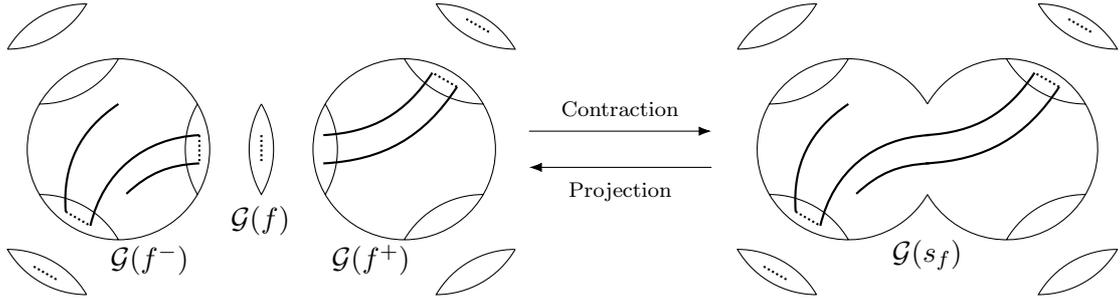
\begin{figure}
    \centering
    \begin{tikzpicture}[scale=1.2,vertex/.style={inner sep=1pt,circle,draw,fill},use Hobby shortcut]

        \draw 
        (30:1) arc(30:330:1cm)
        (90:1) arc(-30:-90:1cm)
        (-90:1) arc(30:90:1cm);

        \draw [shift = {(120:.7)}, rotate = 120]
        (30:1) arc(30:-30:1cm) arc(210:150:1cm);
        
        \draw [shift = {(-120:.7)}, rotate = -120]
        (30:1) arc(30:-30:1cm) arc(210:150:1cm);

        \begin{scope}[shift={(-30:1)},shift={(30:1)},rotate=180]
            \draw 
            (30:1) arc(30:330:1cm)
            (90:1) arc(-30:-90:1cm)
            (-90:1) arc(30:90:1cm);

            \draw [shift = {(120:.7)}, rotate = 120]
            (30:1) arc(30:-30:1cm) arc(210:150:1cm);
        
            \draw [shift = {(-120:.7)}, rotate = -120]
            (30:1) arc(30:-30:1cm) arc(210:150:1cm);
        \end{scope}

            \draw[thick] (90:.5) to [bend right] (-130:.9);
            \draw[thick,densely dotted] (-130:.9)--(-110:.9);
            \draw [shift = {(-120:.7)},thick,densely dotted] (-130:.9)--(-110:.9);
            \draw[thick] (-110:.9) to [bend left=37] (10:.9);
        \begin{scope}[shift={(-30:1)},shift={(30:1)},rotate=180]
            \draw[thick] (-10:.9) to[bend right=27] (-110:.9);
            \draw[thick,densely dotted] (-130:.9)--(-110:.9);
            \draw [shift = {(-120:.7)},thick,densely dotted] (-130:.9)--(-110:.9);
            \draw[thick] (-130:.9) to [bend left=27] (10:.9);
        \end{scope}
        \draw[thick] (-10:.9) to [bend right=20] (-80:.5);

        \node[yshift=-.8cm] at (-30:1) {$\Gcal(s_f)$};

        \begin{scope}[xshift=-8cm]
            \draw 
            (0:1) arc(0:360:1cm)
            (90:1) arc(-30:-90:1cm)
            (-90:1) arc(30:90:1cm)
            (30:1) arc(150:210:1cm);

            \draw [shift = {(0:.7)}]
            (30:1) arc(30:-30:1cm) arc(210:150:1cm);
            
            \draw [shift = {(120:.7)}, rotate = 120]
            (30:1) arc(30:-30:1cm) arc(210:150:1cm);
        
            \draw [shift = {(-120:.7)}, rotate = -120]
            (30:1) arc(30:-30:1cm) arc(210:150:1cm);

            \draw[thick] (90:.5) to [bend right] (-130:.9);
            \draw[thick,densely dotted] (-130:.9)--(-110:.9);
            \draw [shift = {(-120:.7)},thick,densely dotted] (-130:.9)--(-110:.9);
            \draw[thick] (-110:.9) to [bend left=37] (10:.9);
            \draw[thick,densely dotted] (10:.9)--(-10:.9);
            \draw [shift = {(0:.68)},thick,densely dotted] (10:.9)--(-10:.9);
            \draw[thick] (-10:.9) to [bend right=20] (-80:.5);

            \begin{scope}[xshift=1.4cm, shift = {(30:1)}, shift = {(-30:1)},rotate=180]
            \draw
            (0:1) arc(0:360:1cm)
            (90:1) arc(-30:-90:1cm)
            (-90:1) arc(30:90:1cm)
            (30:1) arc(150:210:1cm);
            
            \draw [shift = {(120:.7)}, rotate = 120]
            (30:1) arc(30:-30:1cm) arc(210:150:1cm);
        
            \draw [shift = {(-120:.7)}, rotate = -120]
            (30:1) arc(30:-30:1cm) arc(210:150:1cm);

            \draw[thick] (-10:.9) to[bend right=27] (-110:.9);
            \draw[thick,densely dotted] (-130:.9)--(-110:.9);
            \draw [shift = {(-120:.7)},thick,densely dotted] (-130:.9)--(-110:.9);
            \draw[thick] (-130:.9) to [bend left=27] (10:.9);
                
            \end{scope}

            \node[anchor=north] at (-70:1) {$\Gcal(f^-)$};
            \node[anchor=north,xshift=2cm,shift={(30:1)},shift={(-30:1)}] at (-110:1) {$\Gcal(f^+)$};
            \node[anchor=north,xshift=1cm,shift={(-30:1)}] at (0,-.1) {$\Gcal(f)$};

        \end{scope}

            \draw[-Latex] (-3.5,.2)--(-1.5,.2);
            \draw[Latex-] (-3.5,-.2)--(-1.5,-.2);
            \node[anchor=south] at (-2.5,.25) {\scriptsize Contraction};
            \node[anchor=north] at (-2.5,-.25) {\scriptsize Projection};

    \end{tikzpicture}
    \caption{Sketch of contraction (see Construction \ref{cons:contraction}) and projection (see Construction~\ref{cons:projection}) of arrangements with respect to the arc $f$.}
    \label{fig:contraction}
\end{figure}

Given an arrangement with respect to the tree decomposition $\Tcal$, we can define the contraction of the arrangement as described below. The construction and its inverse (Construction \ref{cons:projection}) are sketched in Figure \ref{fig:contraction}.

\begin{cons}\label{cons:contraction}
Let $A=(P,C)$ be an arrangement on the open subtree $\sta(f)$ of $T$ and let $\sta(s_f)$ be the open star in $T/f$ centered in the new vertex $s_f$ introduced by contracting $f$ in $T$. We construct a pair $(P/f,C/f)$ and show that it is an arrangement on $\sta(s_f)$. Configurations stay the same; we set $C/f(e)=C(e)$ for every $e \in \Eout(s_f)$. By Remark \ref{rmk:int-edges-non-empty} the configuration $C(f)$ is non-empty, let $Q(f)=(v_1, e_1, \dots, e_{k-1}, v_k)$. By \ref{itm:compatible-intersection} we have that $P(f^-)$ and $P(f^+)$ both visit $v_1, \dots, v_k$ in this order and no other vertices of $\Vcal(f)$. Let $a_0=X(f)$, let $a_k=Y(f)$ and for $j \in [k-1]$ let $a_j \in \{f,\bar{f}\}$ such that $e_j \in \Ecal(a_j)$.  
We define the walk $P/f(s_f)$ as the concatenation
\[
    P(a_0^+)v_1P(a_1^+)v_2\dots v_kP(a_k^+).
\]

In other words, $P/f(s_f)$ is obtained from $P(f^-)$ and $P(f^+)$ by deleting all arcs in $\Ecal(f)$ and then piecing the walk components of the resulting multi-walks together in a consistent manner.
\end{cons}

Intuitively it should be clear that applying Construction \ref{cons:contraction} to an arrangement yields an arrangement on the contracted tree decomposition. We will now prove this formally.

\begin{lem}
\label{lem:contraction-walk}
The walk $P/f(s_f)$ is a self-avoiding walk on $\Gcal(s_f)$ satisfying $ P/f(s_f) \cap \Gcal(f^-) = P(f^-)  - \Ecal(f)$ and $P/f(s_f) \cap \Gcal(f^+) = P(f^+)  - \Ecal(\bar f)$. In particular, the set of arcs contained in $P/f(s_f)$ consists of the arc sets of $P(f^-)  - \Ecal(f)$ and $P(f^+)  - \Ecal(\bar f)$.
\end{lem}
\begin{proof}
Let  $P/f(s_f) = P(a_0^+)v_1P(a_1^+)v_2\dots v_kP(a_k^+)$ as defined above. If one of $P(f^-)$ and $P(f^+)$ is the empty walk, then all claimed properties are trivially satisfied, so we may assume that both walks are non-empty. By \ref{itm:compatible-in}, $P(X(f)^-)$ must start in $v_1$ and by \ref{itm:compatible-out}, $P(Y(f)^-)$ must end in $v_k$, so both $P(a_0^-)v_1$ and $v_kP(a_k^-)$ are trivial. By definition $Q(f)$ is a walk consisting of arcs $e_1, \dots e_{k-1}$ and by \ref{itm:compatible-intersection}, $v_jP(a_j^-)v_{j+1}$ consists only of the virtual arc $e_j$. Combining these observations with the fact that $P(f^-)$ can be decomposed as $P(f^-) = P(f^-)v_1P(f^-)v_2\dots v_l P(f^-)$
we conclude that 
\[P/f(s_f) \cap \Gcal(f^-) = P(f^-)  - \Ecal(f),\] and similarly for $f^+$. This implies that $P/f(s_f)$ uses no vertex more than once: for vertices in $\Vcal(f)$, this holds by definition, for vertices outside of $\Vcal(f)$, this follows from the fact that $P(f^-)$ and $P(f^+)$ are self-avoiding. Hence $P/f(s_f)$ is self-avoiding.
\end{proof}

\begin{lem}
The pair $A/f=(P/f,C/f)$ defined in Construction~\ref{cons:contraction} is an arrangement on $\sta(s_f)$.
\end{lem}
\begin{proof}
We have already seen that $P/f(s_f)$ is a shape and by definition $C/f(e)=C(e)$ is a configuration, so we only need to verify \ref{itm:config-compatible} -- \ref{itm:config-exit}. Let $e \in \sta(s_f)$ and assume without loss of generality that $e^-=f^-$ in $T$. Then by construction
\[
P/f(s_f) \cap \Gcal(e) = P(f^-) \cap \Gcal(e) = Q(e) \cap \Gcal(f^-)=Q(e) \cap \Gcal(s_f),
\]
so \ref{itm:compatible-intersection} holds. Furthermore $X(e)=e$ if and only if $P(f^-)$ starts in $\Vcal(e)$ and additionally $X(f)=\bar f$ by \ref{itm:config-entry} and thus $P(f^+)$ starts in $\Vcal(f)$. Then by construction the starting vertices of $P/f(s_f)$ and $P(f^-)$ coincide, so $P/f(s_f)$ starts in $\Vcal(e)$ and \ref{itm:compatible-in} is satisfied. Finally either $X(f)=f$ or $X(f)=\bar{f}$, so again by \ref{itm:config-entry} at most one arc $e \in \sta(s_f)$ can satisfy $X(e) = e$. If there is no such arc, then each $e \in \Eout(X(f)^-)$ satisfies $X(e)=\bar e$, so $P(X(f)^-)$ starts with a non-virtual arc, so $P/f(s_f)$ starts with a non-virtual arc and \ref{itm:config-entry} follows. Finally \ref{itm:compatible-out} and \ref{itm:config-exit} follow analogously by considering exit directions.
\end{proof}

We have managed to contract an arrangement on the open tree $\sta(f)$ of $T$ to obtain an arrangement on the open star $\sta(s_f)$ of $T/f$. In the next step, we want to do the converse, projecting a given arrangement $A=(P,C)$ on $\sta(s_f)$ down to obtain an arrangement $\pi A=(\pi P, \pi C)$ on $\sta(f)$, see again Figure \ref{fig:contraction} for a sketch. For technical reasons we assume that $P(s_f)$ meets $\Vcal(f)$.

\begin{cons}\label{cons:projection}
As one might expect, we choose $\pi C(e)=C(e)$ for $e \in \partial \Eout(f)$.
By definition $P(s_f) \cap \Gcal(f)$ is a multi-walk on the adhesion graph $\Gcal(f)$. Each of its walk components consists only of a single vertex $v_i$. Therefore the walk $P(s_f)$ can be written as a concatenation $q_0 \dots q_k$ of sub-walks $q_i$, where $q_0=P(s_f) v_1$, $q_k=v_k P(s_f)$ and $q_j=v_j P(s_f) v_{j+1}$ for $j \in [k-1]$. Note that while $q_0$ and $q_k$ may be trivial, all other $q_j$ must have length at least 1. By the basic properties of tree-decompositions, the arcs of each (non-trivial) $q_j$ are contained in exactly one of $\Gcal(f^-), \Gcal(f^+)$. Let $a_j \in \{f,\bar f\}$ be such that $q_j \in \Gcal(a_j^+)$; in other words, $a_j$ marks on which side of $f$ the part containing $q_j$ lies. For $j \in [k-1]$ let $e_j$ be the (virtual) arc in $\Ecal(a_j)$ connecting $v_j$ and $v_{j+1}$. We are able to construct shapes $\pi P(s)$ for $s \in \{f^-, f^+\}$ and a configuration $\pi C(f)=(\pi Q(f), \pi X(f),\pi Y(f))=\pi C(\bar f)$. The shape $\pi P(s)$ is given as the concatenation $p_0 p_1 \dots p_k$, where 
\[
 p_j=\begin{cases} 
 q_j \quad &\text{if }  a_j^+=s \\
 (v_j, e_j, v_{j+1}) \quad &\text{if } j \in [k-1] \text{ and } a_j^+\neq s \\
 (v_j) \quad &\text{if } j \in \{0,k\} \text{ and } a_j^+\neq s.
 \end{cases}
\]
The walk $\pi Q(f)$ is given as $\pi Q(f)=(v_1, e_1, v_2, \dots, e_{k-1}, v_k)$. Finally, by \ref{itm:config-entry} there is at most one $e \in \Eout(s_f)$ such that $X(e) = e$. If such an arc exists, we choose $\pi X(f)\in \{f, \bar f\}$ such that $\pi X(f)^+=e^-$. Otherwise $P(s_f)$ starts with a non-virtual arc $e_0 \in \Ecal(s_f)$ and we choose $\pi X(f)\in \{f, \bar f\}$ such that $e_0 \in \Ecal(\pi X(f)^+)$. Similarly if there is an arc $e \in \Eout(s_f)$ such that $Y(e) = e$ we choose $\pi Y(f)$ such that $\pi Y(f)^+=e^-$, otherwise we choose $Y(f)$ such that the last arc of $P(s_f)$ is contained in $\Ecal(\pi Y(f)^+)$.
\end{cons}

We no prove that applying Construction \ref{cons:projection} to an arrangement gives an arrangement.

\begin{lem} \label{lem:projection-walk}
The walks $\pi P(f^-)$ and $\pi P(f^+)$ are self-avoiding walks on $\Gcal(f^-)$ and $\Gcal(f^+)$ satisfying $P(s_f) \cap \Gcal(f^-) = \pi P(f^-) - \Ecal(f)$ and $P(s_f) \cap \Gcal(f^+) = \pi P(f^+) - \Ecal(\bar f)$.
\end{lem}
\begin{proof}
By construction $\pi P(f^-)$ consists of all walk-components $q_j$ of $P(s_f)$ contained in $\Gcal(f^-)$, while those $q_j$ not contained in $\Gcal(f^-)$ are replaced by virtual arcs $e_j \in \Ecal(f)$ (shortcuts) connecting the same endpoints. In particular $\pi P(f^-)$ is a walk and $P(s_f) \cap \Gcal(f^-) = \pi P(f^-) - \Ecal(f)$. Furthermore it is self-avoiding because its sequence of vertices also occurs in the SAW $P(s_f)$. The statements for $\pi P(f^+)$ follow analogously.
\end{proof}

\begin{lem} \label{lem:projection-config}
The map $\pi A=(\pi P,\pi C)$ defined in Construction~\ref{cons:projection} is an arrangement on $\sta(f)$.
\end{lem}
\begin{proof}
We show that \ref{itm:config-compatible} -- \ref{itm:config-exit} holds for $s=f^-$, the other vertex $f^+$ is treated analogously. For $e \in \Eout(f^-) \setminus\{f\}$ compatibility follows from the fact that $\pi P(f^-) \cap \Gcal(e)=P(s_f) \cap \Gcal(e)$ and
\[
\pi X(e)= e \iff X(e)= e \iff P(s_f) \text{ starts in } \Vcal(e) \iff \pi P(f^-) \text{ starts in } \Vcal(e).
\]
and the analogue statement for $\pi Y(e)$. We still need to consider the arc $f$. By construction $\pi Q(f) \cap \Gcal(f^-)$ is a non-empty multi-walk consisting of all virtual arcs $e_j$ in $\Ecal(f)$ such that $q_j \in \Gcal(f^+)$, which are exactly the virtual arcs in $\pi P(f^-) \cap \Gcal(f)$. Furthermore $\pi X(f)=f$ holds either if $X(e)=e$ holds for some arc $e \in \Ecal(f^+) \setminus \{\bar f\}$ or if $q_0$ is non-trivial and contained in $\Gcal(f^+)$. In both cases $\pi P(f^-)$ starts at $v_0 \in \Vcal(f)$. Similarly, if $\pi Y(f) = f$, then $\pi P(f^-)$ ends at $v_k \in \Vcal(f)$, so \ref{itm:config-compatible} is satisfied. Again by construction $\pi X(f)= f$ can only hold if $\pi X(e)=\bar e$ for all $e \in \Ecal(f^-) \setminus \{\bar f\}$, so at most one  $e \in \Ecal(f^-)$ satisfies $\pi X(e) \neq e$. If there is no such arc, then $P(s_f)$ starts with a non-virtual arc which is by construction contained in $\Gcal(f^-)$. In particular $\pi P(f^-)$ starts with a non-virtual arc, yielding \ref{itm:config-entry}. Similar arguments for the exit direction $\pi Y$ yield \ref{itm:config-exit}. We conclude that $\pi A$ is indeed an arrangement on $\sta(f)$. 
\end{proof} 

Next we show that contraction and projection are inverses of one another. The technical condition that the shape meets $\Vcal(f)$ ensures that all shapes are non-empty. This is crucial for the bijection between self-avoiding walks and arrangements. Without this technical condition, every self-avoiding walk would correspond to not just one, but infinitely many different arrangements, obtained by attaching arbitrarily many boring configurations and suitable (often empty) shapes.

\begin{pro} \label{pro:bij-edge-contraction}
Contraction of arrangements $A \mapsto A/f$ defines a bijection between arrangements on the open subtree $\sta(f)$ of $T$ with respect to $\Tcal$ and arrangements on the open star $\sta(s_f)$ centered in the contracted vertex $s_f$ of $T/f$ with respect to $\Tcal/f$ whose shape on $s_f$ meets $\Vcal(f)$. Its inverse map is the projection $A \mapsto \pi A$ of arrangements. Moreover $\norm{A}= \norm{A/f}$ holds.
\end{pro}
\begin{proof}
We have already seen that $A \mapsto A/f$ maps arrangements on $\sta(f)$ to arrangements on $\sta(s_f)$ and that $A \mapsto \pi A$ maps arrangements on $\sta(s_f)$ to arrangements on $\sta(f)$. By Remark~\ref{rmk:int-edges-non-empty} the walk $Q(f)$ of the configuration $C(f)$ is non-empty, so by construction $P/f(s_f)$ intersects $\Vcal(f)$. Furthermore by Lemma~\ref{lem:contraction-walk} the set of non-virtual arcs of $P/f(s_f)$ is the disjoint union of the sets of non-virtual arcs of $P(f^-)$ and $P(f^+)$, so $\norm{A}= \norm{A/f}$ holds. 

Let $A$ be an arrangement on $\sta(f)$. We claim that $\pi(A/f)=A$. 
By construction $\pi(A/f)$ is an arrangement on $\sta(f)$ such that $\pi (C/f)(e) = C/f(e)=C(e)$ for every $e \in \Eout(f)$. Additionally, by Lemmas \ref{lem:contraction-walk} and \ref{lem:projection-walk} we have that $\pi(P/f)(f^-)$ satisfies 
\[
\pi (P/f)(f^-)-\Ecal(f)=P/f(s_f)\cap \Gcal(f^-)=P(f^-)-\Ecal(f).
\]
Any virtual arcs of $\pi P/f(f^-)$ in $\Ecal(f)$ must connect the walk-components of $\pi P/f(f^-)-\Ecal(f)$, thus we conclude $\pi (P/f)(f^-)=P(f^-)$. In a similar way we obtain $\pi (P/f)(f^+)=P(f^+)$. 

It remains to show that $\pi(C/f)(f)=C(f)$. Equality of the walks $\pi (Q/f)(f)$ and $Q(f)$ follows directly from \ref{itm:compatible-intersection} of compatibility: Both walks $\pi (Q/f)(f)$ and $Q(f)$ consist of all virtual arcs of the shapes $\pi (P/f)(f^-)=P(f^-)$ and $\pi (P/f)(f^+)=P(f^+)$ contained in $\Gcal(f)$ and also the order of these arcs coincide. Finally, if there is some arc $e \in \partial \Eout(f)$ such that $X(e)= e$, then also $\pi (X/f)(e) = e$ and thus \ref{itm:config-entry} yields $X(f)= \pi (X/f)(f)$. Otherwise $X(e)=\bar e$ holds for all $e \in \Eout(X(f)^+)$, so by \ref{itm:config-entry} the walk $P(X(f)^+)$ starts with a non-virtual arc $e_0$ and by \ref{itm:compatible-in} the walk $P(X(f)^-)$ starts in $\Vcal(f)$. In particular, the first arc of $P/f(s_f)$ is $e_0$ and by construction $\pi (X/f)(f)=X(f)$. Analogous arguments yield $\pi (Y/f)(f)=Y(f)$.

The proof that $(\pi A)/f=A$ holds for every arrangement $A=(P,C)$ on $\sta(s_f)$ such that $P(s_f)$ intersects $\Vcal(e)$ works similarly and is left to the reader.
\end{proof}

Let $S$ be a finite open sub-tree of $T$. Starting with a given tree decomposition $\Tcal=(T,\Vcal)$ of $G$ the process of edge contraction can be iteratively applied to contract each interior edge of $S$. In this way $S$ is contracted into an open star centered at a single vertex $s_S$. We denote the obtained tree decomposition by $\Tcal/S=(T/S,\Vcal/S)$. As before, an arrangement $A$ on $S$ will be contracted to an arrangement $A/S$ on the open star $\sta(s_S)$ of $T/S$. 

\begin{rmk} \label{rmk:treedecomp}
It is not hard to see that neither the contracted tree decomposition $\Tcal/S=(T/S, \Vcal/S)$ nor the contracted arrangement $A/S=(P/S,C/S)$ on $\sta(s_S)$ depends on the order of edge contractions: Clearly the tree obtained by consecutive edge contractions does not depend on the order. Additionally $\Vcal/S(s_S)$ is the union of all vertices in $\Vcal(s)$ for $s \in V(S)$ and $\Vcal/S(t)=\Vcal(t)$ for $t \in V(T/S)\setminus \{s_S\}$. In particular $\Tcal/S$ does not depend on the order of contractions of internal edge in $S$.  

The configurations on boundary arcs remain the same after each step of contraction of $A$, so in particular $C/S(e)=C(e)$ holds for any arc $e \in \Eout(s_S)=\partial E(S)$. Additionally, by Lemma~\ref{lem:contraction-walk} the walk $P(s_S)$ consists of all non-virtual arcs of walks $P(s)$ for $s \in V(S)$ and its direction is uniquely defined. Again, this representation of $A/S$ on $S$ does not depend on the order of contractions of edges in $S$.
\end{rmk}

We immediately obtain the following corollary of Proposition~\ref{pro:bij-edge-contraction}.

\begin{cor}\label{cor:bij-tree-contraction}
Let $S$ be an open subtree of $T$. The contraction of arrangements $A \mapsto A/S$ defines a bijection between arrangements on $S$ with respect to $\Tcal$ and arrangements on the open star $\sta(s_S)$ centered in the contracted vertex $s_S$ of $T/S$ with respect to $\Tcal/S$ whose shape on $s_S$ meets $\Vcal(s)$ for every $s \in V(S)$.
Moreover $\norm{A}= \norm{A/S}$ holds.
\end{cor}

Corollary \ref{cor:bij-tree-contraction} can be used to define a map $\varphi$ translating any arrangement $A$ on an open subtree $S$ of $T$ to the self-avoiding walk $\varphi(A)=P/S(s_S)$ on the graph $\Gcal(S)$. Observe that by definition the length of $\varphi(A)$ is $\norm{A}$. We call $\varphi(A)$ the walk \emph{represented} by $A$ and $A$ a \emph{representation} of $\varphi(A)$. Theorem \ref{thm:shape-to-arrangement} below shows that every SAW has some representation. 

Let $S$ be a not necessarily finite subtree of $T$ and $w$ be a SAW on $\Gcal(S)$. The \emph{support} of $w$ in $T$ is the smallest open subtree $S'$ of $S$ such that all arcs of $w$ are contained in $\Gcal(S)$. 

\begin{thm}\label{thm:shape-to-arrangement}
Let $T'$ be an open subtree of $T$ and let $w$ be a SAW on $\Gcal(T')$. Then there is an arrangement $A=(P,C)$ on the support of $w$ in $T$ representing $w$. Moreover, if $w$ starts in $\Vcal(f)$ for some $f \in \partial E(T')$ and/or ends in $\Vcal(f')$ for some $f' \in \partial E(T')$, we may choose $A$ such that $X(f)=f$ and/or $Y(f')=f'$ holds, respectively.
\end{thm}
\begin{proof}
Let $S$ be the support of $w$ in $T$. By definition $w$ is a shape on the contracted vertex $s_{S}$ of $T/S$, thus Lemma~\ref{lem:arrangement-from-shape}~\ref{itm:arrangement-existence} provides us with an arrangement $A'=(P',C')$ on $\sta_{T/S}(s_{S})$ such that $P'(s_{S})=w$. By minimality of $S$ the shape $w$ meets $\Vcal(s)$ for every $s \in V(S)$. Thus Corollary~\ref{cor:bij-tree-contraction} provides us with an arrangement $A$ on $S$ such that $A/S=A'$ and this arrangement satisfies $\varphi(A)=w$. The moreover part follows from Lemma~\ref{lem:arrangement-from-shape}~\ref{itm:specific-entry-exit}.
\end{proof}

We have just seen that every arrangement $A$ on an open subtree $S$ of $T$ represents a self-avoiding walk $\varphi(A)$ on $\Gcal(S)$ and that every such walk is represented by some arrangement. However, in general there can be several different arrangements representing a single walk. To enumerate SAWs on $G$, we need each walk to be represented by exactly one arrangement. This can be done by restricting ourselves to complete arrangements, as the upcoming Theorem~\ref{thm:saw-to-arrangement} shows. For its proof, we need the following auxiliary result.

\begin{lem}\label{lem:non-virtual-edges-in-leaves}
Let $A=(P,C)$ be a complete arrangement on an open subtree $S$ of $T$ and let $t$ be a leaf of $S$. Then $P(t)$ contains a non-virtual arc.
\end{lem}
\begin{proof}
Let $t$ be a leaf of $S$. If $S=\sta(t)$, completeness implies that $X(e)=\bar e$ for every $e \in E(S)$. Thus by \ref{itm:config-entry} the shape $P(t)$ starts with a non-virtual arc. 

Suppose now that $S$ is not an open star and let $f$ be the unique arc in $\Eout(t)\setminus \partial E(S)$. By completeness of $A$ we have $X(e)=Y(e)=\bar e$ for every $e \in \partial E(S)$. If $X(f)=\bar f$ or $Y(f)=\bar f$, then again  $P(t)$ has to start or end with a non-virtual arc. Thus we may assume that $X(f)=Y(f)=f$. By reducedness $C(f)$ is not boring, so $Q(f)$ has to contain a virtual arc in $\Ecal(\bar f)$. Then by compatibility \ref{itm:compatible-intersection} the shape $P(t)$ contains a sub-walk $p$ connecting the two endpoints of this walk. Now $C(e)$ is boring for every $e \in \partial E(S)$ and thus $p$ cannot contain any virtual arcs. We conclude that $P(t)$ contains a non-virtual arc.
\end{proof}

\begin{thm}\label{thm:saw-to-arrangement}
Let $w$ be a self-avoiding walk on $G$. Then there is a unique complete arrangement representing $w$.
\end{thm}
\begin{proof}
We start by constructing a representation of $w$ as in the proof of Theorem~\ref{thm:saw-to-arrangement}. Let $S$ be the support of $w$ in $T$ and let $S'=\sta_{T/S}(s_{S})$. Then Lemma~\ref{lem:arrangement-from-shape}~\ref{itm:arrangement-existence} provides us with an arrangement $A'=(P',C')$ on $S'$ such that $P'(s_{S})=w$. Note that we can choose $A'$ to be complete: The walk $w$ starts and ends with a non-virtual arc and $X'(e)=Y'(e)=\bar e$ is a valid choice for every $e \in E(S')$. Additionally, all arcs of $w$ are non-virtual arcs of $\Gcal(S)=\Gcal(s_S)$, thus $Q'(e)$ contains only arcs in $\Ecal(\bar e)$ and $C'(e)$ is boring. Corollary~\ref{cor:bij-tree-contraction} provides us with an arrangement $A$ on $S$ such that $A/S=A'$ and thus $\varphi(A)=w$. By Lemma \ref{lem:boringextension} whenever $C(e)$ is boring for some arc $e$ of $S$, $P(s)$ cannot contain any non-virtual arcs for every $s \in V(K_{X(e)}) \cap V(S)$. By minimality of $S$ we conclude $e \in \partial E(S)$. Thus $A$ is complete. 

For the uniqueness part, observe that whenever $A$ is a complete arrangement on a subtree $S$ of $T$ representing $w$, by Lemma \ref{lem:non-virtual-edges-in-leaves} this subtree must be the support of $w$. By Corollary~\ref{cor:bij-tree-contraction} each $A$ corresponds to a unique arrangement $A/S$ on $S'$ and clearly $A/S$ must be complete. By completeness $X/S(e)=Y/S(e)=\bar e$ holds for every $e \in E(S')$ and by Lemma~\ref{lem:arrangement-from-shape}~\ref{itm:config-from-shape} the configuration $C/S(e)$ is uniquely defined by $w=P/S(s_S)$, so $A/S$ is uniquely defined by $w$.
\end{proof}

\begin{rmk}\label{rmk:source-of-arrangement}
For any complete arrangement $A$ on a finite open subtree $S$ of $T$ there is exactly one vertex $s_0$ of $S$ such that $X(e)=\bar e$ for every $e \in \Eout(s_0)$. This can be easily seen by double counting the set $\{(s,e) \mid e^-=s \text{ and } X(e) = e\}$ and using \ref{itm:config-entry} and that configurations on boundary arcs $e$ satisfy $X(e)=\bar e$. We call $s_0$ the \emph{source} of the complete arrangement $A$.
Similarly there is a unique vertex $t_0$ of $S$ such that $Y(e)=\bar e$ for every $e \in \Eout(t_0)$ and this vertex is called the \emph{target} of $A$.
\end{rmk}

\begin{lem}\label{lem:source-contains-first-edge}
Let $A$ be a complete arrangement representing a SAW $w$ on $G$. The part graph $\Gcal(s_0)$ of the source $s_0$ of $A$ contains the first arc of $w$. Similarly, the graph $\Gcal(t_0)$ of the target $t_0$ contains the final arc of $w$.
\end{lem}
\begin{proof}
Observe that for any $t \in V(T)$ by Construction~\ref{cons:projection} the walk $P(t)$ visits exactly the vertices in $w \cap \Vcal(t)$ in the same order as $w$. Let $t$ be a vertex of $T$ such that $w^- \in \Vcal(t)$ and let $s$ be the neighbour of $t$ on the unique $s_0$--$t$-path and let $e=st$. Then $X(e)=e$, so by \ref{itm:compatible-in} the initial vertex $w^-$ of $P(t)$ must be in $\Ecal(e)$ and thus $w^- \in \Vcal(s)$. Induction yields $w^- \in \Vcal(s_0)$. 
By \ref{itm:config-entry} the first arc of $P(s_0)$ is non-virtual and thus coincides with the first arc of $w$.
The statement about the target follows analogously.
\end{proof}

\section{A generating system of equations for configurations}\label{sec: system}

Throughout this section let $G$ be a locally finite connected graph and let $\Gamma$ be a group acting quasi-transitively on $G$. We always assume that $\Gamma$ does not fix an end of $G$. Corollary \ref{cor:reducedtd} provides us with a $\Gamma$-invariant tree decomposition $(\Tcal,\Vcal)$ of $G$ such that there are exactly two $\Gamma$-orbits on the arcs of $T$. More precisely, we can pick an arc $e_0$ such that every arc $e$ of $T$ either lies in the orbit $\Gamma e_0$ or in the orbit $\Gamma \bar e_0$. We denote by $\rho(e) \in \{e_0, \bar e_0\}$ the respective representative of the arc $e$. 

Two non-boring configurations $c=(q,x,y)$ and $c'=(q',x',y')$ on the arcs $e$ and $e'$ of $T$ are \emph{$\Gamma$-equivalent} if $(q',x',y')=(\gamma(q), \gamma(x), \gamma(y))$ holds for some $\gamma \in \Gamma$, where $\gamma(q)$ refers to the natural extension of $\gamma$ to the virtual edges. Clearly this defines an equivalence relation on the set of configurations on $E(T)$. Let $\Ccal$ be a set of representatives chosen in a way such that each $c \in \Ccal$ is a non-boring configuration on the arc $e_0$. Observe that the set $\Ccal$ has finite cardinality because the adhesion graph $\Gcal(e_0)$ is finite and thus carries only finitely many different configurations. For any non-boring configuration $c$ we denote its representative in $\Ccal$ by $\rho(c)$.

Let $c=(q,x,y)$ be a non-boring configuration. A \emph{$c$-completion} is an arrangement $A=(P,C)$ on a finite open sub-tree $S$ of $K_{x}$ containing $x$ such that $C(x)=c$ and $C(e)$ is boring on $e \in E(S)$ if and only if $e \in \partial E(S) \setminus \{x\}$. We define the generating function of $c$-completions as 
\[
F_{c}(z)=\sum_{\text{$A$ $c$-completion}} z^{\norm{A}},
\]
Let $r_c$ denote the radius of convergence of the generating function $F_c(z)$ and let $R=\min_{c \in \Ccal} r_c$. 

The main goal in this section is to find a system of equations for the generating functions of $c$-completions and study this system to obtain properties of these generating functions.

For an arrangement $A$ on an open subtree $S$ of $T$ and a vertex $s$ of $S$, we denote by $\nb(A,s)$ the set of arcs $e \in \Eout(s)$ such that $A$ is non-boring on $e$. Observe that by definition $\nb(A,s)$ is a finite set.

Let $c=(q,x,y)$ be a configuration and let $A=(P,C)$ be a $c$-completion on the open cone $K_x$. Then we can decompose $A$ into its restriction $A_S$ on the open star $S=\sta(x^-)$ and $C(e)$-completions $A_e$ on the cones $K_{\bar{e}}$ for all outgoing arcs $e$ of $S$ different from $x$. Observe that Lemma \ref{lem:boringextension} implies that whenever $C(e)$ is boring, the arrangement $A_e$ is uniquely given by $C(e)$ and does not contain any non-virtual arcs. Thus it is possible to recursively build all $c$-completions by completing arrangements $A$ on $\sta(x^-)$ satisfying $C(x)=c$ with $C(e)$-completions for all $e \in \Eout(x^-)\setminus \{x\}$ carrying non-boring configurations $C(e)$. Thus we obtain
\begin{equation}\label{eq:defPc}
F_{c}(z) = \sum_{\substack{A \text{ arr. on } \sta(x^-): \\ C(x)=c}} z^{\norm{A}} \prod_{e \in \nb(A,x^-) \setminus \{x\}} F_{C(e)}(z).
\end{equation}

It follows from the definitions that $F_c(z)=F_{c'}(z)$ whenever $c$ and $c'$ are $\Gamma$-equivalent configurations. Write $\Fbf(z)=(F_{c}(z))_{c \in \Ccal}$. Then the equations above can be rewritten by replacing every $F_c(z)$ by $F_{\rho(c)}(z)$ to obtain 
\begin{equation}
F_{c}(z)=P_{c}(z,\Fbf(z)), \quad c \in \Ccal. \label{eq:soe}
\end{equation}
Here
\begin{equation*}
    P_c(z,\mathbf y)= \sum_{\mathbf{n} \in \N_0^{\abs{\Ccal}}} a_{c,\mathbf{n}}(z) \mathbf{y}^\mathbf{n}, \quad c \in \Ccal,
\end{equation*}
is a formal power series in the variables $\mathbf y=(y_c)_{c \in \Ccal}$ whose coefficients are formal power series $a_{c,\mathbf n}(z)$ in the variable $z$ given by the sums in equation \eqref{eq:defPc}, respectively, where 
\[
\mathbf{y}^\mathbf{n}=\prod_{c \in \Ccal} y_c^{n_c}
\]
for an exponent vector $\mathbf{n}=(n_c)_{c \in \Ccal}$.

Let $c \in \Ccal$ be a non-boring configuration. By Lemma~\ref{lem:non-virtual-edges-in-leaves} each $c$-completion contains at least one non-virtual arc, so that $F_c(0)=0$. Furthermore we may assume without loss of generality that $F_c(z)$ is not identically zero, that is, that there exists at least one $c$-completion. If this is not the case for some configuration $c \in \Ccal$, we remove $c$ from $\Ccal$ and replace every occurrence of $y_c$ in $P_c(z,\ybf)$ by $0$. Observe that after this process equations \eqref{eq:soe} remain valid. 

Consider the Jacobian matrix 
\[
\Jac(z,\ybf)=\left(\frac{ \partial P_{c}}{\partial y_{c'}}(z,\ybf)\right)_{c,c' \in \Ccal}.
\]

Let $c=(q,x,y)$ and $c'$ be two non-boring configurations. A \emph{$c$--$c'$-completion} is an arrangement $A=(P,C)$ on a finite open subtree $S$ of $K_x$ containing $x$ such that $C(x)=c$ and there is an arc $f \in \partial E(S)$ such that $\rho(C(f))=c'$ and $C(e)$ is boring on $e \in E(S)$ if and only if $e \in \partial E(S) \setminus \{x,f\}$. The arc $x$ is the \emph{source arc} of the $c$--$c'$-completion $A$ and $f$ is the \emph{target arc}. The length of $A$ is the {\lidi} between the source arc $x$ and the target arc $f$. We denote by $\Acal(c,c',n)$ the set of $c$--$c'$-completions of length $n$.

Let us now take a closer look at the entries of $\Jac(z,\Fbf(z))$. To simplify notation, for any expression $\phi(z,\ybf)$ depending on the variable $z$ and the vector $\ybf$, we denote its evaluation at $\ybf=\Fbf(z)$ by
\[
\phi(z)=\phi(z,\Fbf(z)).
\]

\begin{lem}\label{lem:jacobian-completions}
Let $c,c'$ be non-boring configurations. Then
\[
(\Jac(z))_{c,c'}= \sum_{ A \in \Acal(c,c',1)} z^{\norm{A}}.
\]
\end{lem}
\begin{proof}
Let $c=(q,x,y)$. From equations \eqref{eq:defPc} and \eqref{eq:soe} we obtain
\[
    P_c(z,\ybf)= \sum_{\substack{A \text{ arr. on } \sta(x^-): \\ C(x)=c}} z^{\norm{A}} \prod_{e \in \nb(A,x^-) \setminus \{x\}} y_{\rho(C(e))}.
\]
An application of the Leibniz rule provides
\[
    \frac{ \partial P_{c}}{\partial y_{c'}}(z,\ybf) =\sum_{\substack{A \text{ arr. on } \sta(x^-): \\ C(x)=c}} z^{\norm{A}} \sum_{\substack{f \in \nb(A,x^-) \setminus \{x\}: \\ \rho(C(f))=c'}} \quad \prod_{e \in \nb(A,x^-) \setminus \{x,f\}} y_{\rho(C(e))}
\]
By definition every $c$--$c'$-completion $A$ of length $1$ consists of an arrangement $A_s=(P_s,C_s)$ on $\sta(x^-)$ such that there is an arc $f \in \Eout(x^-) \setminus \{x\}$ with $\rho(C_s(f))=c'$ and $C_s(e)$-completions $A_e$ for all $e \in \nb(A,x^-) \setminus \{x,f\}$. Conversely, every such $A_s$ together with some $C_s(e)$-completions $A_e$ forms a $c$--$c'$-completion. Finally note that in this case 
\[
\norm{A}=\norm{A_s} + \sum_{e \in \nb(A,x^-) \setminus \{x,f\}} \norm{A_e}.
\]  
We conclude
\[
    (\Jac(z))_{c,c'}= \sum_{ A \in \Acal(c,c',1)} z^{\norm{A_s}} \prod_{e \in \nb(A_s,x^-) \setminus \{x,f\}} F_{\rho(C_s(e))} = \sum_{ A \in \Acal(c,c',1)} z^{\norm{A}}.
\]
    
\end{proof}

Observe that a $c$--$c''$-completion $A_1$ and a $c''$--$c'$-completion $A_2$ can be concatenated to obtain a $c$--$c'$-completion whose length is the sum of the lengths of $A_1$ and $A_2$. In particular, the matrix counting $c$--$c'$-completions of length $n$ can be obtained as the $n$-th power of the matrix of $c$--$c'$-completions of length $1$.

\begin{cor}\label{cor:jacobian-completions}
Let $c,c'$ be non-boring configurations and denote $\Acal(c,c',n)$ the set of $c$--$c'$-completions of length $n$. Then
\[
(\Jac(z)^n)_{c,c'}= \sum_{ A \in \Acal(c,c',n)} z^{\norm{A}}.
\]
\end{cor}

The \emph{dependency digraph} $D$ of the system of equations \eqref{eq:soe} is given as follows. Its vertex set is the set $\Ccal$ of representatives of configurations and there is an arc from $c=(q,x,y)$ to $c'$ if $\frac{\partial P_{c}}{\partial y_{c'}}(z)$ is non-zero. In this case we also write $c \rightarrow c'$. Recall that we removed all $c$ such that $F_c(z)$ is identically zero from $\Ccal$. Therefore we have that $c \rightarrow c'$ if and only if $y_{c'}$ occurs in the power series $P_{c}$, that is, if $a_{c',\mathbf{n}}(z) > 0$ for some $\mathbf{n}=(n_c)_{c \in \Ccal}$ with $n_{c'}>0$. By definition this is the case if and only if there is a $c$-completion $A=(P,C)$ such that $\rho(C(f))=c'$ holds for some $f \in \Eout(x^-) \setminus \{x\}$. 

The following lemma follows readily from the definitions. 

\begin{lem} \label{lem:strongcompradius}
If $c \rightarrow c'$ for two configuration $c,c' \in \Ccal$, then $r_{c} \leq r_{c'}$.
\end{lem}

We further classify configurations as follows. A configuration $c=(q,x,y)$ is called an \emph{I-configuration} if $x \neq y$ and it is called a \emph{U-configuration} if $x=y$. We denote the respective subsets of $\Ccal$ by $\Ical$ and $\Ucal$, so that $\Ccal=\Ical \cup \Ucal$ is a disjoint union. We call $c \in \Ical$ \emph{simple}, if $q$ consists only of a single vertex and denote by $\Ical_s$ the set of simple configurations. Furthermore, we call $c \in \Ical$ \emph{persistent}, if there are two simple configurations $c_1$ and $c_2$ such that $c$ lies on a walk from $c_1$ to $c_2$ in $D$. Note that this is equivalent to the existence of a $c_1$--$c_2$-completion such that there is some arc $e$ of the shortest walk from the source arc to the target arc for which $C(e)$ is equivalent to $c$. Non-persistent I-configurations are called \emph{transient} and we denote by $\Ical_p$ and $\Ical_t$ the sets of persistent and transient configurations, respectively. Note in particular that all simple configurations are persistent.

\begin{lem} \label{lem:simplecomponent}
Let $e, f \in E(T)$. Then there is an arc $f' \in \Gamma f$ such that for every pair of simple configurations $c=(q,x,y) , c'=(q',x',y')$ with entry directions $x=e$ and $x'=\bar f$ there is a $c$--$c'$-completion with target arc $f'$.
\end{lem}
\begin{proof}
Let $N$ be the constant defined in Proposition~\ref{pro:adhesion-paths}. Proposition~\ref{pro:edges-dense} provides an automorphism $\gamma \in \Gamma$ such that $e$ and $f'=\gamma(f)$ are linkable and have {\lidi} at least $N$. Let $u$ be the single vertex of $q$ and $v$ be the single vertex of $\gamma(q')$. Then Proposition~\ref{pro:adhesion-paths} provides us with a SAW $w$ connecting $u$ and $v$ and meeting $\Vcal(e)$ and $\Vcal(f')$ only in these vertices. 

Theorem \ref{thm:shape-to-arrangement} provides us with a representation $A=(P,C)$ of the SAW $w$ on its support $S$ such that $X(e)=e$ and $Y(f')=f'$. It is easy to check that this representation is a $c$--$c'$-completion with target arc $f'$.
\end{proof}

\begin{rmk}
\label{rmk:rleq1}
The previous lemma immediately implies that $r_c \leq 1$ holds for the radius of convergence $r_c$ of $F_c$ for every simple configuration $c \in \Ical_s$. Indeed, the walk represented by a $c$-$c'$-completion is also represented by a $c$-completion, thus $F_c(z)$ contains infinitely many non-zero coefficients.
\end{rmk}

\begin{lem} \label{lem:D-components}
The dependency-digraph $D$ satisfies the following conditions.
\begin{enumerate}[label=(\roman*)]
    \item There are no arcs from U-configurations to I-configurations. \label{itm:no-U-I}
    \item Each strong component is contained in one of the sets $\Ical_p$, $\Ical_t$ and $\Ucal$. We call components persistent, transient or U-components depending on the type of configurations they contain. \label{itm:components-Ip-It-U}
    \item The set $\Ical_p$ is a strong component of $D$. \label{itm:Ip-strong}
\end{enumerate}
\end{lem}
\begin{proof}
Let $c=(q,x,y) \in \Ccal$ be a U-configuration and let $A=(P,C)$ be an arrangement on $\sta(x^-)$ such that $C(x)=c$. As $c$ is a U-configuration, $x=y \in \Eout(x^-)$, so by \ref{itm:config-entry} and \ref{itm:config-exit} in the definition of arrangements $X(f)=Y(f) = \bar f$ holds for every $f \in \Eout(x^-) \setminus \{x\}$, thus $C(f)$ is a U-configuration. In particular $D$ cannot contain arcs from U-configurations to I-configurations.

Lemma \ref{lem:simplecomponent} shows that all simple configurations are contained in the same strong component of $D$. Furthermore, persistent configurations are defined to lie on paths between simple configurations, so all of them are contained in this component, showing that $\Ical_p$ is contained in a single strong component.

Now let $\Kcal$ be an arbitrary strong component in $D$. If $\Kcal$ contains a U-configuration, then \ref{itm:no-U-I} implies that $\Kcal$ cannot contain any I-configuration, so $\Kcal \subseteq \Ucal$. 

Next assume that $\Kcal$ contains a transient configuration. Recall that by definition a transient configuration $c$ cannot lie on a $c_1$--$c_2$-path in $D$ for two simple configurations $c_1,c_2$. Hence, if $\Kcal$ contains a transient configuration, it cannot contain a simple configuration, so in this case $\Kcal \subseteq \Ical_t$.

Finally, if $\Kcal$ contains neither U-configurations nor transient configurations, then it is contained in $\Ical_p$. Since all of $\Ical_p$ is contained in a single strong component we conclude that in this case $\Kcal = \Ical_p$, in particular $\Ical_p$ is a strong component.
\end{proof}

The submatrix obtained from $\Jac(z)$ by restricting the index set to some subset $\Kcal \subset \Ccal$ is denoted by $\Jac_\Kcal(z)$.
It follows directly from Lemma~\ref{lem:strongcompradius} that if $\Kcal$ is a strong component of $D$, the radii of convergence $r_c$ coincide for all $c \in \Kcal$. We denote the common value by $r_\Kcal$. Let $c \in \Kcal$ and let $c' \in \Ccal$ such that $c \rightarrow c'$. Then Lemma~\ref{lem:strongcompradius} implies that $F_{c'}(z)$ is well-defined and finite for $0\leq z < r_\Kcal$. Hence $\Jac_\Kcal(z)$ is well defined for $0\leq z < r_\Kcal$.
Denote the spectral radius of the matrix $\Jac_\Kcal(z)$ by $\lambda_\Kcal(z)$.

Let us now state some properties of $\Jac_\Kcal(z)$ and $\lambda_\Kcal(z)$ which will be important later on.

\begin{lem}
\label{lem:jac_properties}
For each strong component $\Kcal$ of $D$ the following statements hold.
\begin{enumerate}[label=(\roman*)]
    \item All entries of $\Jac_\Kcal(z)$ are continuous, non-negative, and non-decreasing in $[0, r_\Kcal)$. 
    \item The spectral radius $\lambda_\Kcal(z)$ is continuous and non-decreasing in $[0, r_\Kcal)$. Additionally, $\lambda_{\Ical_p}(z)$ is strictly increasing in $[0, r_{\Ical_p})$
    \item $\Jac_\Kcal(0)$ is nilpotent, in particular $\lambda_\Kcal(0)=0$.
\end{enumerate}
\end{lem}

\begin{proof}
By Lemma \ref{lem:jacobian-completions}, the entries of the matrix $\Jac_{\Kcal}(z)$ are generating functions counting $c$--$c'$-completions of length 1 for $c,c' \in \Kcal$. In particular they are continuous, non-negative and non-decreasing within their radius of convergence and the same holds for the spectral radius $\lambda_\Kcal(z)$. 

By Proposition~\ref{pro:adhesion-paths}, there are two persistent configurations $c, c'$ such that  $(\Jac_{\Ical_p}(z))_{c,c'}$ is non-constant and thus strictly increasing in $z$. Irreducibility of $\Jac_{\Ical_p}(z)$ implies that also $\lambda_{\Ical_p}(z)$ is strictly increasing in $z$.

To see that (iii) holds, note that by Corollary~\ref{cor:finitesubtree} for $n$ large enough, any arcs $e,f \in E(T)$ of {\lidi} at least $n$ satisfy $\Vcal(e) \cap \Vcal(f) = \emptyset$. By Corollary \ref{cor:jacobian-completions} the entries of $\Jac_{\Kcal}(z)^n$ are generating functions counting $c$--$c'$-completions $A$ of length $n$. But the walk represented by such an arrangement $A$ connects $\Vcal(e)$ and $\Vcal(f)$ and thus must be non-trivial. In particular the respective entry of $\Jac_{\Kcal}(z)^n$ has no constant term, so $\Jac_{\Kcal}(0)^n=0$.
\end{proof}

Our aim now is to show that $\Jac_\Kcal(z)$ has finite entries at $z=r_{\mathcal{K}}$ for every strong component $\Kcal$.

\begin{lem}\label{lem:spectral}
For each strong component $\Kcal$ of $D$, the matrix $\Jac_\Kcal(r_{\mathcal{K}})$ is well-defined and has finite entries. Moreover, $\lambda_{\mathcal{K}}(z)< 1$ for every $0\leq z<r_{\mathcal{K}}$ and $\lambda_{\mathcal{K}}(r_{\mathcal{K}})\leq 1$.
\end{lem}

In order to prove this lemma, we first need the following result.

\begin{lem}\label{bounded entries}
Let $A=(a_{ij})\in \mathbb{R}^{n\times n}$ be an irreducible matrix such that $a_{ij}\geq 0$ for every $i,j\in\{1,2,\ldots,n\}$ and $\lambda(A)\leq 1$. Let also $M=\max\{a_{ij} \}$ and $m=\min\{a_{ij} \mid a_{ij}>0\}$. Then $M\leq m^{-n}$.
\end{lem}
\begin{proof}
Consider some $i_0,j_0$ such that $a_{i_0j_0}=M$. By the irreducibility of $A$, there are $l\leq n$ and indices $z_1=j_0,z_2,\ldots,z_l=i_0$ with $a_{z_iz_{i+1}}>0$ for every $i=1,2,\ldots,l-1$. The entry $a^{(l)}_{i_0i_0}$ of the matrix $A^l$ satisfies
\[a^{(l)}_{i_0i_0}\geq a_{i_0j_0}\prod_{i=1}^{l-1} a_{z_i z_{i+1}}\geq Mm^{l-1}.\] 
Since $A^l$ has non-negative entries and $\lambda(A^l)=\lambda(A)^l\leq 1$, we have by the monotonicity of the spectral radius that
\[1 \geq \lambda(A^l)\geq \lambda(B)=Mm^{l-1},\]
where $B$ is the matrix with $b_{i_0i_0}=Mm^{l-1}$, and all other entries 0. 
In particular we see that $m \leq 1$, so $M \leq m^{-l+1} \leq m^{-n}$.
\end{proof}

For each strong component $\mathcal{K}$ of $D$, let $F_{\mathcal{K}}(z)=(F_c(z))_{c\in\mathcal{K}}$. Additionally for $\ybf=(y_c)_{c \in \Kcal}$ let $P_{c',\Kcal}(z,\ybf)$ be obtained from $P_{c'}(z,(y_c)_{c \in \Ccal})$ by substituting $y_c=F_c(z)$ for each $c \in \Ccal \setminus \Kcal$. 


We are now ready to prove Lemma~\ref{lem:spectral}.

\begin{proof}[Proof of Lemma~\ref{lem:spectral}]
Let us start by showing that 
\begin{equation}\label{leq1}
    \lambda_{\mathcal{K}}(z)< 1 \text{ for every } 0\leq z< r_{\mathcal{K}}.
\end{equation} 
We will argue as in \cite[Proposition 4]{ceccherini-woess}.

Assume for a contradiction that $\lambda_{\mathcal{K}}(z_0)\geq 1$ for some $0\leq z_0< r_{\mathcal{K}}$. Since $\lambda_{\mathcal{K}}(0)=0$ and $\lambda_{\mathcal{K}}(z)$ is continuous, there is some $0<s\leq z_0<r_{\mathcal{K}}$ such that $\lambda_{\mathcal{K}}(s)=1$. By the Perron-Frobenius Theorem for non-negative irreducible matrices,  $\lambda_{\mathcal{K}}(s)=1$ is a positive simple eigenvalue of $\Jac_{\mathcal{K}}(s)$ for which we can find a a left eigenvector $\mathbf{x}=(x_c)_{c\in \mathcal{K}}$ with only positive entries.
Define the function
$$\mathcal{F}(z,\mathbf{y})=\sum_{c\in \mathcal{K}}x_c\left(y_c-P_{c,\Kcal}(z,\mathbf{y})\right),$$
where $\mathbf{y}=(y_c)_{c\in\mathcal{K}}$.
Note that for $\zeta=\left(s,F_{\mathcal{K}}(s)\right)$ there is some $\beta \geq 0$ such that
\begin{equation*}
\begin{gathered}
\mathcal{F}(\zeta)=0, 
\quad 
\frac{\partial\mathcal{F}}{\partial z}(\zeta)=-\beta,
\quad
\frac{\partial\mathcal{F}}{\partial y_{c}}(\zeta)=0.
\end{gathered}
\end{equation*}
We claim that $\beta$ is strictly positive, that is, that there is some $c \in \Kcal$ such that $P_{c,\Kcal}(z,\mathbf{y})$ depends non-trivially on $z$. Recall that by Lemma~\ref{lem:non-virtual-edges-in-leaves} every $c'$-completion contains a non-virtual arc, and thus $F_{c'}(z)$ depends on $z$ for every $c'$. If the strong component $\Kcal$ of $D$ has an outgoing arc $c c'$, then $P_{c,\Kcal}$ includes $F_{c'}(z)$, which depends on $z$. Otherwise for $c \in \Kcal$, the function  $P_{c,\Kcal}(z,\mathbf{y})=P_{c}(z,\mathbf{y})$ does not depend on any $y_{c'}$ for $c' \in \Ccal \setminus \Kcal$ and thus must depend on $z$, because $F_c(z)=P_{c}(z,\mathbf{F}(z))$ does.

Taylor expanding around $\zeta$, replacing $\ybf$ by $F_{\mathcal{K}}(z)$, and using that $F_c(z)=P_{c,\Kcal}(z,F_{\mathcal{K}}(z))$ we obtain
\begin{equation*}
\begin{gathered}
    \beta(s-z)=O\left((s-z)^2\right),
\end{gathered}    
\end{equation*}
which is a contradiction. This proves \eqref{leq1}.

We will now show that $\Jac_{\Kcal}(r_{\Kcal})$ has finite entries. Let us write $m(z)$ for the smallest positive entry of $\Jac_{\Kcal}(z)$, and consider some $0<z_1<r_{\mathcal{K}}$.
Then $\Jac_{c,c'}(z) \leq m(z)^{-|\mathcal{K}|}\leq m(z_1)^{-|\mathcal{K}|}$ for every $c,c'\in \mathcal{K}$ and every $z_1\leq z<r_{\mathcal{K}}$ by Lemma~\ref{bounded entries} and the monotonicity of $m(z)$. It follows from the monotone convergence theorem that $\Jac_{c,c'}(r_{\mathcal{K}})$ is well-defined and $\Jac_{c,c'}(r_{\mathcal{K}})\leq m(z_1)^{-|\mathcal{K}|}$ for every $c,c'\in \mathcal{K}$, which proves the first assertion of the lemma.

For the second assertion, it remains to show that $\lambda_{\mathcal{K}}(r_{\mathcal{K}})\leq 1$. This follows from \eqref{leq1} by taking the limit as $z$ goes to $r_{\mathcal{K}}$.
\end{proof}

\section{Analyticity at the critical point}\label{sec: analytic}

Keeping all definitions and assumptions from the previous section, the goal in this section is to prove the following analyticity results; recall that $R = \min_{c \in \Ccal} r_c$ is the minimal radius of convergence of any $F_c(z)$.

\begin{pro}\label{pro:P-analytic}
For every $c\in \Ucal$, $P_c(z,\ybf)$ is analytic at $(z,\ybf)=(R,\Fbf(R))$.
\end{pro}

\begin{pro}\label{pro:J-analytic}
We have that $\Jac_{\Ical}(z,\ybf)$ is analytic at $(z,\ybf)=(R,\Fbf(R))$.
\end{pro}

We will prove Propositions~\ref{pro:P-analytic} and \ref{pro:J-analytic} in a series of lemmas. We start by showing that $F_c(R)$ is finite for every $c\in \Ucal$.

\begin{lem}\label{lem:u-finite}
If $T$ has more than two ends and $\Gamma$ does not fix an end of $T$, then $F_c(R)$ is well-defined and finite for every $c\in \Ucal$.
\end{lem}
\begin{proof}
First note that every $c \in \Ucal$ is by definition non-boring and thus contains a virtual arc in $\Ecal(\bar x)$. If $c$ contains $k \geq 2$ virtual arcs $e_1, \dots e_k \in \Ecal(\bar x)$, then $F_c(z) \leq \prod_{i=1}^k F_{c_i}(z)$, where $c_i=(q_i,x,y)$ is the configuration such that $q_i$ is the walk of length 1 with the single arc $e_i$. 
It thus suffices to prove the assertion for every $c=(q,x,y)\in \Ucal$ that contains exactly two vertices and a virtual arc in $\Ecal(\bar x)$ connecting those vertices. Let $c$ be such a configuration.

    We claim that there are an open subtree $S$ of $T$, an arrangement $A=(P,C)$ on $S$ and three arcs $f,e_1,e_2 \in \partial E(S)$ satisfying the following conditions:
    \begin{enumerate}[label=(\roman*)]
    \itemsep0em 
        \item $C(f)= c$ and $X(f)=\bar f$. \label{itm:u-finite-1}
        \item $C(e_1)$ and $C(e_2)$ are simple I-configurations such that $X(e_1)=e_1$ and $Y(e_2)=e_2$.
        \item $C(e')$ is boring if and only if $e' \in \partial E(S) \setminus \{f,e_1,e_2\}$. \label{itm:u-finite-3}
    \end{enumerate}
    If such an arrangement $A$ exists, then $A$ together with a $c$-completion $A'$ forms a $C(e_1)$--$C(e_2)$-completion, whose length $n$ is the {\lidi} of $e_1$ and $e_2$. In particular
    \[
    (\Jac_p(R)^n)_{C(e_1),C(e_2)} \geq R^{\norm{A}} F_c(R),
    \]
    so Lemma~\ref{lem:spectral} implies that $F_c(R)$ is finite.
    
    For the proof of the claim we start by picking $f=\bar x$. Furthermore let $e_1, e_2$ be two different arcs having the same {\lidi} to $f$. We pick $e_1$ and $e_2$ such that the sets $\Vcal(f)$, $\Vcal(e_1)$ and $\Vcal(e_2)$ are pairwise disjoint.
    
    Let $q=(u,e,v) \in \Vcal(f)$. As in the proof of Proposition~\ref{pro:adhesion-paths} we can find two disjoint $\Vcal(f)$--$\Vcal(e_2)$-paths $\pi_{u}$ and $\pi_{v}$ in $G$, starting at $u$ and $v$, respectively. Moreover, we find a $\Vcal(f)$--$\Vcal(e_1)$-path $\pi$ in $G$ starting at $u$. Let $u'$ be the endpoint of $\pi$ in $\Vcal(e_1)$. We construct a walk $w$ on $\Gcal(K_f)$ as follows. Start at $u'$ and follow $\pi$ until its first intersection with one of the paths $\pi_{u}$ and $\pi_{v}$. If we reached $\pi_u$, we follow it until $u$, add the arc $e$ to reach $v$ and follow $\pi_v$ from $v$ to $\Vcal(e_2)$. Otherwise we follow $\pi_v$ until $v$, add $\bar e$ to reach $u$ and follow $\pi_u$ to reach $\Vcal(e_2)$. In this case we reverse the obtained walk and exchange $e_1$ and $e_2$. It is now easy to check that in both cases we end up with a SAW $w$ on $\Gcal(K_f)$ connecting $\Vcal(e_1)$ and $\Vcal(e_2)$ and containing the arc $e$. Let $S$ be the support of $w$. By construction the arcs $e_1,e_2$ and $f$ are in $\partial E(S)$. Theorem~\ref{thm:shape-to-arrangement} provides us with a representation $A=(P,C)$ of $w$ such that $X(e_1)=e_1$ and $Y(e_2)=e_2$. It is easy to check that $A$ satisfies conditions \ref{itm:u-finite-1}--\ref{itm:u-finite-3}.
\end{proof}

The following construction is essential in the upcoming proofs; it is sketched in Figure~\ref{fig:refl-ext}.

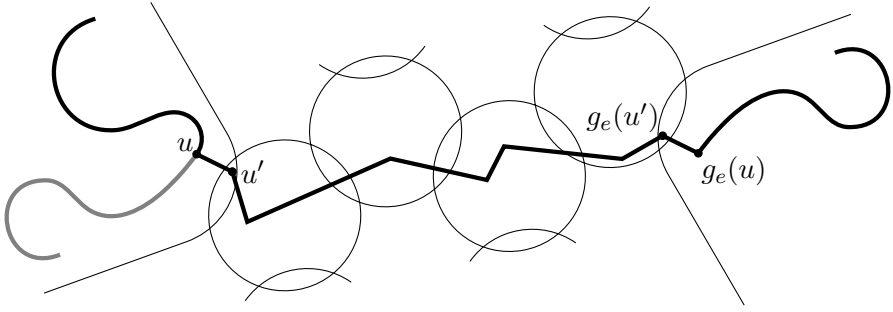
\begin{figure}
    \centering
    \begin{tikzpicture}[use Hobby shortcut,vertex/.style={inner sep=1pt,circle,draw,fill},rotate=10]
        \draw (0,0) 
        circle(1cm)
        ++(0:1cm)
        ++(60:1cm)
        circle(1cm)
        ++(0:1cm)
        ++(-60:1cm)
        circle(1cm)
        ++(0:1cm)
        ++(60:1cm)
        circle(1cm);

        \draw (0,0)
        ++(0:1cm)
        ++(60:1cm)
        ++(60:1cm)
        ++(120:1cm)
        ++(-45:1) arc (-45:-135:1) ++(-135:-1)
        ++(0:1cm)
        ++(60:1cm)
        ++(0:1cm)
        ++(-60:1cm)
        ++(-45:1) arc (-45:-135:1) ++(-135:-1);

        \draw (0,0)
        ++(-60:1cm)
        ++(-120:1cm)
        ++(45:1) arc (45:135:1) ++(135:-1)
        ++(0:1cm)
        ++(60:1cm)
        ++(0:1cm)
        ++(-60:1cm)
        ++(45:1) arc (45:135:1) ++(135:-1);

        \draw (0,0)
        ++(180:1cm)
        ++(120:1cm)
        ++(-80:1)
        ++(190:2)
        -- ++(10:2)
        arc (-80:20:1)
        -- ++(110:2);

        \draw (0,0)
        ++(0:3cm)
        ++(60:2cm)
        ++(-60:1cm)
        
        ++(180:-1cm)
        ++(120:-1cm)
        ++(-80:-1)
        ++(190:-2)
        -- ++(10:-2)
        arc (-80:20:-1)
        -- ++(110:-2);

        \draw[ultra thick]
        (-2,3)..(-2,1.5)..(-1,1.5)..(-1,1);
        \draw[gray, ultra thick]
        (-1,1)..(-1.5,0.6)..(-2.5,0.5)..(-3,1)..(-3,0);

        \draw[ultra thick]
        (-1,1)
        --(130:.9) 
        --(-.5,0)
        --(1.5,.5)
        --(2.7,0)
        --(3,.4)
        --++($(0,-.4) + (0:1) + (60:1) + (90:-.9)$)
        --++($(90:.9) + (-50:.9)$)
        --++($(130:.9) -(-1,1)$);
        \begin{scope}[shift = ($(0:1)+(60:1)$), xshift=3cm,rotate=180]
        \draw[ultra thick]
        (-1,1)..(-1.5,0.6)..(-2.5,0.5)..(-3,1)..(-3,0);
        \node[vertex] at (-1,1){};
        \node[inner sep = 1pt, anchor=north west] at (-1,1){$g_e(u)$};
        \node[inner sep = 1pt, anchor=south east] at (130:.9){$g_e(u')$};
        \node[vertex] at (130:.9) {};
        \end{scope}

        \node[vertex] at (-1,1){};
        \node[inner sep = 1pt, anchor=south east] at (-1,1){$u$};
        \node[vertex] at (130:.9) {};
        \node[inner sep =3pt,anchor=west] at (130:.9){$u'$};
        \node[vertex] at (130:.9) {};
    \end{tikzpicture}
    \caption{Reflection-extension  with splitting point $u$. The last part of the original walk is drawn in gray, the modified walk is drawn in black. Note that the distance between $u'$ and $g_e(u')$ is bounded by an absolute constant, so the increase in length will become negligible when the length of the original walk is large.}
    \label{fig:refl-ext}
\end{figure}

\begin{cons}\label{cons:refl-ext}
Let $a \in \{e_0, \bar e_0\}$ be a representative arc of the action of $\Gamma$ on $E(T)$ and fix an odd integer $N$ which is larger than or equal to the constant of Proposition~\ref{pro:adhesion-paths}. Then Proposition~\ref{pro:edges-dense} provides us with an automorphism $\gamma_{a}$ such that $a$ and $\gamma_{a}(a)$ are linkable and have {\lidi} $N$. By Proposition~\ref{pro:adhesion-paths} for any $v \in \Vcal(a)$ there is a $v$--$\gamma_{a} v$-path $\pi_{a,v}$ in $G$ meeting $\Vcal(a) \cup \Vcal( \gamma_{a} a)$ only in its endpoints.

Moreover, for every $e \in E(T)$ fix an automorphism $h_e$ mapping $e$ to its representative $\rho(e) \in \{e_0,\bar e_0\}$. 

Let $e \in E(T)$ and $w$ be a self-avoiding walk on $\Gcal(K_e)$ not intersecting $\Vcal(e)$. A \emph{reflection-extension} of $w$ through $e$ is constructed as follows. Let $u$ be a vertex of $w$ such that $d_G(u,\Vcal(e))=d_G(w,\Vcal(e))$. We apply $g_e:=h^{-1}_{e}\circ \gamma_{h_e(\bar e)}\circ h_{e}$ to the sub-walk of $w$ starting at $u$. Next we connect $u$ to $\Vcal(e)$ with a geodesic in $G$, which by definition intersects $\Vcal(e)$ only in a single vertex $u'$. Furthermore we connect $g_e(u)$ to $\Vcal(g_e(e))$ with the image of the latter geodesic under $g_e$. Finally, we connect $u'$ to $g(u')$ with $h_e^{-1}(\pi_{h_e(\bar e),h_e(u')})$. Note that by our choice of the vertex $u$, the object we thus obtain is a SAW; we denote by $\refl_e(w)$ the set of all possible reflection-extensions of $w$ through $e$. The vertex $u$ is called \emph{splitting point} of the reflection-extension.
\end{cons}

\begin{lem} \label{lem:injective-refl}
Let $e \in E(T)$ and let $w_1$ and $w_2$ be different self-avoiding walks on $\Gcal(K_e)$. Then $\refl_e(w_1) \cap \refl_e(w_2) = \emptyset$.
\end{lem}
\begin{proof}
We show that the walk $w$ can be uniquely reconstructed from any $w' \in \refl_e(w)$ provided that $e$ is known. Note that by construction $w'$ consists of 5 parts: the sub-walk of $w$ up to the vertex $u$, followed by a geodesic $\pi_u$ from $u$ to $u' \in \Vcal(e)$, a walk from $u'$ to $g_e(u')$, the image of $\pi_u$ under $g_e$, and the image of the sub-walk of $w$ starting at $u$ under $g_e$. In particular, given $w'$ it is easy to identify $u$ as the first vertex of $w'$ such that $g_e(u)$ is also contained in $w'$. Then $w$ is obtained as the concatenation of the sub-walk of $w'$ up to $u$ and the image of the sub-walk of $w'$ starting at $g_e(u)$ under the map $g_e^{-1}$.
\end{proof}

Our next goal is to show that there are exponentially fewer SAWs that stay withing a part than SAWs that use vertices in multiple parts.  Consider some part $\Vcal(t)$ and some vertex $o\in \Vcal(t)$. We define $c_n(o,t)$ to be the number of SAWs $w$ on $G$ of length $n$ starting at $o$ visiting only vertices in $\Vcal(t)$. Let $\mu_{o,t}=\limsup c_n(o,t)^{1/n}$.

\begin{lem}\label{lem:onepart}
If $\Gamma$ does not fix an end of $T$, then $\mu_{o,t}<1/R$ for every $t\in V(T)$ and $o\in \Vcal(t)$.
\end{lem}
\begin{proof}
The statement is trivially satisfied if $\Vcal(t)$ is finite, so we henceforth assume that $\Vcal(t)$ is infinite.
Consider a SAW $w$ of length $n$ starting at $o$ and visiting only vertices in $\Vcal(t)$. We will construct several SAWs starting at $o$ that do not stay in $\Vcal(t)$ by applying Construction~\ref{cons:refl-ext} and then compare the corresponding generating functions. 

We first need to define a suitable set of adhesion sets. Let $M>0$ be the constant of Proposition~\ref{pro:distance-to-adhesion}. Let also $\Bcal(w)$ be a set of arcs of $\Eout(t)$ such that $1 \leq d_G(\Vcal(e),w) \leq M$ for every $e \in \Bcal(w)$, $d_G(\Vcal(e),\Vcal(f))>2M$ for every distinct $e, f\in \Bcal(w)$, and $\Bcal(w)$ is maximal with respect to the latter property.

We now associate to each $e\in \Bcal(w)$ a vertex $u_e=u_e(w)$ of $w$ such that $d_G(u_e,\Vcal(e))=d_G(w,\Vcal(e))\leq M$, and we denote the set thus obtained by $U$. Note that the vertices $u_e$ are distinct because $d_G(\Vcal(e),\Vcal(f))>2M$ for every distinct $e, f\in \Bcal(w)$. Using the ordering of the vertices $u_e$ coming from the ordering of the vertices of $w$, we order the elements of $\Bcal(w)$. 

We claim that $|\Bcal(w)|\geq m:=n/\Delta^{4M+2D}$, where $\Delta$ is the maximal vertex-degree in $G$ and $D$ is maximal distance (in $G$) between two vertices in the sam adhesion set. Indeed, let $v$ be a vertex of $w$, and $f \in \Eout(t)$ such that $d_G(v,\Vcal(f))\leq M$. Then 
\[
d_G\left(\Vcal(f),\bigcup_{e \in \Bcal(w)} \Vcal(e)\right)\leq 2M
\]
by the maximality of $\Acal(w)$. Hence $d_G(v,U)\leq 4M+2D$. This implies that the balls of radius $4M+2D$ around the vertices in $U$ cover $w$. Since each ball of radius $4M+2D$ has size at most $\Delta^{4M+2D}$, the claim follows.

Let $\varepsilon\in(0,1)$ be a small constant to be defined later, and consider a set $H\subseteq \Bcal(w)$ of cardinality $k=\lfloor \varepsilon m \rfloor$. We write $e_i=e_i(H)$ for the $i$th element of $H$. We will successively apply Construction~\ref{cons:refl-ext} to $w$.
More precisely, let $w_1(H) \in \refl_{e_1}(w)$ be a reflection-extension of $w$ through $e_1$ with splitting point $u_{e_1}$ and let $g=g_{e_1}$ be the automorphism applied to the second part of $w$ in the reflection-extension process.

Consider the second element $e_2\in H$, and note that $g(u_{e_2})$ minimizes $d_G(v, \Vcal(g(e_2)))$ among the vertices $v$ of $w_1(H)$, since 
\[
d_G\left(\Vcal(g(e_1)),\Vcal(g(e_2))\right)>2M.
\] 
Thus there is a reflection-extension $w_2(H) \in \refl_{g(e_2)}(w_1(H))$ of $w_1(H)$ through $g(e_2)$ with splitting point $g(u_{e_2})$. Continuing in this way, we obtain a sequence of SAWs 
$w_1(H),\ldots,w_k(H)$. Note that $w_k(H)$ has length between $n$ and 
$n+k\ell$ for some constant $\ell>0$ independent of $w$.

Varying $H$ over all possible subsets of $\Bcal(w)$ of cardinality $k$, we obtain a map $(w,H) \mapsto w_k(H)$. We claim that this map is injective. Indeed, note first that the set $H$ can be reconstructed uniquely from $w_k(H)$. The arc $e_1(H)$ is the unique arc in $\Eout(t)$ such that the configuration defined by $w_k(H)$ on $e_1(H)$ is an I-configuration. This also defines uniquely the map $g=g_{e_1}$ used in the reflection-extension process. Similarly $g(e_2(H))$ is the unique arc in $\Eout(g(t)) \setminus \{e_1(H)\}$ such that the configuration defined by $w_k(H)$ on $g(e_2(H))$ is an I-configuration. Proceeding in this way, we see that $H$ is uniquely given by $w_k(H)$. But then the claim is a consequence of Lemma~\ref{lem:injective-refl}.

We have thus constructed at least
\[
\binom{m}{k}c_n(o,t)
\] 
distinct SAWs that start at $o$ and have length between $n$ and $n+k\ell$.
Since $R \leq 1$ by Remark \ref{rmk:rleq1}, this implies that 
\[
R^{n+k\ell}\binom{m}{k} c_n(o,t) \leq \sum_{(w,H)}R^{\norm{w_k(H)}},
\]
where the sum runs over all SAWs $w$ of length $n$ starting at $o$ and all subsets $H$ of $\Bcal(w)$ of cardinality $k$. We will now estimate this sum in terms of $\Jac_p(R)$.

Each walk $w_k(H)$ is a self-avoiding walk starting at $\Vcal(t)$, and ending at $\Vcal(t')$, where $t'=(g_k \circ \dots \circ g_1)(t)$. Additionally, by construction $w_k(H)$ meets each of $e_1(H)$ and $e_k'(H):=(g_k \circ \dots \circ g_1)(e_k(H))$ only in a single vertex. Decomposing $w_k(H)$ at these two vertices, we see that it consists of a SAW $w_k^1(H)$ containing only vertices of $\Vcal(t)$, a SAW $w_k^2(H)$ connecting $\Vcal(e_1(H))$ and $\Vcal(e_k'(H))$ and a SAW $w_k^3(H)$ containing only vertices of $\Vcal(t')$. By Theorem~\ref{thm:shape-to-arrangement} there is a representation $A_H=(P_H,C_H)$ of $w_k^2(H)$ which is a $C_H(e_1(H))$--$C_H(e_k'(H))$-completion for the simple configurations $C_H(e_1(H))$ and $C_H(e_k'(H))$. By Construction~\ref{cons:refl-ext} the arcs $e_1(H)$ and $e_k'(H)$ have {\lidi} $kN+k-1$. Thus the arrangement $A_H$ is counted in some entry of $\Jac_p(R)^{k N+ k-1}$. 

Letting $c_n^t = \max_{o \in \Vcal(t)} c_n(o,t)$, we thus obtain
\[
\sum_{(w,H)}R^{\norm{w_k(H)}} \leq \norm{\Jac_p(R)}^{k N+ k-1}_1 \sum_{i=0}^n \sum_{j=0}^{n-i} c_i^t c_j^t R^{i+j}.
\]
To estimate the latter sum, let $\mu_{t}=\max_{v\in \Vcal(t)}\mu_{v,t}$.
By taking cases according to whether $\mu_t R \leq 1$ or $\mu_t R>1$ we obtain
\[
\sum_{i=0}^n \sum_{j=0}^{n-i} c_i^t c_j^t R^{i+j}=\max\{\mu_t^n R^n,1\}e^{o(n)}.
\]
This implies that 
\[
c_n(o,t) \leq \frac{\max\{\mu_t^n,1/R^n\}}{R^{k\ell}\binom{m}{k}}\norm{\Jac_p(R)}^{k(N+1)+o(n)}_1.
\]
Recall that $k = \lfloor \varepsilon m \rfloor$. Hence by Stirling's approximation we obtain
\begin{equation}\label{eq:stirling}
\binom{m}{k}=\sqrt{\frac{m}{2\pi k(m-k)}}\frac{m^m}{k^k(m-k)^{m-k}}(1+o(1))=\left(\frac{1}{\varepsilon^\varepsilon(1-\varepsilon)^{1-\varepsilon}}\right)^{m+o(m)}.    
\end{equation}

For $\varepsilon= R^{\ell}\norm{\Jac_p(R)}^{-(N+1)}_1$ we have 
\[
c_n(o,t)\leq \max\{\mu_t^n, 1/R^n\}(1-\varepsilon)^{(1-\varepsilon)m+o(m)}.
\]
Taking $n$-th roots and sending $n$ to infinity we obtain that 
$$\mu_{o,t}<\max\{\mu_t,1/R\}.$$
Since this holds for all $o$, we obtain that $\mu_t<\max\{\mu_t,1/R\}$, hence $\mu_t<1/R$, as desired.
\end{proof}

We are now ready to prove Proposition~\ref{pro:P-analytic}.

\begin{proof}[Proof of Proposition~\ref{pro:P-analytic}]
By assumption all coefficients of $P_c(z,\ybf)$ are non-negative. Thus it is enough to show that there is some constant $\delta > 0$ such that 
\[
F_{c,\delta}(R) := P_c\left((1+\delta)R,(1+\delta)\Fbf(R)\right) < \infty.
\]

As in the proof of Lemma~\ref{lem:u-finite}, it suffices to consider configurations $c\in \Ucal$ that contain only two vertices and a virtual arc connecting those vertices.

Let $c=(q,x,y)$ be such a configuration, and let $t=x^-$. Let us start by defining $\Acal(c,n)$ to be the set of arrangements $A=(P,C)$ on $S=\sta(t)$ such that $C(x)=c$ and $\abs{P(t)}=n$, where here $n$ counts the total number of arcs in $P(t)$ and not just the non-virtual ones. In particular every $A \in \Acal(c,n)$ satisfies
\begin{equation}\label{eq:nb-n}
\norm{A} + \abs{\nb(A,t) \setminus \{x\}} \leq n \leq \norm{A} + K \abs{\nb(A,t) \setminus \{x\}},
\end{equation}
where $K$ is the size of the adhesion sets of our tree decomposition.
With this in mind, we can write
\[
F_{c,\delta}(R) \leq  \sum_{n=0}^\infty (1+\delta)^n \Phi_{n}(R),
\]
where
\[
\Phi_{n}(R) = \sum_{A \in \Acal(c,n)} R^{\norm{A}} \prod_{e \in \nb(A,t) \setminus \{x\}} F_{C(e)}(R).
\]

It is enough find an upper bound for $\Phi_{n}(R)$ which decays exponentially in $n$ because then for some $\delta > 0$ small enough, $F_{c,\delta}(R)$ is upper bounded by a geometric series.
We split the inner sum into three parts and treat those parts individually. In order to split up the sum, we first need some definitions.

For $\varepsilon>0$ we partition $\Acal(c,n)$ into two subsets according to whether they result in a `small' or `large' number of non-boring configurations on the edges of $E(S)$: 
\begin{enumerate}[label=(\roman*)]
    \item $\mathcal{S}(c,\varepsilon,n)$ is the set of  $A\in \Acal(c,n)$ such that the number of arcs $e\in E(S)\setminus\{x\}$ for which $C(e)$ is non-boring, is smaller than $\varepsilon n$,
    \item $\mathcal{L}(c,\varepsilon,n)=\Acal(c,n)\setminus \mathcal{S}(c,\varepsilon,n)$.
\end{enumerate}
Clearly,
\begin{equation}
\label{eq:innersum}
  \Phi_{n}(R) = \sum_{A \in \mathcal{S}(c,\varepsilon,n)} R^{\norm{A}} \prod_{e \in \nb(A,t) \setminus \{x\}} F_{C(e)}(R)
  + \sum_{A \in \mathcal{L}(c,\varepsilon,n)} R^{\norm{A}} \prod_{e \in \nb(A,t) \setminus \{x\}} F_{C(e)}(R).  
\end{equation}

We further split the second sum into a part summing over completions where `many' of the $C(e)$-completions are longer than some constant $L$, and a part summing over completions where `few' of the $C(e)$-completions are longer than $L$. 
More precisely, for $L>0$ and $A\in \mathcal{L}(c,\varepsilon,n)$, let $\mathcal{M}(A,\varepsilon,n,L)$ be the set of sequences $(A_e)_{e\in E(S)\setminus \{x\}}$, where each $A_e=(P_e,C_e)$ is an arrangement on $K_{\bar e}$, such that $C_e(e)=C(e)$ for every $e\in E(S)\setminus \{x\}$, and for at least $\varepsilon n/2$ arcs $e\in E(S)\setminus \{x\}$, we have $\norm{A_e}\geq L$. Let $\mathcal F(A, \varepsilon,n,L)$ be the set of sequences $(A_e)_{e\in E(S)\setminus \{x\}}$ not contained in $\mathcal M(A,\varepsilon,n,L)$.
Write $\norm{(A_e)}=\sum_{e\in E(S)\setminus \{x\}}\norm{A_e}$.

Recall that
\[
    F_{C(e)}(R) = \sum_{A_e \; C(e)\text{-completion}} R^{\norm{A_e}}.
\]
Using this identity on the factors of the product inside the second sum in the expression \eqref{eq:innersum}, we arrive at
$
    \Phi_{n}(R) = \Sigma_{n}(R) + \Lambda_{n}^1(R)+ \Lambda_{n}^2(R),
$
where
\begin{align*}
\Sigma_{n}(R) & =  \sum_{A \in \mathcal{S}(c,\varepsilon,n)} R^{\norm{A}} \prod_{e \in \nb(A,t) \setminus \{x\}} F_{C(e)}(R) \\
\Lambda_{n}^1(R) & =  \sum_{A\in \mathcal{L}(c,\varepsilon,n)} R^{\norm{A}} \sum_{(A_e)\in \mathcal{M}(A,\varepsilon,n,L)} R^{\norm{(A_e)}}  \\
\Lambda_{n}^2(R) & =  \sum_{A\in \mathcal{L}(c,\varepsilon,n)} R^{\norm{A}} \sum_{(A_{e})\in \mathcal{F}(A,\varepsilon,n,L)} R^{\norm{(A_{e})}}.
\end{align*}

Let us first consider $\Sigma_n(R)$. Using the upper bound in \eqref{eq:nb-n} together with the fact that $R \leq 1$ by Remark~\ref{rmk:rleq1}, we immediately obtain
\[
\Sigma_n(R) \leq  t_1^{\varepsilon n} R^{(1-K \varepsilon) n} \abs{ \mathcal{S}(c,\varepsilon,n)} ,
\]
where $1<t_1:=1+\max\{F_c(R)\mid c\in \Ucal\}$ is a finite constant by Lemma~\ref{lem:u-finite}.
We will show that the right-hand side of the inequality decays exponentially, provided that $\varepsilon$ is small enough.

Let us estimate the cardinality of $\mathcal{S}(c,\varepsilon,n)$. Each $A=(P,C) \in \mathcal{S}(c,\varepsilon,n)$ consists of a walk $P(t)$ and a collection of configurations $(C(e))_{e \in E(S)}$ with all but at most $\varepsilon n$ of them being boring.  
Note that $P(t)$ already uniquely determines $A$ because $c$ is a U-configuration. Indeed, the configurations $C(e)=(Q(e),X(e),Y(e))$ must be U-configurations and by Lemma~\ref{lem:arrangement-from-shape} each $Q(e)$ is determined by $P(t)$.
We derive an upper bound for the number of possible walks $P(t)$. We start by counting the possible walks with exactly $k$ virtual arcs, and whose walk components after removing the virtual arcs have lengths $(n_0,\dots, n_k)$. It is possible that some $n_i=0$. Note that $\sum_{i=0}^k n_i = n-k$. Denote the set of such walks by $\Wcal(n_0,\dots, n_k)$. We can construct such walks inductively. Start by choosing a walk $W_0$ of length $n_0$ starting at the initial vertex of the walk $q$ of the configuration $c$. In the $i$-th step we attach a virtual arc $e_i$ and a walk $W_i$ of length $n_i$ to the walk constructed so far; if no suitable virtual arcs exists, we do not proceed with the construction.
There are at most $c_{n_i}(t)$ ways to choose the $i$-th walk, and there is a constant upper bound $t_2$ for the number of virtual arcs incident to a vertex. Thus 
\[|\Wcal (n_0,\dots n_k)| \leq t_2^{k}\prod_{i=0}^k c_{n_i}(t),\]
where $c_{n_i}(t)=1$ if $n_i=0$.
To bound this product, consider some $r\in (\mu_t,1/R)$, where $\mu_t=\max_{o\in \Vcal(t)} \mu_{o,t}$ Note that such an $r$ exists by Lemma~\ref{lem:onepart} and that there exists a constant $t_3\geq 1$ such that $c_m(t)\leq t_3 r^m$ for every $m\geq 1$ and every $o\in \Vcal(t)$. Moreover, every virtual arc lies in an adhesion set such that $A$ is non-boring on the corresponding arc. This implies that $k\leq K \varepsilon n$. We conclude that 

\[
|\Wcal (n_0,\dots n_k)| \leq (t_2 t_3)^{K\varepsilon n } r^{n-k}.
\]

Every $A \in \Scal(c,\varepsilon,n)$ is uniquely determined by some $W \in \Wcal(n_0,\dots,n_k)$ for some $k \leq K \varepsilon n$ where $n_0, \dots, n_k$ are given by the positions of the $k$ virtual arcs in $P(t)$. We conclude that
\[
\abs{ \mathcal{S}(c,\varepsilon,n)} \leq \sum_{k=0}^{\lfloor K\varepsilon n \rfloor} \binom{n}{k} (t_2 t_3)^{K\varepsilon n } r^{n-k}.
\]
Using Stirling's approximation as in \eqref{eq:stirling} and the fact that the binomial coefficients $\binom{n}{k}$ are monotonic in $k$, we see that 
\[
\binom{n}{k}\leq \left(\frac{1}{(K\varepsilon)^{K\varepsilon}(1-K\varepsilon)^{1-K\varepsilon}}\right)^{n+o(n)}
\]
for every $0\leq k \leq K \varepsilon n$ and every $0<\varepsilon\leq \frac{1}{2K}$. Therefore, for some positive constant $t_4$ depending on $t_1$, $t_2$, $t_3$ and $K$,
\[
\Sigma_n(R) \leq \sum_{k=0}^{\lfloor K \varepsilon n \rfloor}R^{(1-K \varepsilon) n} t_4^{\varepsilon n} r^{n-k} \left(\frac{1}{(K\varepsilon)^{K\varepsilon}(1-K\varepsilon)^{1-K\varepsilon}}\right)^{n+o(n)}.
\]

It is possible to choose $0<\varepsilon\leq \frac{1}{2K}$ to be small enough so that 
\[
\frac{2^{\varepsilon} t_4^{\varepsilon} r^{1-k/n}}{R^{K\varepsilon}(K\varepsilon)^{K\varepsilon}(1-K\varepsilon)^{1-K\varepsilon}} \leq  \frac{1}{R}
\]
holds for $n$ large enough because the left side goes to $r$ as $\varepsilon$ tends to $0$. Hence, we obtain
\[
\Sigma_n(R) \leq (K\varepsilon n +1) R^{(1-K \varepsilon) n} 2^{-\varepsilon n} R^{K\varepsilon n} R^{-n+o(n)} \leq 2^{-\varepsilon n+o(n)},
\]
as desired.

Let us next bound $\Lambda_n^1(R)$.
Our aim is to estimate $\sum_{(A_e)\in \mathcal{M}(A,\varepsilon,n,L_0)} R^{\norm{(A_e)}}$ for a certain length $L_0$. Let $p=2^{-{4/\varepsilon}}$ and note that there exist a constant $L_0>0$ such that 
\[
\sum_{\substack{A \; c'\text{-completion}\\\norm{A}\geq L_0}} R^{\norm{A}}\leq p F_{c'}(R).
\]
for every $c'\in \Ucal$.

For a fixed $H \subseteq \nb(A,t)\setminus \{x\}$ with $|H| \geq \varepsilon n/2$, we have
\[
\sum_{\substack{(A_e)\in \mathcal{M}(A,\varepsilon,n,L_0)\\ \norm{A_e}\geq L_0 \text{ for } e \in H}} R^{\norm{(A_e)}} \leq p^{\varepsilon n/2} \prod_{e \in \nb(A,t) \setminus\{x\}}F_{C(e)}(R).
\]
Each $(A_e)\in \mathcal{M}(A,\varepsilon,n,L_0)$ by definition satisfies $\norm{A_e}\geq L_0$ for at least $\varepsilon n/2$ arcs $e\in E(S)\setminus \{x\}$. Summing over all possible $H$ yields

\[
\sum_{(A_e)\in \mathcal{M}(A,\varepsilon,n,L_0)} R^{\norm{(A_e)}}\leq 
2^n p^{\varepsilon n/2} \prod_{e \in \nb(A,t) \setminus\{x\}}F_{C(e)}(R).
\]
It follows that
\begin{equation*}
\begin{split}
\Lambda_n^1(R) \leq  2^n p^{\varepsilon n/2} \sum_{A\in \mathcal{L}(c,\varepsilon,n)} R^{\norm{A}} \prod_{e \in \nb(A,t) \setminus\{x\}}F_{C(e)}(R) \leq 2^n p^{\varepsilon n/2}  F_c(R) \leq 2^{-n}F_c(R).
\end{split}
\end{equation*}

Finally we find an upper bound for $\Lambda_n^2(R)$. To this end, we employ a similar strategy as in Lemma~\ref{lem:onepart}.

For $e_0 \in E(T)$, Proposition~\ref{pro:edges-dense} yields an arc $f_0 \in E(K_{e_0})$ in the orbit of $\bar e_0$ such that $e_0$ and $f_0$ are linkable. Let $M=d_G(\Vcal(e_0),\Vcal(f_0))$ and let $M'= d_T(e_0,f_0)$. We choose $f_0$ such that that $M > L_0$. 

For a $c$-completion $A = (P,C)$, we denote by $A_S$ and $A_e$ the restrictions of $A$ to $S$ and $K_{\bar e}$, respectively. Assume that we are given a $c$-completion $A$ such that $A_S \in \mathcal{L}(c,\varepsilon,n)$ and $(A_e)_{e\in E(S)\setminus \{x\}}\in \mathcal{F}(A,\varepsilon,n,L_0)$ and let $w=\varphi(A)$ be the walk represented by $A$. Consider the set of arcs $e\in \nb(A,t) \setminus \{x\}$ such that $\norm{A_e}< L_0$. As in the proof of Lemma~\ref{lem:onepart} we can find a subset $\Bcal(w)$ of these arcs of size at least $m := \lfloor\varepsilon' n \rfloor$ for some $0<\varepsilon'<\varepsilon/2$ such that $d_G(\Vcal(e),\Vcal(f)) > 2 M+ 2 D$ for distinct $e,f \in \Bcal(w)$, where $D$ is the diameter of the adhesion sets. Indeed, at least $\varepsilon n/2$ vertices of the walk $w$ lie in adhesion sets $\Vcal(e)$ for $e\in \nb(A,t) \setminus \{x\}$ such that $\norm{A_e}< L_0$, and the $(2M+2D)$-ball around each such vertex covers only a constant number of other vertices.

Let $\varepsilon''\in (0,1)$ be a constant to be defined. 
Consider a subset $H$ of $\Bcal(w)$ of cardinality $k:=\lfloor \varepsilon'' m \rfloor$. Each $e\in H$ either lies in the orbit of $e_0$ or in the orbit of $f_0$. If there is $\gamma \in \Gamma$ such that $\gamma e_0 = e$, then we choose $f(e,w) = \gamma f_0$. Otherwise there is $\gamma \in \Gamma$ such that $\gamma f_0 = e$ and we choose $f(e,w) = \gamma e_0$.
Note that $M > L_0$ and the choice of $e$ such that $\norm{A_e} < L_0$ implies that $0 < d_G(w,\Vcal(f(e,w))) < M+D$. Choose a vertex $u_e=u_e(w)$ of $w$ such that $d_G(u_e,\Vcal(f(e,w)))=d_G(w,\Vcal(f(e,w)))$. The walk $w$ induces an order on the vertices $\{u_e\mid e \in H\}$, denote them by $u_1, \dots, u_k$. This order induces an order on $H$, denoted by $e_1,\dots,e_k$, and consequently also on the arcs $\{f(e,w)\mid e \in H\}$ which we denote by $f_1, \dots, f_k$. As in Lemma~\ref{lem:onepart}, we successively apply Construction~\ref{cons:refl-ext} to $w$. Let $w_1(H) \in \refl_{f_1}(w)$ be a reflection-extension of $w$ through $f_1$ with splitting point $u_1$ and let $g_1$ be the automorphism applied to the second part of $w$ in the reflection-extension process.

Consider the second element $f_2\in H$, and note that $g_1(u_2)$ minimizes $d_G(v, \Vcal(g_1(f_2)))$ among the vertices $v$ of $w_1(H)$, since 
\[
d_G\left(\Vcal(g_1(f_1)),\Vcal(g_1(f_2))\right)>2M+2D.
\] 
Thus there is a reflection-extension $w_2(H) \in \refl_{g_1(f_2)}(w_1(H))$ of $w_1(H)$ through $g_1(f_2)$ with splitting point $g_1(u_2)$. Continuing in this way, we obtain a sequence of SAWs $w_1(H),\ldots,w_k(H)$.

Letting $H$ vary over all possible subsets of $\Bcal(w)$ of cardinality $k$, we obtain a map $(w,H) \mapsto w_k(H)$. We claim that this map is injective. Indeed, note first that the set $H$ can be reconstructed uniquely from $w_k(H)$. The arc $e_1$ is the unique arc in $\Eout(t)$ such that the configuration defined by $w_k(H)$ on $e_1$ is an I-configuration. Then, up to orientation, $f_1$ is the unique arc of $K_{\bar e_1}$ such that $d_G(\Vcal(e_1),\Vcal(f_1))=M$ and $w_k(H)$ visits $\Vcal(f_1)$. This also defines uniquely the map $g_{1}$ used in the reflection-extension process. Similarly $g_1(e_2)$ is the unique arc in $\Eout(g_1(t)) \setminus \{e_1\}$ such that the configuration defined by $w_k(H)$ on $g_1(e_2)$ is an I-configuration. Proceeding in this way, we see that $H$ is uniquely given by $w_k(H)$. But then the claim is a consequence of Lemma~\ref{lem:injective-refl}.

Note that each time we apply the local modification, the length of the SAW we obtain increases by at most $\ell$, where $\ell>0$ is a uniform constant, hence $\abs{w_k(H)}\leq \abs{w}+k\ell$. Therefore,
\[
    \Lambda_n^2(R) = \sum_{A\in \mathcal{L}(c,\varepsilon,n)} \sum_{(A_{e})\in \mathcal{F}(A,\varepsilon,n,L)} R^{\abs{w}} \leq \frac{1}{\binom{m}{k}R^{k\ell} } \sum_{(w,H)} R^{\abs{w_k(H)}}.
\]
where the sum index $(w,H)$ ranges over the set of all walks $w$ represented by a $c$-completion $A$ such that $A_S \in \mathcal{L}(c,\varepsilon,n)$ and $(A_e)_{e\in E(S)\setminus \{x\}}\in \mathcal{F}(A,\varepsilon,n,L_0)$, and all possible choices of subsets $H$ of $\Bcal(w)$ of size $k$.

Each walk $w_k(H)$ is a self-avoiding walk starting at $\Vcal(x)$, and ending at $\Vcal(x')$, where $x'=(g_k \circ \dots \circ g_1)(x)$. Note that $g_1, \dots, g_k$, and thus also $x'$, depend on $H$. By Theorem~\ref{thm:shape-to-arrangement} there is a representation $A_H=(P_H,C_H)$ of $w_k(H)$ such that $X_H(x)=x$ and $Y_H(x')=x'$. Note also that by construction $C_H(x)$ and $C_H(x')$ are simple configurations because $w$ meets $\Vcal(x)$ only in its endpoints. Observe that the distance of $x$ and $x'$ is $2 k M' + k N'+k+1$ because for every $i=1, \dots, k$ the distance of $e_i$ and $f_i$ is $M'$ and the distance of $f_i$ and $g_i (f_i)$ is $N'$. Thus the arrangement $A_H$ is counted in some entry of $\Jac_p(R)^{k (2M' + N'+ 1) + 1}$. We obtain
\[
\sum_{(w,H)} R^{\abs{w_k(H)}}\leq \norm{\Jac_p(R)^{k (2M' + N'+ 1) + 1}}_1 \leq e^{t_5 k}
\]
for some constant $t_5$.

Using the bound \eqref{eq:stirling} for the binomial coefficient, we see that we can choose $\varepsilon''$ small enough so that
\[
\Lambda_n^2(R) \leq \frac{e^{t_5k}}{\binom{m}{k}R^{k\ell} } \leq e^{-\varepsilon'' m+o(m)},
\]
Since $m = \lfloor \varepsilon' n \rfloor$, the desired assertion follows readily.
\end{proof}

Proposition~\ref{pro:J-analytic} can be now proved in a similar way.

\begin{proof}[Proof of Proposition~\ref{pro:J-analytic}]
For $c,c'\in \Ical_p$, the analyticity of $\Jac_{c,c'}(z,\ybf)$ can be proved by arguing as in Proposition~\ref{pro:P-analytic}. Indeed, we can split $\Phi_n(R)$ into three sums $\Sigma_n(R)$, $\Lambda^1_n(R)$ and $\Lambda^2_n(R)$, where each of them can be handled in the same way. We emphasise that this is indeed the case for $\Lambda^2_n(R)$ because of our assumption that $c,c'\in \Ical_p$, which allows us to obtain an arrangement that contributes to $\Jac_p(R)^N$, for some $N>0$, after applying a successive application of our construction.

In general, for $c=(q,x,y),c'=(q',x',y')\in \Ical$, we can see that $c$ and $c'$ contain simple configurations $c_1=(q_1,x,y)$ and $c_1'=(q_1',x',y')$ and U-configurations $c_2=(q_2,x,x),c_2'=(q_2', \bar{x}',\bar{x}')$.
Then for $(z,\ybf)=((1+\delta)R,(1+\delta)\Fbf(R))$
\[\Jac_{c,c'}(z,\ybf)\leq \Jac_{c_1,c'_1}(z,\ybf) P_{c_2}(z,\ybf) P_{c_2'}(z,\ybf),\] and each term in the product is finite for some $\delta>0$ sufficiently small.
\end{proof}

\section{Spectral radius and radius of convergence}\label{sec: spectral}

Our goal in this section is to show that $R < r_\Kcal$ for every U-component $\Kcal$. Our approach loosely follows Alm and Janson \cite{AlJa90} of decomposing U-configurations into smaller U-configurations and pairs of I-configurations. In order to formalise this approach we first need a few definitions.

A \emph{pair configuration} is a pair $(c_1,c_2)$ of persistent I-configurations $c_1=(q_1,x_1,y_1), c_2=(q_2,x_2,y_2) \in \Ical_p$ such that $x_1 = x_2$; we write  $\Ical_\pair=\{(c_1,c_2) \in \Ical_p \times \Ical_p: {x_1}={x_2}\}$ for the set of all pair configurations.

Let $(c_1,c_2)$ and $(c_1',c_2')$ be pair configurations.
A $(c_1,c_2)$--$(c_1',c_2')$-completion is a pair $(A_1,A_2)$ consisting of a $c_1$--$c_1'$-completion $A_1$ and a $c_2$--$c_2'$-completion $A_2$ having the same target arc $f$. The length of $(A_1,A_2)$ is the length of $A_1$ (which equals the length of $A_2$) and we call it \emph{disjoint} if the walks $p_1$ and $p_2$ represented by $A_1$ and $A_2$ are disjoint. Denote by $\Acal(c_1,c_2,c_1',c_2',n)$ the set of all $(c_1,c_2)$--$(c_1',c_2')$-completions of length $n$ and by $\Acal_{\disj}(c_1,c_2,c_1',c_2',n)$ the subset of those which are disjoint. We define a matrix $\Jac_{\Ical_\pair}(z)$ with index set $\Ical_\pair \times \Ical_\pair$ entry-wise by
\[
(\Jac_{\Ical_\pair}(z))_{(c_1,c_2),(c_1',c_2')}=\sum_{(A_1,A_2) \in \Acal(c_1,c_2,c_1',c_2',1)} z^{\norm{A_1}+\norm{A_2}}.
\]

Like in the case of single configurations, we define a dependency digraph $D_\pair$ for pair configurations. Its vertex set is $\Ical_\pair$ and we have an arc from $(c_1,c_2)$ to $(c_1',c_2')$ if the entry at position $(c_1,c_2),(c_1',c_2')$ in the matrix $\Jac_{\Ical_\pair}(z)$ is non-zero. 

A pair configuration $(c_1,c_2)$ is called simple if both $c_1$ and $c_2$ are simple. It is called persistent, if it lies on a walk connecting two simple pair configurations in $D_\pair$. The set of persistent pair configurations is denoted $\Ical_{\pair,p}$. Note that while every persistent pair configuration consists of two persistent I-configurations, the converse is not necessarily true. This is due to the fact that a $c_1$--$c_1'$-completion $A_1$ and a $c_2$--$c_2'$-completion $A_2$ might have the same source arc but different target arcs.

Finally we define the matrix $\Jac_{\pair,\disj}(z)$ with index set $\Ical_{\pair,p} \times \Ical_{\pair,p}$ entry-wise by
\[
    (\Jac_{\pair,\disj}(z))_{(c_1,c_2),(c_1',c_2')}=\sum_{(A_1,A_2) \in \Acal_\disj(c_1,c_2,c_1',c_2',1)} z^{\norm{A_1}+\norm{A_2}}.
\]
Denote by $\lambda_{\pair,\disj}(z)$ the spectral radius of $\Jac_{\pair,\disj}(z)$.

\begin{lem}\label{lem:pair-p}
For every $z\in (0,R]$, $\Jac_{\pair,\disj}(z)$ has finite entries and furthermore, we have $\lambda_{\pair,\disj}(z)< \lambda_{\Ical_p}(z)^2$.
\end{lem}
\begin{proof}
    We start by noting that for $z \geq 0$ all matrices considered in this proof are non-negative, so whenever we restrict the index set or reduce some entries of a matrix, the spectral radius of the resulting matrix must be smaller than or equal to the spectral radius of the original matrix.
    
    Consider the Kronecker product $\Jac_{\Ical_p \times \Ical_p}(z)=\Jac_{\Ical_p}(z) \otimes \Jac_{\Ical_p}(z)$. The eigenvalues of the Kronecker product of two matrices $M_1$ and $M_2$ are precisely the products $\lambda_1\lambda_2$ of eigenvalues $\lambda_1$ of $M_1$ and $\lambda_2$ of $M_2$. Therefore the spectral radius of $\Jac_{\Ical_p \times \Ical_p}(z)$ is $\lambda_{\Ical_p}(z)^2$. By Lemma~\ref{lem:spectral}, the matrix $\Jac_{\Ical_p}(z)$ has finite entries, thus the same holds for $\Jac_{\Ical_p \times \Ical_p}(z)$.

    By definition every $(c_1,c_2)$--$(c_1',c_2')$-completion $(A_1,A_2)$ consists of a $c_1$--$c_1'$-completion $A_1$ and a $c_2$--$c_2'$-completion $A_2$, so for $z \in (0,R]$ all entries of the matrix $\Jac_{\Ical_\pair}(z)$ are smaller than or equal to the respective entries of $\Jac_{\Ical_p \times \Ical_p}(z)$. We conclude that its spectral radius $\lambda_{\Ical_\pair}(z)$ satisfies $\lambda_{\Ical_\pair}(z) \leq \lambda_{\Ical_p}(z)^2$. By Lemma \ref{lem:simplecomponent} all simple pair configurations are contained in the same strong component of $D_\pair$, so the set $\Ical_{\pair,p}$ of persistent pair configurations is a strong component of $D_\pair$. In particular for $z > 0$ the submatrix $\Jac_{\Ical_{\pair,p}}(z)$ of $\Jac_{\Ical_\pair}(z)$ obtained by restricting the index set to $\Ical_{\pair,p}$ is irreducible. Finally, the entries of the matrix $\Jac_{\pair,\disj}(z)$ are smaller than or equal to the respective entries of $\Jac_{\Ical_{\pair,p}}(z)$ with strict inequality for some entries, for instance the diagonal entries. Thus $\lambda_{\pair,\disj}(z) < \lambda_{\Ical_\pair}(z) \leq \lambda_{\Ical_p}(z)^2$ holds. 
\end{proof}

To simplify notation let $\lambda_{\Ucal}(z):=\max\{\lambda_{\Kcal}(z) \mid \Kcal \subseteq \Ucal \text{ strong component of $D$}\}$ and let $\lambda_{\Ical_t}(z):=\max\{\lambda_{\Kcal}(z) \mid \Kcal \subseteq \Ical_t \text{ strong component of $D$}\}$. 

\begin{lem}\label{lem:U-pair}
For every $z\in (0,R]$ we have $\lambda_\Ucal(z) \leq \lambda_{\pair,\disj}(z)$.
\end{lem}
\begin{proof}
Let $\Kcal \subseteq \Ucal$ be a strong component of $D$. Let $c,c'$ be two configurations in $\Kcal$ and recall that $\Acal(c,c',n)$ denotes the set of $c$--$c'$-completions of length $n$. Then
\begin{equation}\label{eq:decompu}
    (\Jac_\Kcal(z)^n)_{c,c'} = \sum_{A \in \Acal(c,c',n)} z^{\norm{A}}.
\end{equation}

By Corollary~\ref{cor:finitesubtree} there is a constant $N'$ such that $\Vcal(e) \cap \Vcal(f)=\emptyset$ whenever $d_T(e,f) \geq N'$ holds for two edges $e,f \in E(T)$.
Let $A \in \Acal(c,c',n)$ for some $n \geq N'$, let $S$ be the support of $A$ and let $\iota$ and $\tau$ be source and target arc of $A$, respectively. The arrangement $A$ represents a unique SAW $\varphi(A)$ on $\Gcal(S)$ starting and ending in $\Vcal(\iota)$. Let $v_0,v_1, \dots, v_k$ be the sequence of vertices of $\varphi(A)$ contained in at least one of $\Vcal(\iota)$ and $\Vcal(\tau)$, ordered by their appearance in $\varphi(A)$. We decompose $\varphi(A)$ as
\[
\varphi(A) = v_0w_1v_1w_2v_3\dots w_k v_k,
\]
in other words, $w_i$ is the sub-walk of $\varphi(A)$ starting at $v_{i-1}$ and ending at $v_i$. Denote by $K$ the size of an adhesion set of the tree decomposition. Then the number of sub-walks $k$ is bounded from above by $2K$. We call a walk $w_i$ virtual if it contains only a single arc and this arc is contained in $\Ecal(\iota)$ or $\Ecal(\tau)$, and non-virtual otherwise. Then each non-virtual $w_i$ belongs to one of four possible classes. 
\begin{enumerate}[label=(\roman*)]
\itemsep0em 
    \item $w_i$ starts and ends in $\Vcal(\iota)$ \label{itm:ll-walk} 
    \item $w_i$ starts and ends in $\Vcal(\tau)$ \label{itm:rr-walk}
    \item $w_i$ starts in $\Vcal(\iota)$ and ends in $\Vcal(\tau)$ \label{itm:lr-walk}
    \item $w_i$ starts in $\Vcal(\tau)$ and ends in $\Vcal(\iota)$ \label{itm:rl-walk}
\end{enumerate}
Observe that the number of walks of class \ref{itm:lr-walk} coincides with the number of walks of class \ref{itm:rl-walk} because $\varphi(A)$ starts and ends in $\Vcal(\iota)$. Hence we can group them into pairs, each consisting of a walk of class \ref{itm:lr-walk} and a walk of class \ref{itm:rl-walk}. Denote by $\Wcal_1$, $\Wcal_2$ and $\Wcal_3$ the sets of walks $w_i$ of class \ref{itm:ll-walk}, class \ref{itm:rr-walk} and the set of pairs $(w_i,w_j)$ of classes \ref{itm:lr-walk} and \ref{itm:rl-walk}, respectively. For technical reasons, pairs $(w_i,w_j)$ such that the final vertex of $w_i$ coincides with the initial vertex of $w_j$ (that is, $j=i+1$) are not added to $\Wcal_3$; instead their concatenation is included in $\Wcal_1$. Note that $\Wcal_3$ contains at least one pair of walks, because $c'$ is non-boring. 

By Theorem \ref{thm:shape-to-arrangement} each $w \in \Wcal_1$ is represented by a $c_w$-completion $A_w$, where $c_w$ is a U-configuration on $\iota$. Similarly, each $w \in \Wcal_2$ is represented by a $c_w$-completion $A_w$, where $c_w$ is a U-configuration on $\tau$. Finally, each pair $(w,w') \in \Wcal_3$ is represented by a disjoint $(c_w,c_{w'})$--$(c_{w}',c_{w'}')$-completion $(A_w,A_{w'})$ of length $n$, where $c_w$ and $c_{w'}$ are simple configurations on $\iota$ and $c_w'$ and $c_{w'}'$ are simple configurations on $\tau$. 

Note that by construction 
\begin{equation*}
    z^{\norm{A}}=\prod_{w\in\Wcal_1} z^{\norm{A_w}} \prod_{w\in\Wcal_2} z^{\norm{A_w}} \prod_{(w,w')\in \Wcal_3} z^{\norm{A_w}+\norm{A_{w'}}}.
\end{equation*}
Moreover, the arrangement $A$ can be uniquely reconstructed from the pair $c,c'$, the $c_w$-completions $A_w$, and the $(c_w,c_{w'})$--$(c_{w}',c_{w'}')$-completions $(A_w,A_{w'})$. Let $L$ be the number of possible configurations on any fixed arc $e$. Then obviously for each $c \in \Ccal$ there are at most $L$ configurations on $e$ which are $\Gamma$-equivalent to $c$, so at most $L$ of the configurations $c_w$ and $c_w'$ are equivalent to $c$. 
We conclude that for fixed $k_1,k_2 \geq 0$ and $k_3 \geq 1$, the subset of all $A \in \Acal(c,c',n)$ such that $\abs{\Wcal_i}=k_i$ for $i=1,2,3$ gives a contribution to the sum \eqref{eq:decompu} of at most
\begin{equation}\label{eq:decompu2}
\left(L\sum_{c_1 \in \Ucal} F_{c_1}(z) \right)^{k_1+k_2} \left( L^2 \sum_{( c_3, c_4), (c_3', c_4') \in  \Ical_{\pair,p}} (\Jac_{\pair,\disj}(z)^n)_{( c_3, c_4), (c_3', c_4')} \right)^{k_3}
\end{equation}

Lemma~\ref{lem:u-finite} tells us that $F_{c_1}(R)$ is finite for every $c_1 \in \Ucal$. In particular, by monotonicity of the entries of $F_{c_1}(z)$, there is some constant $L'$ which does not depend on $n$ such that
\[L \sum_{c_1 \in \Ucal} F_{c_1}(z) \leq L'\] 
for all $z \in [0,R]$. 

Recall that for any matrix norm $\norm{\cdot}$ and any matrix $A$ with spectral radius $\lambda$ there are finite constants $M,m$ such that $\norm{A^n} \leq M n^m \lambda^n$ holds for every $n \in \N$. An easy way to verify this is by computing powers of Jordan blocks in the Jordan normal form of $A$.
Applying this to $\Jac_{\pair,\disj}(z)$, we obtain
\begin{equation}\label{eq:decompu3}
L^2 \sum_{( c_3, c_4), (c_3', c_4') \in  \Ical_{\pair,p}} (\Jac_{\pair,\disj}(z)^n)_{( c_3, c_4), (c_3', c_4')} \leq M n^m \lambda_{\pair,\disj}^n(z)
\end{equation}

Utilizing these estimates and that $0 \leq k_1+k_2 \leq 2K$ and $1 \leq k_3 \leq K$, from \eqref{eq:decompu} we obtain  
\[
 (\Jac_\Kcal(z)^n)_{c,c'} \leq  \sum_{k_1=0}^{2K} \sum_{k_3=1}^K  (L')^{k_1} \left(M n^m \lambda_{\pair,\disj}^n(z)\right)^{k_3}
\]

Note that $\lambda_{\pair,\disj}(z) \leq 1$ by Lemma~\ref{lem:pair-p} combined with Lemma~\ref{lem:spectral}. Thus 
\[
 (\Jac_\Kcal(z)^n)_{c,c'} \leq M' n^{m'} \lambda_{\pair,\disj}^n(z) 
\]
holds for some new constants $M'$ and $m'$ independent of $n$.

Taking $n$-th roots and sending $n$ to infinity, an application of Gelfand's formula completes the proof:
\[
\lambda_\Kcal(z)= \lim_{n \to \infty} \norm{\Jac_\Kcal(z)^n}_{\infty}^{1/n} \leq \lambda_{\pair,\disj}(z)
\]
\end{proof}

\begin{lem}\label{lem:Itrans-pair}
For every $z\in (0,R]$ we have $\lambda_{\Ical_t}(z) \leq \lambda_{\pair,\disj}(z)$.
\end{lem}
\begin{proof}
    We follow a similar strategy as in the proof of Lemma~\ref{lem:U-pair}. Let $\Kcal \subseteq \Ucal$ be a strong component of $D$ and let $c,c'$ be two configurations in $\Kcal$. Then
\begin{equation}\label{eq:decomp-itrans}
    (\Jac_\Kcal(z)^n)_{c,c'} = \sum_{A \in \Acal(c,c',n)} z^{\norm{A}}.
\end{equation}
As in the previous proof we translate every $A \in \Acal(c,c',n)$ to the SAW $\varphi(A)$ represented by $A$. Then we decompose $\varphi(A)$ at every vertex contained in the adhesion sets $\Vcal(\iota) \cup \Vcal(\tau)$, where $\iota$ and $\tau$ are source and target arc of $A$. We end up with sub-walks $w_1, \dots, w_k$ of $\varphi(A)$ and categorize them into classes \ref{itm:ll-walk} to \ref{itm:rl-walk}. However, because $A$ is a $c$--$c'$-completion for two I-configurations $c$ and $c'$, the walk $\varphi(A)$ starts at $\Vcal(\iota)$ and ends at $\Vcal(\tau)$. This means that class~\ref{itm:lr-walk} contains one more walk than class~\ref{itm:rl-walk}. We denote by $w_0$ the walk $w_i$ with maximal index $i$ in class \ref{itm:lr-walk}. Excluding $w_0$, as before we can build pairs, each consisting of a walk of class \ref{itm:lr-walk} and a walk of class \ref{itm:rl-walk}. Denote by $\Wcal_1$, $\Wcal_2$ and $\Wcal_3$ the sets of walks $w_i$ of class \ref{itm:ll-walk}, class \ref{itm:rr-walk} and the set of pairs $(w_i,w_j)$ of classes \ref{itm:lr-walk} and \ref{itm:rl-walk}, respectively.

As before, for every $w \in \Wcal_1 \cup \Wcal_2$ there is a $c_w$-completion $A_w$ representing $w$ provided by Theorem \ref{thm:shape-to-arrangement} such that $c_w$ is a U-configuration. Furthermore, for every $(w,w') \in \Wcal_3$ there is a $(c_w,c_{w'})$--$(c_{w}',c_{w'}')$-completion $(A_w,A_{w'})$ representing $(w,w')$ such that $c_w,c_{w'},c_{w}'$ and $c_{w'}'$ are simple configurations. Finally, there is a $c_0$--$c_0'$-completion $A_{w_0}$ representing $w_0$ such that $c_0$ and $c_0'$ are simple configurations.

The main difference to the previous proof is that the size $k_3$ of the set $\Wcal_3$ need not necessarily be larger than 0. We denote by $\Acal_{=0}(c,c',n)$ the subset of all $A \in \Acal(c,c',n)$ such that $k_3=0$ and by $\Acal_{>0}(c,c',n)$ its complement $\Acal(c,c',n) \setminus \Acal_{=0}(c,c',n)$. Let us first deal with arrangements in $\Acal_{>0}(c,c',n)$, which works similarly to the previous proof. The only difference is an additional factor
\[
    L \sum_{c_0,c_0' \in  \Ical_p} (\Jac_{\Ical_p}(z)^n)_{c_0,c_0'}
\]
in the analogue of equation \eqref{eq:decompu2} coming from the walk $w_0$. As in \eqref{eq:decompu3}, we can bound this factor from above by $M_0 n^{m_0} \lambda_{I_p}^n(z)$, for some constants $M_0$ and $m_0$ independent from $n$.
By Lemma \ref{lem:spectral} we have that $\lambda_{I_p}(z) \leq 1$, thus we again end up with 
\begin{equation}\label{eq:decomtr-1}
\sum_{A \in \Acal_{>0}(c,c',n)} z^{\norm{A}} \leq M' n^{m'} \lambda_{\pair,\disj}^n(z) 
\end{equation}
for some constants $M', m'$ independent of $n$.

We are left to deal with arrangements in $\Acal_{=0}(c,c',n)$. Here we need to treat the $c_w$-completions  $A_w$ for $w \in \Wcal_1 \cup \Wcal_2$ more carefully. Let $S_w$ be the support of $A_w$ and denote by $W$ the open path in $T$ connecting $\iota$ and $\tau$. Then $W \cap S_w$ is an open path of $W$ containing $\iota$ if $w \in \Wcal_1$ and $\tau$ if $w \in \Wcal_2$. Denote by $j_w \in [n]$ the length of this open path. Then $A_w$ consists of a $c_w$--$c'_w$-completion of length $j_w-1$ for some non-boring U-configuration $c'_w$ and a $c'_w$-completion. We claim that there is a constant $N$ independent of $n$ such that 
\begin{equation}\label{eq:j-sum}
\sum_{w \in \Wcal_1 \cup \Wcal_2} j_w \geq n-N
\end{equation}
holds for every $A \in \Acal_{=0}(c,c',n)$. Indeed, by the choice of $A$ for every interior arc $e$ of $W$ the configuration $\rho(C(e))$ lies in the component $\Kcal$ and thus must be transient by Lemma~\ref{lem:D-components}. In particular $\Vcal(e)$ must be visited by some $w \in \Wcal_1 \cap \Wcal_2$, otherwise $C(e)=C_{w_0}(e)$ holds and thus $\rho(C(e))$ lies in the same component as the simple configurations 
$c_0$ and $c_0'$ and thus is persistent. By Corollary~\ref{cor:finitesubtree} there is a constant $N'$ such that $\Vcal(e) \cap \Vcal(f)=\emptyset$ whenever $d_T(e,f) \geq N'$. In particular, $w \in \Wcal_1$ cannot visit $\Vcal(e)$ for any arc $e$ such that $d_T(\iota, e)>j_w+N'$. This implies in particular that \eqref{eq:j-sum} holds for $N=2N'$.

For fixed $k \geq 0$ and integers $j_1, \dots j_k \in [n]$, the subset of all $A \in \Acal_{=0}(c,c',n)$ such that $\Wcal_i$ contains $k$ elements $w_1, \dots, w_k$ with $j_{w_i}=j_i$ gives a contribution to the sum \eqref{eq:decomp-itrans} of at most 
\[
\left( L \sum_{c_0,c_0' \in  \Ical_p} (\Jac_{\Ical_p}(z)^n)_{c_0,c_0'}\right) \prod_{i=1}^k \left( L \sum_{c_1,c_1' \in \Ucal} (\Jac_{\Ucal}(z)^{j_i-1})_{c_1,c_1'} F_{c_1'}(z) \right) 
\]
The first factor is bounded by $M_0 n^{m_0}$ as in the previous case. For the other factors, we use again that by Lemma \ref{lem:u-finite} there is a constant $L'$ such that $F_{c}(z) \leq L'$ holds for every $c \in \Ucal$ and every $z \in [0,R]$. Therefore, we bound each factor as in \eqref{eq:decompu3} to obtain the upper bound
\[
\prod_{i=1}^k \left( L \sum_{c_1,c_1' \in \Ucal} (\Jac_{\Ucal}(z)^{j_i-1})_{c_1,c_1'} F_{c_1'}(z) \right) \leq (M_1)^k \prod_{i=1}^k (j_i-1)^{m_1} \lambda_{\Ucal}(z)^{j_i-1} \leq M_2^k n^{m_2} \lambda_{\Ucal}(z)^n,
\]
where $M_1, M_2, m_1$ and $m_2$ are constants independent of $n$ and for the last inequality we use \eqref{eq:j-sum} and that $\lambda_U(z) \leq 1$ by Lemma~\ref{lem:U-pair}.

With these estimates we obtain
\begin{equation}\label{eq:decomtr-2}
\sum_{A \in \Acal_{=0}(c,c',n)} z^{\norm{A}} \leq M_0 n^{m_0} \sum_{k=0}^{2K} \; \sum_{j_1, \dots,j_k=1}^n M_2^k n^{m_2} \lambda_{\Ucal}^n(z) \leq M_3 n ^{m_3} \lambda_{\Ucal}^n(z)
\end{equation}
for some new constants $M_3$ and $m_3$.

Combining equations \eqref{eq:decomtr-1} and \eqref{eq:decomtr-2} and using Lemma \ref{lem:U-pair}, we end up with
\[
(\Jac_\Kcal(z)^n)_{c,c'} \leq M'' n^{m''} \lambda_{\pair,\disj}^n(z) 
\]
for some new constants $M''$ and $m''$. Taking the $n$-th root and sending $n$ to infinity completes the proof.
\end{proof}

\begin{cor}\label{cor:u analytic}
If $\Gamma$ does not fix an end of $T$, then $R_{\Ucal}>R$.
\end{cor}
\begin{proof}
Since $P_c(z,\ybf)$ is analytic at $(z,\ybf)=(R,(F_{c'}(R))_{c'\in \mathcal{C}})$ for all $c\in \Ucal$ by Proposition~\ref{pro:P-analytic}, the matrix $\Jac_{\Ucal}(R)$ is well-defined and has finite entries (for the strong components of $\Ucal$ this is already implied by Lemma~\ref{lem:spectral}).
It follows from Lemmas~\ref{lem:spectral}, \ref{lem:pair-p}, and \ref{lem:U-pair} that $\lambda_{\Ucal}(R)<1$, hence $I-\Jac_{\Ucal}(R)$ is invertible. Since $P_c(z,\ybf)$ is analytic at $(z,\ybf)=(R,(F_{c'}(R))_{c'\in \mathcal{C}})$ for all $c\in \Ucal$, it follows from the analytic Implicit Function Theorem that the functions $F_c(z)$ for $c\in \Ucal$ are analytic in a neighbourhood of $R$.
\end{proof}

\section{Proof of main results}\label{sec: proofs}

In this section we will prove the main results of this paper building on the results developed in the previous sections.

Our overall strategy is to exploit the close connection between $c$-completions and SAWs to obtain an expression for the SAW-generating function $F_{\SAW}(z)$ in terms of the functions $F_c(z)$. To do this, recall that Theorem~\ref{thm:saw-to-arrangement} allows us to bijectively translate every SAW $w$ on $G$ into a complete arrangement $A=(P,C)$ on the support of $w$ representing $w$ such that $\norm{A}$ coincides with the length of $w$. By Lemma~\ref{lem:source-contains-first-edge} the shape $P(s_0)$ of the source $s_0$ of $A$ starts with the first arc of $p$. In particular, even when fixing the starting vertex $o$ of our SAWs, $s_0$ may still vary depending on the specific choice of $w$, the only restriction is that $\Vcal(s_0)$ has to contain $o$. To get rid of this small inconvenience, let us first manipulate our tree decomposition. 

Fix $o \in V(G)$ and let $S$ be the subtree of $T$ induced by all vertices $t \in V(T)$ for which $\Vcal(t)$ contains a vertex at distance at most $2$ from $o$. Note that $S$ is finite by Proposition \ref{pro:finitesubtree} and that $\Gcal(S)$ contains all edges of $G$ incident to $o$ or one of its neighbours.
Let $\Tcal'=(T',\Vcal') = \Tcal/S$ be the tree decomposition obtained from $\Tcal$ by contracting all interior edges of $S$ and let $s_0$ be the vertex of $T'$ representing $V(S)$. Clearly the tree decomposition $\Tcal'$ is not $\Gamma$-invariant, as the new part $\Vcal'(s_0)$ plays a special role and cannot be mapped to any other parts. However, every arc $e \in E(T')$ pointing towards $s_0$ was also present in $T$ and the open cone $K_e$ in $T'$ coincides with the respective cone in $T$, so $c$-completions and their generating functions remain the same. By Theorem~\ref{thm:saw-to-arrangement} any SAW $w$ starting at $o$ is represented by a unique complete arrangement $A$ on the support of $w$ in $T'$. Clearly $P(s_0)$ starts at $o$. Moreover, all edges incident to $o$ are contained in the part $\Vcal'(s_0)$ by construction of $\Tcal'$. Thus by Lemma~\ref{lem:source-contains-first-edge} the source of $A$ has to be $s_0$. We obtain

\[
F_{\SAW}(z)= \sum_{w \text{ SAW starting at } o} z^{\abs{w}}= \sum_{\substack{\text{$A$ compl. arr.} \\ \text{with source $s_0$ s.t.} \\ \text{$P(s_0)$ starts at $o$}}} z^{\norm{A}}.
\]

By decomposing the tree $T'$ into an open star centered at $s_0$ and possibly infinitely many cones $K_{\bar e}$ for $e \in \Eout(s_0)$ and the respective arrangement on the open subtree $S$ into an arrangement $A=(P,C)$ on $\sta(s_0)$ and $C(e)$-completions for all non-boring configurations on arcs $e$ of $\sta(s_0)$, we conclude
\[
F_{\SAW}(z)= \sum_{\substack{\text{$A$ arr. on $\sta(s_0)$:}\\ \text{$P(s_0)$ starts at $o$}}} z^{\norm{A}} \prod_{e \in \nb(A,s_0)} F_{C(e)}(z). 
\]
We point out that all arrangements $A$ included in the sum satisfy $X(e) = \bar e$ for every $e \in \Eout(s_0)$ because by construction $o$ is not included in any adhesion set.

To analyse $F_{\SAW}(z)$ we need to first isolate the terms that determine its radius of convergence. To this end we write
\begin{equation} \label{eq:expPSAW}
F_{\SAW}(z)= F^1_{\SAW}(z)+F^2_{\SAW}(z),
\end{equation}
where   
\[
F^1_{\SAW}(z)= \sum_{\substack{\text{$A$ arr. on $\sta(s_0)$:}\\ \text{$P(s_0)$ starts at $o$ and} \\ \forall e \in \Eout(s_0)\colon Y(e)=\bar{e}}} z^{\norm{A}} \prod_{e \in \nb(A,s_0)} F_{C(e)}(z)
\]
and
\[
F^2_{\SAW}(z)=\sum_{\substack{\text{$A$ arr. on $\sta(s_0)$:}\\ \text{$P(s_0)$ starts at $o$} \\  \exists \, e_0\in E(s_0)\colon Y(e_0)=e_0}} z^{\norm{A}} \prod_{e \in \nb(A,s_0)} F_{C(e)}(z).
\]
Proposition \ref{pro:Fz-analytic} below implies that $F^1_{\SAW}(z)$ is analytic at $z=R$. Hence only $F^2_{\SAW}(z)$ is relevant. In order to study the latter we further write
\begin{equation} \label{eq:F2ps}
F^2_{\SAW}(z)=\Fbf^{\init}(z)\sum_{k=0}^\infty \Jac_{\Ical}(z)^k \Fbf^{\term}(z)
\end{equation}
where $\Fbf^{\init}(z)=(F^{\init}_c(z))_{c\in \Ical}$ is the vector with entries
\[
F^{\init}_c(z)= \sum_{\substack{\text{$A$ arr. on $\sta(s_0)$:}\\ \text{$P(s_0)$ starts at $o$} \\ \exists \, e_0\in E(s_0) \colon C(e_0)=c}} z^{\norm{A}} \prod_{e \in \nb(A,s_0)\setminus\{e_0\}} F_{C(e)}(z)   
\]
and $\Fbf^{\term}(z)=(F^{\term}_c(z))_{c\in \Ical}$ is the vector with entries  
\[
F^{\term}_c(z)=\sum_{\substack{\text{$A$ arr. on $\sta(x^-)$:}\\ \text{$C(x)=c$ and} \\ \forall e \in \Eout(x^-)\colon Y(e)=\bar{e}}} z^{\norm{A}} \prod_{e \in \nb(A,x^-)\setminus \{x\}} F_{C(e)}(z).
\]
By Lemma \ref{lem:spectral}, we know that
\begin{equation} \label{eq:F2}
F^2_{\SAW}(z)=\Fbf^{\init}(z)(I-\Jac_{\Ical}(z))^{-1} \Fbf^{\term}(z)
\end{equation}
for $z<R$.

In order to prove Theorem~\ref{thm:c_n asym}, by \cite[Theorem 4]{P19} it suffices to show that $\mu_p<\mu_w$, where $\mu_p$ and $\mu_w$ are the connective constants for SAPs and SAWs, respectively.
To stay within the framework of our definitions, it will be convenient to work with self-avoiding returns (SARs) instead of SAPs. By construction of $\Tcal'$ all edges incident to neighbours of $o$ are contained in the part $\Vcal'(s_0)$. Thus by Lemma~\ref{lem:source-contains-first-edge} every SAR $w$ starting at $o$ is represented by a unique complete arrangement $A$ on the support of $w$ in $T'$ with source and target $s_0$ such that $P(s_0)$ starts at $o$ and ends at a neighbour of $o$. Let 
\[F_{\SAR}(z)= \sum_{w \text{ SAR starting at } o} z^{\abs{w}}.\]
As above, we obtain 
\begin{equation} \label{eq:defPSAP}
F_{\SAR}(z)= \sum_{\substack{\text{$A$ arr. on $\sta(s_0)$:}\\ \text{$P(s_0)$ starts at $o$ and ends in $N(o)$}}} z^{\norm{A}} \prod_{e \in \nb(A,s_0)} F_{C(e)}(z). 
\end{equation}
Observe that the condition $X(e)=Y(e)=\bar{e}$ for every $e \in \Eout(s_0)$ implies that each configuration $C(e)$ is a U-configuration.

We will now prove the following result. 

\begin{pro}\label{pro:Fz-analytic}
Each of the functions $F^1_{\SAW}(z)$, $F_{\SAR}(z)$, $F^{\init}(z)$ and $F^{\term}(z)$ is analytic at $z=R$. Moreover, the function $F_{\SAW}(z)$ is analytic in the interval $(0,R)$.
\end{pro}
\begin{proof}
Let $F(z)$ be any of the functions $F^1_{\SAW}(z)$, $F_{\SAR}(z)$, $F_c^{\init}(z)$ and $F_c^{\term}(z)$. Then there is some $s \in V(T)$ such that each of the walks counted by $F(z)$ starts at some given vertex $v_0$ in $V(s)$ and ends in $V(s)$. Let $w$ be one of these walks. We decompose $w$ as follows.

Let $e_1$ and $e_2$ be the first and last arc of $w$ not in $\Ecal(s)$, respectively, and let $v_1=e_1^-$ and $v_2 = e_2^+$. If no such arcs exist, we set $v_1 = v_2 = w^+$. Then $w v_1$ and $v_2 w$ are SAWs on $\Gcal(s)$. If $v_1 \neq v_2$, let $f_1,f_2 \in \Eout(s)$ be such that $e_i$ is an edge of $\Gcal(K_{\bar f_i})$. Then $v_1 w v_2$ is a SAW on $G$ starting in $\Vcal(f_1)$ and ending in $\Vcal(f_2)$. If $f_1 = f_2$, then $v_1 w v_2$ can be represented by a pair consisting of a $c$-completion on the cone $K_{f_1}$ and a $c'$-completion on the cone  $K_{\bar f_1}$. Note that $c$ and $c'$ are U-configurations which only differ in their entry and exit direction.
Assume now that $f_1 \neq f_2$. Then $v_1 w v_2$ can be represented by a triple consisting of a $c_1$-completion on the cone $K_{\bar f_1}$, a $c_2$-completion on $K_{\bar f_2}$, and a $c_1'$--$c_2'$-completion of length 1 on $K_{f_1} \cap K_{f_2}$ which contributes to $\Jac_{\Ical}$. Here $c_1$ and $c_2$ are U-configurations, $c_1'$ and $c_2'$ are I-configurations, $c_1$ and $c_1'$ only differ in their entry direction while $c_2$ and $c_2'$ only differ in their exit direction.

With this decomposition in hand, the analyticity of $F(z)$ follows from Proposition~\ref{pro:P-analytic}, Proposition~\ref{pro:J-analytic} and Lemma~\ref{lem:onepart}.

Too see that $F_{\SAW}(z)$ is analytic in the interval $(0,R)$, recall that by \eqref{eq:expPSAW} and \eqref{eq:F2} we have that
\[
F_{\SAW}(z)=F^1_{\SAW}(z)+\Fbf^{\init}(z)(I-\Jac_{\Ical}(z))^{-1} \Fbf^{\term}(z).
\]
is finite for $z \in (0,R)$.
\end{proof}

With the above result in hands we can show that $R$ is determined by $\Ical_p$.

\begin{pro}\label{pro:Ip R}
If $\Gamma$ does not fix an end of $T$, then we have $R_{\Ical_p}=R$.    
\end{pro}
\begin{proof}
First, we claim that $\lambda_{\Ical_p}(R)=1$. Assume for a contradiction that $\lambda_{\Ical_p}(R)<1$. Then $\lambda_{\Ical}(R)<1$ by Lemmas \ref{lem:spectral}, \ref{lem:pair-p}, and \ref{lem:Itrans-pair}. Recall that $\Jac_{\Ical}(z)$ is analytic at $z=R$ by Proposition~\ref{pro:J-analytic}. Then $\lambda_{\Ical}(z)<1$ for every $z$ in a neighbourhood of $R$. Now, observe that 
\begin{equation}\label{eq: i-ji}
 (F_c(z))_{c\in\Ical}=\sum_{n=0}^\infty\Jac_{\Ical}(z)^n \Fbf^{\term}(z)
\end{equation}
provided that $\Jac_{\Ical}(z)$ and $\Fbf^{\term}(z)$ are analytic.
If $\lambda_{\Ical}(z)<1$, this implies that 
$(F_c(z))_{c\in\Ical}= (I-\Jac_{\Ical}(z))^{-1} \Fbf^{\term}(z)$.

Since this holds at $z=R$, it follows that $(F_c(z))_{c\in\Ical}$ is analytic at $z=R$, hence $R_{\Ical}>R$. The latter together with Corollary \ref{cor:u analytic} implies that $R>R$, which is absurd. This proves the claim.

Since $\lambda_{\Ical_p}(R)=1$, equation \eqref{eq: i-ji} together with the fact that $\Ical_p$ is a strong component implies that $F_c(R)$ is infinite for every $c\in \Ical_p$. The desired assertion follows.
%
\end{proof}

We are now ready to prove the main results of this paper. We first show that SAW is ballistic. In order to prove Theorem~\ref{thm:ballistic}, we will show the strict inequality $\mu_p<\mu_w$. The fact that this implies the ballisticity of SAW was proved in \cite[Theorem 4]{P19}. This result was proved for transitive graphs but as explained in \cite[Remark 4.1]{P19}, it can be extended to graphs which are not necessarily transitive but satisfy \[\limsup_{n\to\infty}\left( \sup_{x\in V} p_n(x) \right)^{1/n}<\liminf_{n\to\infty} \left(\inf_{x\in V} c_n(x)\right)^{1/n}.\] In particular, the result of \cite[Theorem 4]{P19} holds for quasi-transitive graphs.

\begin{proof}[Proof of Theorem~\ref{thm:ballistic}]
By Proposition~\ref{pro:Fz-analytic}, the generating function $F_{\SAR}(z)$ is analytic at $z=R$. Since there are fewer self-avoiding polygons of length $n$ containing $o$ than self-avoiding returns starting at $o$, it follows that $\mu_p < 1/R$.

Let $p$ be a shortest walk from $o$ to some adhesion set $\Vcal(e)$, where $e^+ = s_0$. 
Let $c$ be the simple configuration $(q,x,y)$ where $q$ is the walk consisting only of the terminal vertex of $p$, $x=e$, and $y=\bar e$. Then  $z^{\abs p} F_c(z) \leq F_{\SAW}(z)$ because $p$ together with the walk corresponding to any $c$-completion is a self-avoiding walk. This together with Proposition~\ref{pro:Ip R} implies that $\mu_w \geq 1/R_{\Ical_p} = 1/R$. The desired result follows.
\end{proof}

\begin{rmk}\label{rem: mu R}
As mentioned in the proof of Theorem~\ref{thm:ballistic}, we have $\mu_w\geq 1/R$. The reverse inequality $\mu_w\leq 1/R$ follows from the fact that $F_{\SAW}(z)$ is finite for every $z<R$ by Proposition~\ref{pro:Fz-analytic}.
This proves that $\mu_w=1/R$.
\end{rmk}

We will now prove that $c_n$ grows asymptotically like $\mu_w^n$, that is, the subexponential factor is $O(1)$.

\begin{proof}[Proof of Theorem~\ref{thm:c_n asym}]
We will show that all singularities of $F_{\SAW}(z)$ on the circle $|z|=R$ are simple poles and they are located at the complex $k$-th roots of $R^k$ for some $k$, from which the desired result will follow. To this end, let $r>R$ be such that each of $F^1_{\SAW}(z), F^{\init}(z), F^{\term}(z)$ and $\Jac_{\Ical}(z)$ is analytic for every $z\in \mathbb{C}$ such that $|z|<r$. Such an $r$ exists by Propositions~\ref{pro:J-analytic} and \ref{pro:Fz-analytic}. Then the points of singularity of $F_{\SAW}(z)$ for $|z|<r$ are the points at which $\det(I-\Jac_{\Ical}(z))=0$ by \eqref{eq:expPSAW} and \eqref{eq:F2}. 

Since $\det(I-\Jac_{\Ical}(z))$ is a non-constant analytic function, its zeros are isolated, hence we can assume without loss of generality that $\det(I-\Jac_{\Ical}(z))\neq 0$ for every $R<|z|<r$. For every $|z|<R$ we have $\lambda_{\Ical}(z)<1$ by Lemma~\ref{lem:spectral}, thus the only singularities of $F_{\SAW}(z)$ for $|z|<r$ are at the points of the circle $|z|=R$ at which $1$ is an eigenvalue of $\Jac_{\Ical}(z)$. Let us denote this set by $S$ and consider some $w\in S$. 

We claim that $1$ is a simple eigenvalue of $\Jac_{\Ical}(w)$. Indeed, 
note that $1$ is a simple eigenvalue of the irreducible matrix $\Jac_{\Ical_p}(R)$ by the Perron-Frobenius theorem, and that $\lambda_{\Ical_t}(R)<1$ by Lemmas~\ref{lem:Itrans-pair} and \ref{lem:pair-p}. Since $|\Jac_{c,c'}(w)|\leq \Jac_{c,c'}(R)$ for every $c,c'\in \Ical$ by the triangle inequality we have $\lambda_{\Ical_p}(w)=1$ and $\lambda_{\Ical_t}(w)<1$. It thus suffices to show that $1$ is a simple eigenvalue of $\Jac_{\Ical_p}(w)$. 

Note that $\Jac_{\Ical_p}(w)$ is not a real matrix in general, and we cannot apply Perron-Frobenius. Instead, we argue as follows. The matrix $\Jac_{\Ical_p}(R)$ is irreducible, $|\Jac_{c,c'}(w)|\leq \Jac_{c,c'}(R)$ for $c,c' \in \Ical_p$, and  $\Jac_{\Ical_p}(w)$ and $\Jac_{\Ical_p}(R)$ have the same spectral radius. Because $1$ is the eigenvalue of $\Jac_{\Ical_p}(w)$ of maximal modulus, it follows from \cite[Chapter I, Proposition 6.4]{Banach} that $\Jac_{\Ical_p}(w)$ and $\Jac_{\Ical_p}(R)$ are similar matrices and thus have the same spectrum. Hence $1$ is a simple eigenvalue of $\Jac_{\Ical_p}(w)$. Since $\lambda_{\Ical_t}(w)<1$ we deduce that $1$ is a simple eigenvalue of $\Jac_{\Ical}(w)$ thus proving our claim.

Our aim is to deduce that $(I-\Jac_{\Ical}(z))^{-1}$ has a simple pole at every point in $S$, which implies that the same holds for $F_{\SAW}(z)$. 

Applying \cite[Lemma 9]{AlJa90} we obtain for any $w\in S$ that 
\begin{equation}\label{eq:pole}
\lim_{z\to w} (z-w)(I-\Jac_{\Ical}(z))^{-1}=\left(\eta^T(w) \Jac_{\Ical}'(w) \xi(w) \right)^{-1}\xi(w) \eta^T(w),
\end{equation}
provided that $\eta^T(w) \Jac_{\Ical}'(w) \xi(w)\neq 0$, where $\eta(w)$ and $\xi(w)$ are left and right eigenvectors of $\Jac(w)$ corresponding to $1$ normalised so that $\eta^T(w)\xi(w)=1$. 

To verify the condition note that since $\lambda_{\Ical_t}(w)<1$, the only non-zero entries of $\eta(w)$ and $\xi(w)$ correspond to elements of $\Ical_p$, hence $\eta^T(w)\Jac_{\Ical}'(w) \xi(w)=\eta^T_p(w)\Jac_{\Ical_p}'(w) \xi_p(w)$, where $\eta_p(w)$ and $\xi_p(w)$ are the restrictions to $\Ical_p$. In the particular case $w=R$, we have $\eta^T_p(R)\Jac_{\Ical_p}'(R) \xi_p(R)\neq 0$ because by the Perron-Frobenius theorem the entries of $\eta_p(R)$ and $\xi_p(R)$ are strictly positive and at least one entry of $\Jac_{\Ical_p}(z)$ is strictly increasing in $z$. Thus it suffices to show that $\eta^T_p(w) \Jac_{\Ical_p}'(w) \xi_p(w)=\eta^T_p(R) \Jac_{\Ical_p}'(R) \xi_p(R)$. 

To this end, let $D=D(w)$ be an invertible matrix such that $D^{-1}\Jac_{\Ical_p}(w) D=\Jac_{\Ical_p}(R)$. Then $D^{-1}\xi_p(w)=\alpha \xi_p(R)$ and $\eta^T_p(w)D=\alpha^{-1}\eta^T_p(R)$ for some $\alpha \neq 0$ due to the fact that each of $D^{-1}\xi_p(w),\xi_p(R),\eta^T_p(w)D,\eta^T_p(R)$ is an eigenvector of an one-dimensional eigenspace of $\Jac_{\Ical_p}(R)$. The scalar factor $\alpha$ is due to the normalisation. Applying again \cite[Chapter I, Proposition 6.4]{Banach} we obtain that $|\Jac_{c,c'}(w)|=\Jac_{c,c'}(R)$ for all $c,c'\in \Ical_p$, in other words, we have equality in the triangle inequality. This implies that there exist integers $k,\ell$ such that $t_n$ is a non-zero Taylor coefficient of $\Jac_{c,c'}(z)$ only if $n\in k\mathbb{Z}+\ell$ and $w/R$ is a $k$th root of unity. 

Thus $D^{-1}\Jac_{\Ical_p}(\beta w) D=\Jac_{\Ical_p}(\beta R)$ for every $\beta \in [0,1]$, hence $D^{-1}\Jac_{\Ical_p}'(w) D=\Jac_{\Ical_p}'(R)$. This in turn implies that $\eta^T_p(w) \Jac_{\Ical_p}'(w) \xi_p(w)=\eta^T_p(R) \Jac_{\Ical_p}'(R) \xi_p(R)$, and we can conclude that \eqref{eq:pole} holds.

Therefore there exist $b_1,b_2,\ldots, b_k \in \mathbb{C}$ such that the function
\[
H(z)=F_{\SAW}(z)-\sum_{i=1}^k \frac{b_i}{w_i-z}
\]
is analytic in the open disk $|z|<r$, where $w_1,w_2,\ldots,w_k$ are the elements of $S$. This implies that the $n$th Taylor coefficient $h_n$ of $H(z)$ around $0$ satisfies $h_n=O((r-\varepsilon)^{-n})$ for $\varepsilon=(r-R)/2>0$. Taylor expanding we see that
\[
c_n-\sum_{i=1}^k \frac{b_i}{w^{n+1}_i}=h_n=O((r-\varepsilon)^{-n}).
\]
Since each $w_i/R$ is a $k$th root of unity, $\mu_w=1/R$ by Remark~\ref{rem: mu R}, and $c_n\geq 0$, it follows that \eqref{eq:c_n asym} holds for some $a_1,a_2,\ldots,a_k\geq 0$. Since on a quasi-transitive graph $c_n^{1/n}$ converges to $\mu_w$ \cite{MR0091568}, it follows that each $a_i$ is strictly positive. 

The second part of the theorem follows now immediately, recalling that on a transitive graph we have $c_n\geq \mu_w^n$ for every $n\geq 1$.
\end{proof}



\bibliographystyle{abbrv}
\bibliography{latex}

\end{document}